\documentclass[11pt,reqno]{amsart}

\pdfoutput=1

\usepackage{LSS1_Macros}

\begin{document}

\title[On the Gross-Pitaevskii evolution linearized around the degree-one vortex]{On the Gross-Pitaevskii evolution \\ linearized around the degree-one vortex}

\begin{abstract}
    We study the evolution of the Gross-Pitaevskii equation linearized around the Ginzburg-Landau vortex of degree one under equivariant symmetry. 
    Among the main results of this work, we determine the spectrum of the linearized operator, uncover a remarkable $L^2$-norm growth phenomenon related to a zero-energy resonance, and provide a complete construction of the distorted Fourier transform at small energies. The latter hinges upon a meticulous analysis of the behavior of the resolvent in the upper and lower half-planes in a small disk around zero-energy.
\end{abstract}

\author[J. L\"uhrmann]{Jonas L\"uhrmann}
\address{Department of Mathematics \\ Texas A\&M University \\ College Station, TX 77843, USA}
\email{luhrmann@tamu.edu}

\author[W. Schlag]{Wilhelm Schlag}
\address{Department of Mathematics \\ Yale University \\ New Haven, CT 06511, USA}
\email{wilhelm.schlag@yale.edu}

\author[S. Shahshahani]{Sohrab Shahshahani}
\address{Department of Mathematics \\ University of Massachusetts \\ Amherst, MA 01003, USA}
\email{sshahshahani@umass.edu}

\thanks{The first author was partially supported by NSF CAREER grant DMS-2235233. The second author was partially supported by NSF grant DMS-2350356. The third author was partially supported by the Simons Foundation grant 639284. The second author thanks Ovidiu Costin for discussions on the Stokes phenomenon.}

\maketitle

\tableofcontents

\section{Introduction} \label{sec:intro}

The complex Ginzburg-Landau equation in the plane
\begin{equation} \label{equ:complex_GL}
 \Delta \Psi + (1 - |\Psi|^2) \Psi = 0, \quad x \in \bbR^2,
\end{equation} 
is the Euler-Lagrange equation associated with the energy functional
\begin{align*}
    \int_{\bbR^2} \biggl( \frac{1}{2} |\nabla\Psi|^2 + \frac{1}{4} \bigl(1-|\Psi|^2\bigr)^2 \biggr) \ud x.
\end{align*}
It is known that \eqref{equ:complex_GL} admits a family of smooth solutions, called Ginzburg-Landau vortices, of the form $V_n(x)=\rho_n(r)e^{in\theta}$, $n \in \bbZ \backslash\{0\}$, where $x = (r \cos(\theta), r \sin(\theta))$. 
The integer $n$ is referred to as the degree or the winding number of $V_n(x)$.
The profile $\rho_n(r)$ is the unique solution to the ordinary differential equation
\begin{equation}\label{eq:rhoeq1}
 \left\{ \begin{aligned}
  &\partial_r^2 \rho_n + \frac{1}{r} \partial_r \rho_n - \frac{n^2}{r^2} \rho_n + \bigl(1-\rho_n^2\bigr) \rho_n = 0, \quad 0 < r < \infty, \\
  &\rho_n(0) = 0, \quad \rho_n(\infty) = 1, \quad \rho_n'(r) > 0.
 \end{aligned} \right.
\end{equation}
It satisfies the asymptotics
\begin{equation}\label{eq:rhobounds1}
 \begin{aligned}
  \rho_n(r) \sim a r^n \Bigl( 1 - \frac{r^2}{4n+4} \Bigr) \quad \text{as} \quad r \to 0, \qquad 
  \rho_n(r) \sim 1 - \frac{n^2}{2r^2} \quad \text{as} \quad r \to \infty,
 \end{aligned}
\end{equation}
where $a > 0$ is some positive constant, cf. \cite{CEQ94,HerHer94}.
A peculiar feature of the vortices $V_n(x)$ is that they do not have finite energy, because their angular derivatives logarithmically fail to be $L^2$-integrable. 

The natural Schr\"odinger evolution equation associated with \eqref{equ:complex_GL} is the Gross-Pitaevskii equation 
\begin{equation} \label{eq:GP1}
    i\partial_t\Psi+\Delta\Psi+(1-|\Psi|^2)\Psi=0, \quad (t,x) \in \bbR \times \bbR^2.
\end{equation}
The vortices $V_n(x)$ are time-independent solutions to \eqref{eq:GP1} and form well-known examples of topological solitons. 
A fundamental question in the dynamics of the Gross-Pitaevskii equation concerns the stability of its vortex solutions. However, since these solutions do not possess finite energy, their stability analysis is particularly delicate.
By invariance under complex conjugation it suffices to consider positive degrees.
It appears that the degree-one vortex has some stability properties, while the vortices $V_n$, $n \geq 2$, are unstable, at least outside symmetry.
By considering a suitable renormalized energy, \cite{GPS21} recently obtained an orbital stability result for $V_1$.
As a first step towards the long-standing open asymptotic stability problem for the degree-one vortex, we study the Gross-Pitaevskii evolution linearized around $V_1$. 

\subsection{The Gross-Pitaevskii equation linearized around the degree-one vortex}

Decomposing a solution to the Gross-Pitaevskii equation \eqref{eq:GP1} as 
\begin{align} \label{equ:decomposition}
    \Psi(t,x) = \bigl( \rho_1(r) + \phi(t,r,\theta) \bigr) e^{i\theta},
\end{align}
we arrive at the linearized evolution equation 
\begin{equation} \label{equ:decomp_inserted2}
 \begin{aligned}
    i\pt \phi + \Delta \phi+\frac{2i}{r^2}\partial_\theta\phi - \frac{1}{r^2} \phi + \phi - 2 \rho_1^2 \phi - \rho_1^2 \bar{\phi} = 0.   
 \end{aligned}
\end{equation}
Expressing the perturbation $\phi(t,r,\theta)$ in terms of its real and imaginary parts
\begin{equation*}
 \begin{aligned}
  \phi(t,r,\theta) = \alpha(t,r,\theta) + i \beta(t,r,\theta),
 \end{aligned}
\end{equation*}
we can equivalently write \eqref{equ:decomp_inserted2} as
\begin{equation} \label{eq:bfLintro1}
    \bigl( \pt - \bfL \bigr) \begin{bmatrix} \alpha \\ \beta \end{bmatrix} = 0,
\end{equation}
where 
\begin{align*}
    \bfL := \begin{bmatrix} \: -\frac{2}{r^2} \partial_\theta & -\Delta + \frac{1}{r^2} + \rho_1^2 - 1 \\  - \bigl( - \Delta + \frac{1}{r^2} + 3 \rho_1^2 - 1 \bigr) & \: -\frac{2}{r^2} \partial_\theta \end{bmatrix}. 
\end{align*}
The symmetries of the Gross-Pitaevskii equation give rise to zero energy solutions $\bfL \boldsymbol{\varphi} = 0$ for the linearized operator.
Specifically, the invariance under spatial translations yields the zero energy solutions
\begin{equation*}
  \boldsymbol{\varphi}_1 := \begin{bmatrix} \cos(\theta) \rho_1'(r) \\ - \sin(\theta) \frac{1}{r} \rho_1(r) \end{bmatrix}, \quad \boldsymbol{\varphi}_2 := \begin{bmatrix} \sin(\theta) \rho_1'(r) \\ \cos(\theta) \frac{1}{r} \rho_1(r) \end{bmatrix},
\end{equation*}
while the invariance under phase shifts gives the zero energy solution
\begin{equation*}
  \boldsymbol{\varphi}_3 := \begin{bmatrix} 0 \\ \rho_1(r) \end{bmatrix}.
\end{equation*}
A quick computation shows that $\boldsymbol{\varphi}_j$, $j = 1, 2$, belong to $L^p(\bbR^2)$ for $2 < p \leq \infty$ and are thus ``p-wave resonances'', 
while $\boldsymbol{\varphi}_3$ belongs only to $L^\infty(\bbR^2)$, and is therefore an ``s-wave resonance'' in the jargon for Schr\"odinger operators in two space dimensions.

In this work we restrict to equivariant perturbations
$\phi(t,r) = \alpha(t,r) + i \beta(t,r)$
that do not depend on the angular variable. 
Conjugating the corresponding linearized operator to the half-line by the weight $r^{\frac12}$ gives rise to the operator
\begin{equation} \label{equ:intro_decomposition_calL}
 \calL := r^{\frac12} \cdot \begin{bmatrix} 0 & -\Delta_{\mathrm{rad}} + \frac{1}{r^2} + \rho_1^2 -1 \\ -\bigl(-\Delta_{\mathrm{rad}} + \frac{1}{r^2} + 3\rho_1^2 - 1 \bigr) & 0 \end{bmatrix} \cdot r^{-\frac12} = \begin{bmatrix} 0 & L_1\\ -L_2 & 0 \end{bmatrix},
\end{equation}
where
\begin{equation*}
\begin{aligned}
  L_1 := -\partial_r^2 + \frac{3}{4 r^2} + \rho_1^2 - 1, \quad L_2 := -\partial_r^2 + \frac{3}{4 r^2} + 3\rho_1^2 - 1.
 \end{aligned}
\end{equation*}
Note that the scalar Schr\"odinger operators $L_1$ and $L_2$ can be written as
\begin{align}\label{eq:L012def1}
    L_1 = L_0+\rho_1^2-1, \quad L_2 = L_0+3\rho_1^2-1, \quad \text{with} \quad L_0 := -\partial_r^2+\frac{3}{4r^2}.
\end{align}
It is standard that $L_0$ is limit point at both $r=0$ and $r=\infty$, and that it is essentially self-adjoint. 
We denote its domain by 
\begin{align}\label{eq:domaindef1}
    \calD := \bigl\{ f \in L^2_r(\bbR_+) \, \big| \, L_0 f\in L^2_r(\bbR_+) \bigr\} = r^{\frac12} \cdot H^2_{\mathrm{rad}}(\R^2).
\end{align}
Then $L_1$ and $L_2$ are also self-adjoint on~$\calD$ by the Kato-Rellich theorem.  
Observe that $\calL$ is not self-adjoint, but it is skew-adjoint relative to the standard symplectic form 
\begin{align} \label{eq:omegasymp1}
    \omega(\phi, \psi) = \langle \bfJ \phi,\psi\rangle, \quad \bfJ := \pmat{0&1\\-1&0}, 
\end{align}
where $\langle\cdot,\cdot\rangle$ denotes the real inner product on $L^2_r(\bbR_+) \times L^2_r(\bbR_+)$. 
In this work we analyze the evolution $e^{t\calL}$ on the Hilbert space $L^2_r(\bbR_+) \times L^2_r(\bbR_+)$.

\subsection{Main results} \label{subsec:mainresult}

Our main results include the determination of the spectrum of the linearized operator $\calL$, the uncovering of a remarkable $L^2$ norm growth phenomenon related to a zero-energy resonance, and the complete construction of the distorted Fourier transform associated with $\calL$ at small energies.
The following theorem summarizes our findings.
 
\begin{thm}\label{thm:summary}
The operator $\calL$ on $L^2_r(\bbR_+) \times L^2_r(\bbR_+)$ with domain $\calD \times \calD$ satisfies:
\begin{enumerate}[(i)]
    \item $\spec(i\calL) = \R$. 
    
    \item $\calL$ is closed on $\calD\times\calD$, and the operator semi-group $\{ e^{t\calL} \}_{t\in\bbR}$ is defined via the Hille-Yosida theorem as a bounded map on $L^2_r(\bbR_+) \times L^2_r(\bbR_+)$. It satisfies the operator norm bound $\bigl\| e^{t\calL} \bigr\| \leq~e^{|t|}$ for all $t\in\R$. 
    
    \item For some constant $c > 0$ one has $\bigl\|e^{t\calL}\bigr\|\ge c \jap{t}$ for all $t \ge 0$.
    
    \item There exist absolute constants $C \geq 1$ and $0 < \delta_0 \ll 1$ such that for all intervals $I \subseteq [-\delta_0,\delta_0]$, the frequency-localized evolution $e^{t\calL} P_I$ introduced in Definition~\eqref{def:PIevol} satisfies  
    \begin{equation} \label{equ:operator_norm_bound_thm}
        \bigl\| e^{t\calL} P_I \bigr\| \leq C \jap{t}, \quad t \geq 0.
    \end{equation}    
    Moreover, if $J \subset [-\delta_0,\delta_0]$ is compact with $0 \notin J$, then the following stability bound holds
    \begin{equation} \label{equ:operator_norm_bound_thm_away_zero}
        \sup_{t \geq 0} \, \bigl\| e^{t\calL} P_J \bigr\| \leq C(J)
    \end{equation}
    for some constant $C(J) > 0$ whose size depends on $J$.
\end{enumerate}
\end{thm}
The bounds obtained in part (iv) indicate that the growth rate of the operator norm in (iii) is sharp and that it is related to the zero-energy resonance of the linearized operator.
As part of the proof of (iv), we provide a complete construction of the distorted Fourier transform associated with the linearized matrix operator $\calL$ at small energies. 
A similar analysis can also be carried out at intermediate and large energies. 
This leads to an oscillatory integral representation of the entire evolution in terms of the distorted Fourier transform, which serves as a launching point for proving dispersive decay estimates for the linear evolution.
The latter have in fact been obtained in the remarkable contemporaneous work \cite{CGP25}. Our findings complement the analysis in \cite{CGP25}.

At this point, we wish to highlight what we consider the most significant contribution of this work,
even though not all terminology has been systematically introduced yet.
It is well-known that the distorted Fourier transform for selfadjoint Schr\"odinger operators can be derived elegantly from Stone’s formula. Gesztesy-Zinchenko~\cite{GZ} follow this approach for selfadjoint Schr\"odinger operators on the half-line, which involves a detailed justification of passing to the limit in the resolvents from the upper and lower half-planes onto the spectrum. For $L^1_{\rm{loc}}$ potentials one can rely on the Herglotz property of the Weyl $m$-function. In that case, $\Im m(\lambda+i\varepsilon)$ converges in the weak-star sense to the Herglotz measure, which then acts as the spectral measure. Instead, for strongly singular potentials, such as the inverse square potentials arising in this work, the $m$-function is no longer Herglotz and~\cite{GZ} carefully deduce the existence of the limit of the resolvents onto the spectrum directly via the spectral theorem for selfadjoint operators. Note that the limiting measure can contain atoms, which then correspond to eigenvalues.

For non-selfadjoint operators -- such as the operator $\calL$ considered in this paper -- no such tool is available,  and the existence of the limit needs to be justified by hand. Due to the fact that zero-energy is embedded in the spectrum and that the operator exhibits delicate behavior there, as reflected by the linear growth in $t$ of the operator norm of the associated semi-group, we devote substantial effort to rigorously work out the existence and the shape of the limit of the resolvents at small complex energies $z = \lambda \pm i \epsilon$ for $\lambda \in \bbR$ as $\epsilon \to 0+$.  
In this regard, we point out that in the context of non-selfadjoint matrix Schr\"odinger operators arising from linearizing around solitary wave solutions to focusing nonlinear Schr\"odinger equations on the line, \cite[Lemma 6.8]{KS} passes to the spectrum. In other words, it applies the limiting absorption principle in a Stone-type formula that is obtained by hand using the Hille-Yosida theorem and contour integration. This is valid due to the presence of a spectral gap and the fact that there are no singularities on the spectrum; threshold resonances  and embedded eigenvalues are ruled out in \cite[Proposition 9.2]{KS}.
In contrast, the spectrum of the matrix operator $\calL$ considered here does not exhibit a gap and there is a resonance at zero energy.

It does not suffice to formally follow \cite{KS} as a procedure for real energies. Any approach that lacks a rigorous analysis in the complex plane is incomplete.
In fact, one of our main concerns had been whether or not a residue appears in the Stone-type formula from an expression of the form $z^{-2} e^{itz} T_0$ for a finite rank operator $T_0$, which would then contribute $t T_0$ to the evolution operator. We obtained a negative answer to this question only after working out the resolvents for small energies in the upper and lower half-planes and after computing the limit. Then it became clear that the linear growth in $t$ of the operator norm arises from the essential spectrum and not from any finite-dimensional subspace. But as the example $z^{-1}-\bar{z}^{-1}=\frac{-2i\varepsilon}{\lambda^2+\varepsilon^2}\to -2\pi i\delta_0$ in the weak-star sense shows, it is insufficient to consider only real energies $\lambda$.

\subsection{Proof ideas and discussion}

In this subsection we discuss the main ideas entering the proof of Theorem~\ref{thm:summary}.

\subsubsection{Spectrum of the linearized operator $\calL$}
In order to determine the spectrum of the linearized matrix operator $\calL$, we decompose it as 
\begin{equation}
    \calL = \calL_0 + \calV_0, \quad
    \calL_0 := \begin{bmatrix} 0 & -\partial_r^2+\frac{3}{4r^2} \\ -\bigl( -\partial_r^2 +\frac{3}{4r^2} +2 \bigr) & 0\end{bmatrix}, \quad \calV_0 := \begin{bmatrix} 0 & \rho_1^2 - 1 \\ -3(\rho_1^2-1) & 0 \end{bmatrix}.
\end{equation}
Here, $\calL_0$ is the corresponding \emph{free} non-selfadjoint matrix operator at spatial infinity, and $\calV_0$ is a matrix with globally bounded, decaying potentials.
Using the resolvent kernel of $i \calL_0$ computed in Appendix~\ref{app:calL0}, one can show that $i\calL_0$ only has essential spectrum equal to $\bbR$.
The analytic Fredholm alternative (see \cite[Theorem VI.14]{RS1}) then allows us to deduce that $\bbR \subseteq \mathrm{spec}(i\calL)$. 
On the other hand, by a delicate analysis using the Schur complement as a tool, we conclude that $\mathrm{spec}(i\calL) \subseteq \bbR$, whence $\mathrm{spec}(i\calL) = \bbR$. We refer to Lemma~\ref{lem:HYos} for the details. 

We point out that \cite[Theorem 6.1]{WeinXin96} establishes that $i\calL$ does not have eigenvalues in $\bbC\backslash\bbR$.
It is tempting to try to use a Weyl criterion type argument based on the decomposition \eqref{equ:intro_decomposition_calL} to deduce that $i\calL$ has the same essential spectrum as the free operator $i\calL_0$. One could then conclude that $\mathrm{spec}(i\calL) = \bbR$.
However, since the reference operator $\calL_0$ is not self-adjoint, it is for instance inadmissible to invoke the Weyl criterion as in \cite[Theorem XIII.14]{RS4}.

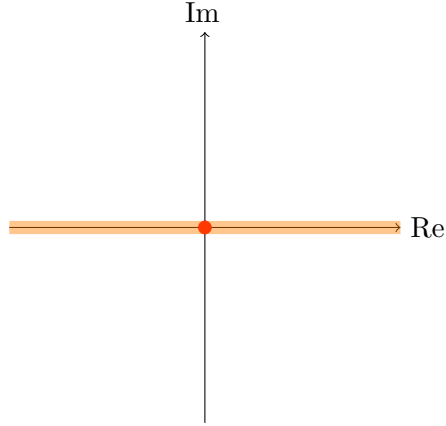
\begin{figure}[ht]
\centering
\begin{tikzpicture}[domain=-1.74:1.74,samples=50, scale=1.3]

\coordinate (px) at ($(0:0)$);		

\draw[->] ($(180:2)$) --++ ($(0:4)$) node[right] {$\Re$}; 
\draw[->] ($(270:2)$) --++ ($(90:4)$)node[above] {$\Im$};

\fill[red] (px) circle (2pt);		
\draw[line width=5pt,color=orange,draw opacity=0.45] (px) --++ ($(0:2)$);
\draw[line width=5pt,color=orange,draw opacity=0.45] (px) --++ ($(0:-2)$);	

\end{tikzpicture}
\caption{Spectral properties of the operator $i\calL$: The orange band corresponds to the spectrum. The red dot at the origin indicates the zero-energy resonance.}
\label{fig:spectrum_picture}
\end{figure}

\subsubsection{The zero-energy resonance and growth of the $L^2$-norm} 
As mentioned earlier, the phase invariance of~\eqref{equ:complex_GL} leads to a zero-energy resonance of $L_1$ of the form $L_1(\sqrt{r}\rho_1(r))=0$. At least formally this implies that $\ker\calL\subsetneq \ker\calL^2$. The formal aspect here is that these kernels are spanned by non-$L^2$ functions 
\begin{equation}
    \psi := \begin{bmatrix} 0 \\ \sqrt{r}\rho_1 \end{bmatrix}, \quad \widetilde{\psi} := \begin{bmatrix} L_2^{-1}(\sqrt{r}\rho_1) \\ 0 \end{bmatrix}
\end{equation}
satisfying $\calL\psi=0$ and $\calL \widetilde{\psi} = \psi$, whence $\calL^2 \widetilde{\psi} = 0$.  
Continuing the formal analogy, this leads to a nilpotent growth as in
\begin{equation}
    e^{tB} = \begin{bmatrix} 1 & t \\ 0 & 1 \end{bmatrix}, \quad B := \begin{bmatrix} 0 & 1 \\ 0 & 0 \end{bmatrix}.
\end{equation}
Since the aforementioned kernel elements of $\calL$ are not in $L^2$, this is not immediately obvious. 
In Lemma~\ref{lem:nilpotent} we exhibit linear growth of the operator norm of $e^{t\calL}$ by making the preceding formal analogy rigorous via carefully designed cut-offs.

\subsubsection{Construction of the distorted Fourier transform}
Most of the work in this paper goes into the meticulous construction of the distorted Fourier transform associated with the linearized operator $\calL$ at small energies, which gives rise to an oscillatory integral representation of the evolution $e^{t\calL}$.
Broadly speaking, our approach for the matrix operator $\calL$ mimics how the distorted Fourier transform can be constructed from Stone's formula for scalar selfadjoint Schr\"odinger operators with strongly singular potentials on the half-line, where the existence of the limit of the jump of the resolvent across the spectrum has to be carefully justified, as in \cite{GZ}.
While for selfadjoint Schr\"odinger operators Stone's formula just follows from the spectral theorem, no such tool is readily available for the non-selfadjoint matrix operator $\calL$.
Instead, we deduce an analogous Stone-type formula by hand following the scheme in \cite[Lemma 12]{ES2} and \cite[Lemma 6.8]{KS}.
Specifically, using a contour argument and a limiting absorption estimate, we show in Lemma~\ref{lem:6.8} that for any $\phi, \psi \in \calD \times \calD$ and any $b>0$,
\begin{equation}
         \bigl\langle e^{t\calL} \phi, \psi \bigr\rangle = \frac{1}{2\pi i} \int_{-\infty}^\infty e^{it\lambda} \bigl\langle \bigl[ e^{-bt} \bigl(i\calL-(\lambda +ib)\bigr)^{-1} - e^{bt} \bigl( i\calL-(\lambda -ib) \bigr)^{-1} \bigr] \phi, \psi\bigr\rangle \ud \lambda,
\end{equation}
where the indefinite integral converges.  

We then turn to proving that for any interval $I \subseteq [-\delta_0,\delta_0]$ with $0 < \delta_0 \ll 1$ sufficiently small, the following \emph{defining limit} of the frequency localized evolution $e^{t\calL} P_I$ exists and is given by the integral of the jump of the resolvent across the frequency interval $I$,
\begin{equation} \label{eq:Stonerepintro1}
\begin{aligned}
    \bigl\langle e^{t\calL} P_{I} \phi, \psi \bigr\rangle &:= \lim_{b\to0+} \frac{1}{2\pi} \int_{I} e^{it\lambda} \bigl\langle \bigl[ e^{-bt} \bigl(i\calL-(\lambda +ib)\bigr)^{-1} - e^{bt} \bigl( i\calL-(\lambda -ib) \bigr)^{-1} \bigr] \phi, \psi\bigr\rangle \ud \lambda \\  
    &= \int_{I} e^{it\lambda} \bigl\langle \bigl[ \bigl(i\calL-(\lambda +i0^{+})\bigr)^{-1} - \bigl( i\calL-(\lambda - i0^{+}) \bigr)^{-1} \bigr] \phi, \psi\bigr\rangle \ud \lambda,
\end{aligned}
\end{equation}
where it suffices to consider test functions $\phi, \psi \in L^2_r(\bbR_+) \times L^2_r(\bbR_+)$ with compact support.
Upon computing the jump of the resolvent across the spectrum, we can read off the distorted Fourier transform associated with $\calL$.

This process begins with the determination of the integral kernels $\calG_\pm(r,s;z)$ of the resolvent $(i\calL-z)^{-1}$ for $0 < |z| \leq \delta_0$ with $\pm \Im z > 0$ in terms of the Weyl solutions for $i\calL$ near $r=0$ and near $r=\infty$.
This leads to the formulae
\begin{equation} \label{equ:intro_calGplus_kernel}
 \begin{aligned}
  \calG_+(r,s;z) = \left\{ \begin{aligned} i \Psi_+(r,z) \bigl( D_+(z)^{-1} \bigr)^t F_1(s,z)^t \sigma_1,& \quad 0 < s \leq r, \\
  i F_1(r,z) D_+(z)^{-1} \Psi_+(s,z)^t \sigma_1,& \quad r \leq s < \infty, \end{aligned} \right.
 \end{aligned}
\end{equation}
and
\begin{equation} \label{equ:intro_calGminus_kernel}
 \begin{aligned}
  \calG_-(r,s;z) = \left\{ \begin{aligned} i \Psi_-(r,z) \bigl( D_-(z)^{-1} \bigr)^t F_1(s,z)^t \sigma_1,& \quad 0 < s \leq r, \\
  i F_1(r,z) D_-(z)^{-1} \Psi_-(s,z)^t \sigma_1,& \quad r \leq s < \infty. \end{aligned} \right.
 \end{aligned}
\end{equation}
Here, $\Psi_\pm(r,z)$ are the exponentially decaying Weyl-Titchmarsh matrix solutions near $r=\infty$ for $\pm\Im z>0$, while $F_1(r,z)$ is the matrix solution to $i\calL F = z F$, which belongs to $L^2_r((0,1))$ and is unique up to invertible linear combinations of its columns.
Moreover, $D_\pm(z) := \calW\bigl[\Psi_\pm(\cdot,z), F_1(\cdot,z)\bigr]$ are matrix Wronskians between $\Psi_\pm(\cdot,z)$ and $F_1(\cdot,z)$, and $\sigma_1$ is a Pauli matrix. 
Inserting these formulae into \eqref{eq:Stonerepintro1} and verifying that the limit $b\to0^+$ can be passed inside the integral over $\lambda \in I$ comes with a number of interesting challenges, some of which we now describe.

\medskip 

\noindent {\it The connection problem.}
Determining the entries of the matrix Wronskians $D_\pm(z)$ requires sharp estimates on the matrix solutions $F_1(r,z)$ and $\Psi_\pm(r,z)$ in overlapping regions of $r$. This is the connection problem, which we solve at the \emph{turning point} $r \simeq |z|^{-1}$. Specifically, we fix a string of small absolute constants $0 < \delta_0 \ll \epsilon_\infty \ll \epsilon_0 \ll 1$. Then we obtain precise estimates on $F_1(r,z)$ in the region $r \in [0,\epsilon_0 |z|^{-1})$ for $0 < |z| \leq \delta_0$, and we establish precise estimates on $\Psi_\pm(r,z)$ in the region $r \in [\epsilon_\infty |z|^{-1},\infty)$, where $0 < |z| \leq \delta_0$ with $\pm \Im z > 0$.
This allows us to evaluate the matrix Wronskians in the overlapping region $r \in [\epsilon_\infty |z|^{-1}, \epsilon_0 |z|^{-1}]$. 

\medskip 

\noindent {\it Power series construction of the Weyl solution near $r=0$.}
In order to construct $F_1(r,z)$ up to the turning point region, we proceed by a series expansion modeled on \cite{KST}. This is done by iteratively inverting $L_1$ and $L_2$ using Green's functions that are well-behaved near $r=0$. However, the exponential growth of the corresponding kernels at large $r \simeq |z|^{-1}$ leads to factorial growth of the coefficients of the series expansion. We are able to avoid this growth by working with truncated Green's functions. See Propositions \ref{lem:F1_1} and~\ref{lem:F2_1} for the details.

\medskip 

\noindent {\it Lyapunov-Perron construction of a fundamental system at $r=\infty$.}
We obtain the Weyl matrix solutions $\Psi_\pm(r,z)$ for large $r \gg 1$ as part of the construction of a fundamental system of matrix solutions at $r=\infty$ via a contraction mapping argument based on the Lyapunov-Perron method. This is natural since for energies $z \notin \bbR$ the system $i\calL \Psi = z \Psi$ exhibits both a two-dimensional stable and a two-dimensional unstable manifold. 

To gain a first understanding of the nuanced technical details, we consider as a natural candidate for the reference operator at large $r \gg 1$ the constant coefficient operator
\begin{equation}
    \calL_\infty^0 := \begin{bmatrix} 0 & -\partial_r^2 \\ -(-\partial_r^2 + 2) & 0 \end{bmatrix}
\end{equation}
obtained by replacing $L_1$ and $L_2$ in $\calL$ by their large $r$ formal limits $L_1^\infty := -\partial_r^2$ and $L_2^\infty := -\partial_r^2+2$. Then the system $i\calL_\infty^0 \Psi=z\Psi$ turns into a fourth-order scalar ordinary differential equation with solutions of the form $e^{ik(z)r}$, provided $k(z)$ is a root of
\begin{equation}\label{eq:zpolyintro1}
    P(k,z) := k^4 + 2 k^2 - z^2 = 0.
\end{equation}
One checks that the Riemann surface of the roots branches only at $z=\pm i$ (and at $z=\infty$) and that the four sheets near $z=0$ are $k_1(z)=-k_3(z)=z/\sqrt{2}+\calO(z^3)$ and $k_2(z)=-k_4(z)=i\sqrt{2}+\calO(z^2)$ as $z\to0$, see~\eqref{eq:four roots}. Thus, as we approach the real axis, i.e., as $b\to0^+$ in~\eqref{eq:Stonerepintro1}, we find two oscillatory solutions, one exponentially growing branch, and one exponentially decaying branch. Note that in contrast to the setting on the line in \cite{KS}, here we have a globally recessive branch corresponding to $k_2(z)$, and a globally dominant branch corresponding to $k_4(z)$. Also, unlike the scalar Schr\"odinger equation the oscillatory branches are, to leading order, of the form $e^{\pm izr}$ as opposed to $e^{\pm i\sqrt{z}r}$. This is related to the fact that the limiting operator $L_1^\infty$ is massless, while $L_2^\infty$ has non-zero mass. 
As a consequence, the evolution $e^{t\calL}$ exhibits wave-like behavior at small energies.

Returning to the construction of $\Psi_\pm(r,z)$, the goal is to obtain the two columns of $\Psi_\pm(r,z)$ as the (vector) solutions to $\calL \psi = z \psi$ that converge to the corresponding solutions of the reference operator $\calL_\infty^0 \psi = z \psi$ associated with the roots $k_1(z)$ and $k_2(z)$ when $\Im(z) > 0$, and with the roots $k_3(z)$ and $k_2(z)$ when $\Im(z) < 0$. 
However, the appearance of slowly decaying inverse square potentials in the difference $\calL - \calL_\infty^0$ is an obstruction to a perturbative construction via the Lyapunov-Perron method 
all the way down to the turning point.
Instead, we incorporate some of the inverse square potentials into the following definition of the reference operator at spatial infinity
\begin{equation}
 \calL_\infty := \begin{bmatrix} 0 & -\partial_r^2 - \frac{1}{4r^2} \\ -\bigl(-\partial_r^2 - \frac{1}{4r^2} + 2\bigr) & 0 \end{bmatrix}.
\end{equation}
This forces us to work with modified Hankel functions denoted by $h_\pm(k_j(z) r)$ in place of the exponential functions $e^{ik_j(z)r}$. The choice of which inverse square potentials to include in $\calL_\infty$ is motivated by the requirement that the resulting fourth-order ordinary differential equation have solutions of the form $h_\pm(k(z)r)$ with $k(z)$ a solution of \eqref{eq:zpolyintro1}. 
It turns out that the remaining inverse square potentials in $\calL-\calL_\infty$ can be treated perturbatively for small $z$, we refer to Lemma~\ref{lem:contr} for the details.

\medskip 

\noindent {\it Stokes phenomenon.} 
The Stokes phenomenon refers to  a  discontinuity of the asymptotic expansion of an analytic function $f(z)$ as $z$ crosses a curve in the complex plane, known as a \emph{Stokes line}. Specifically, the subdominant terms in the asymptotic expansion change discontinuously when crossing a Stokes line. We refer to \cite[Section 7.4]{Olv} and \cite[Section 4.4d]{Costin} for general background. 

The relevance of the Stokes phenomenon in this paper comes from the need to work with modified Hankel functions in the construction of the Weyl solutions at infinity $\Psi_\pm(r,z)$, $\pm\Im z>0$.
To see this note that in the expression for the resolvent kernel of $i\calL$, the modified Hankel functions appear as $h_{-}(k_1(\lambda)r) := h_{-}(k_1(\lambda+i0)r)$ and $h_{+}(k_1(\lambda)r) := h_+(k_1(\lambda+i0)r)$. 
While $h_+(k_1(\lambda)r)$ has leading order asymptotics $e^{ik_1(\lambda)r}$ as $r\to\infty$ for $\lambda \in \bbR\backslash\{0\}$, the behavior of $h_{-}(k_1(\lambda)r)$ is more subtle. By our choice of normalization, $h_{-}(k_1(\lambda)r) \sim e^{-ik_1(\lambda)r}$ as $r\to\infty$ for $\lambda>0$, but $h_{-}(k_1(\lambda)r) \sim e^{-ik_1(\lambda)r} + c e^{ik_1(\lambda)r}$ as $r\to\infty$ for $\lambda<0$ for some non-zero constant $c$. It is not trivial to extract the exact value $c = 2i$ of this constant from well-known references on special functions. As part of Lemma~\ref{lem:Hankel0}, we present a self-contained computation of this constant, which we hope to be of independent interest. 

\medskip 

\noindent {\it Potential singularities and passing the limit $b \to 0+$ inside the $\lambda$-integral.}
Due to the presence of an exponentially growing branch, it may in principle be possible that the matrix Wronskians $D_\pm(z)$, and thus the resolvent kernels $\calG_\pm(r,s;z)$ in \eqref{equ:intro_calGplus_kernel}--\eqref{equ:intro_calGminus_kernel}, could feature singular terms in $z$. This would potentially prohibit exchanging the order of taking the limit $b \to 0^+$ and integrating with respect to $\lambda$ in \eqref{eq:Stonerepintro1}. However, a careful analysis demonstrates that the potentially singular terms in $\calG_\pm(r,s;z)$ actually cancel out, see Corollary~\ref{cor:fy14cancel1}. 

\medskip 

\noindent {\it Reading off the distorted Fourier transform.}
The final outcome of the analysis of \eqref{eq:Stonerepintro1} is a representation formula for the evolution at small energies, which we record here as
\begin{equation} \label{eq:Stonerepintro2}
    \bigl\langle e^{t\calL} P_I \phi, \psi \bigr\rangle = \frac{1}{2\pi i} \int_I e^{it\lambda}\frac{\kappa(\lambda)}{d_{+}(\lambda) d_{-}(\lambda)} \bigl\langle \upfy(\cdot, \lambda), \sigma_1 \phi(\cdot) \bigr\rangle \bigl\langle \upfy(\cdot, \lambda), \psi(\cdot) \bigr\rangle \ud \lambda.
\end{equation}
The coefficient $\kappa(\lambda)$ stems from computing the matrix Wronskian between $\Psi_+(r,\lambda)$ and $\Psi_-(r,\lambda)$, while $d_\pm(\lambda)$ are the determinants of the matrix Wronskians $D_\pm(\lambda) := D_\pm(\lambda \pm i0)$. The vector $\upfy(r,\lambda)$ is a linear combination of the columns of $F_1(r,\lambda)$, and for large $r \gg 1$ the entries of $\upfy(r,\lambda)$ are linear combinations of $e^{ik_1(\lambda)r}$ and $e^{ik_3(\lambda)r}$.
One can read off the distorted Fourier transform associated with $\calL$ from \eqref{eq:Stonerepintro2}.
See Proposition~\ref{prop:PIStone1} and Lemma~\ref{lem:varfystructure}.

\subsubsection{Upper bound on the growth of the $L^2$ norm}
The representation formula \eqref{eq:Stonerepintro2} for the evolution $e^{t\calL}P_I$ sheds more light on the nature of the $L^2$ norm growth phenomenon. Our analysis reveals that the most singular term with respect to $\lambda$ in \eqref{eq:Stonerepintro2} is schematically of the form
\begin{equation} \label{eq:sintlambdaintro1}
    \int_{-\delta_0}^{\delta_0}\frac{e^{it\lambda}}{|\lambda|}\mathrm{sgn}(\lambda) \bigl\langle\uppsi(\cdot,\lambda), \sigma_1 \phi \bigr\rangle \bigl\langle\uppsi(\cdot,\lambda),\psi \bigr\rangle \ud\lambda = 2i \int_{0}^{\delta_0}\frac{\sin(t\lambda)}{\lambda} \bigl\langle\uppsi(\cdot,\lambda), \sigma_1 \phi\bigr\rangle \bigl\langle \uppsi(\cdot,\lambda),\psi \bigr\rangle \ud \lambda,
\end{equation}
where $\lambda \mapsto \uppsi(\cdot,\lambda)$ is even, and one has $\|\langle\uppsi(\cdot,\lambda),\phi\rangle\|_{L^2_\lambda([0,\delta_0])}\lesssim \|\phi\|_{L^2_r(0,\infty)}$. 
The estimate $|\sin(t\lambda)|\leq t|\lambda|$ then allows us to prove the upper bound on the growth of the $L^2$ norm of $e^{t\calL} P_I \phi$, see Section~\ref{sec:proofofthm} for the details. Formula~\eqref{eq:sintlambdaintro1} also shows that the growth of the $L^2$ norm is not a finite rank phenomenon, and in fact it is not difficult to see that there are no $L^2$ functions $\phi$ for which $\|e^{t\calL}P_I\phi\|_{L^2_r(0,\infty)}$ grows like $t$.\footnote{For any $a\in [0,1)$, the uniform boundedness principle implies the existence of functions $\phi$ in $L^2_r(\bbR_+)$ for which $\|e^{t\calL}\phi\|_{L^2_r(0\infty)}\gtrsim t^{a}$.}  
Verifying that $\lambda \mapsto \uppsi(\cdot,\lambda)$ is even with respect to $\lambda$ is rather subtle, because it involves the Stokes phenomenon described above. Special care is needed in showing that the constant $c$ in the asymptotic expansion $h_{-}(k_3(\lambda)r)\sim e^{-ik_1(\lambda)r}+c e^{ik_1(\lambda)r}$ for $\lambda < 0$ does not affect the parity of~$\lambda\mapsto\uppsi(\cdot,\lambda)$.

\subsubsection{Projecting away the $L^2$ growth}

Our proof also shows that if $\|\jap{r}^{\frac{1}{2}}\phi\|_{L^1_r(\bbR_+)}<\infty$, then 
\begin{equation}
    \bigl\| e^{t\calL} P_I \phi \bigr\|_{L^2_{r}(\bbR_+)} \lesssim \sqrt{\log(1+\jap{t})}  \bigl\|\jap{r}^{\frac{1}{2}}\phi \bigr\|_{L^1_r(\bbR_+)}.
\end{equation}
Indeed, this follows from the fact that $\uppsi$ in \eqref{eq:sintlambdaintro1} satisfies $|\uppsi(r,\lambda)|\lesssim \sqrt{|\lambda| r}$ for $r \lesssim |\lambda|^{-1}$ and $|\uppsi(r,\lambda)|\lesssim 1$ for $r \gtrsim |\lambda|^{-1}$. In fact, $\uppsi(r,\lambda) = (0, \rho_1(r))+\calO(|\lambda|^2r^{\frac{5}{2}})$, so if $\phi = (\phi_1, \phi_2)$ has sufficient decay and $\int_0^\infty \phi_1(r) r^{\frac12} \rho_1(r) \ud r = 0$, then $\|e^{t\calL} P_I \phi\|_{L^2_r(\bbR_+)}$ is bounded uniformly in time. However, as mentioned above, the $L^2$-norm growth is not a finite rank phenomenon when we view $e^{t\calL}P_I$ as an operator on $L^2_r(\bbR_+) \times L^2_r(\bbR_+)$. 

\subsection{References}

The Ginzburg-Landau equation~\eqref{equ:complex_GL} arises in the theory of superconductivity \cite{BBH,WE,FHLin}, and the Gross-Pitaevskii equation \eqref{eq:GP1} models superfluids \cite{CM,Don,EM,FPR,Neu90}. 

The existence and uniqueness of vortex solutions to the Ginzburg-Landau equation \eqref{equ:complex_GL} are established in \cite{CEQ94,HerHer94}.
Global existence of solutions to the Gross-Pitaevskii equation \eqref{eq:GP1} for finite energy initial data is proved in~\cite{Gerard}.
The final-state problem for small perturbations of the constant equilibrium solution to the Gross-Pitaevskii equation \eqref{eq:GP1} is considered in \cite{GNT1}. We note that in three space dimensions, asymptotic stability of the constant equilibrium solution is established in \cite{GNT1,GNT2, GuoHaniNakanishi18}. 
The orbital stability of the degree-one vortex under the Gross-Pitaevskii evolution is shown in \cite{GPS21}, see also \cite{dPFK,Wein}. Moreover, we mention the recent result \cite{dPJM} on the invertibility of the linearized operator around the degree-one vortex. For other related works we refer to \cite{BJS,BS,CJ,CP1,CP2,CP3,JOST,JS,JS2,LX,OS1,OS2,OS3} and references therein.

We also point out the closely related recent work \cite{PP24} on dispersive decay estimates for the linearized evolution around Ginzburg-Landau vortices in the relativistic case in the absence of an electro-magnetic field.

The construction of the distorted Fourier transform associated with the linearized operator $\calL$ in this paper in particular builds on techniques and insights from \cite{KS, GZ, KST, KMS, BP, ES2}.
The space-time resonances method based on the distorted Fourier transform for $\calL$ appears as a promising approach to tackle the full asymptotic stability problem for the degree-one vortex under the Gross-Pitaevskii evolution. In recent years this method has emerged as a powerful tool in the study of the asymptotic stability of solitons when dispersion is weak.
We refer to the review article \cite{Germain24_Review} and to \cite{GP20, GermPusZhang22, DelMas20, LS1, Chen21, CL1, CG23, LL2} for a sample of recent works. It is worth pointing out that \cite{KS, CG23, LL2, CL1, Li23} involve the development of the spectral and distorted Fourier theory for non-selfadjoint matrix operators.

\subsection{Organization of the paper}

In Section~\ref{sec:setup} we prove that the spectrum of $i\calL$ is $\bbR$ and we define the evolution via the Hille-Yosida theorem. Moreover, we exhibit the linear growth of the operator norm $\|e^{t\calL}\|$, and we establish a Stone-type formula for the linear evolution.
In Section~\ref{sec:Fs} we construct a fundamental system of matrix solutions to $i\calL F = zF$ near $r=0$, and in Section~\ref{sec:Psis} we obtain a fundamental system of matrix solutions near $r=\infty$.
Building on the analysis from Sections~\ref{sec:Fs} and~\ref{sec:Psis}, we compute the Wronskians $D_{\pm}(z)$ in Section~\ref{sec:Greens}. Here we also justify the existence of the limit in \eqref{eq:Stonerepintro1} and prove that the limit can be taken inside the integral in this formula. This leads to formula \eqref{eq:Stonerepintro2}, and the derivation of a number of finer structural properties of the distorted Fourier basis elements. In Section~\ref{sec:proofofthm}, we present the upper bound on the growth of the operator norm $\|e^{t\calL}P_I\|$ and conclude the proof of Theorem~\ref{thm:summary}.

\subsection{Notation and conventions}

We collect several notational conventions that will be used in this work. An absolute constant whose value may change from line to line is denoted by $C>0$. For non-negative $X$ and $Y$ we write $X\lesssim Y$ if $X \leq C Y$ and $X\ll Y$ to indicate that the implicit constant is small. To emphasize the dependence of the constant $C$ on a parameter or function $a$ we use the notations $C(a)$ and $C_a$. Similarly $C(a_1,\dots,a_n)$ and $C_{a_1,\dots, a_n}$ denote dependence on  $a_1,\dots,a_n$. We use the notation $Y=\calO(X)$ if $|Y|\leq CX$. The notation $\bfJ$ stands for the matrix
\begin{align*}
    \bfJ:=\pmat{0&1\\-1&0},
\end{align*}
and $\sigma_1$, $\sigma_2$, $\sigma_3$ denote the Pauli matrices
\begin{equation} \label{eq:pauli}
    \sigma_1 = \begin{bmatrix} 0 & 1 \\ 1 & 0 \end{bmatrix}, \quad \sigma_2 = \begin{bmatrix} 0 & -i \\ i & 0 \end{bmatrix}, \quad \sigma_3 = \begin{bmatrix} 1 & 0 \\ 0 & -1 \end{bmatrix}.
\end{equation}
For a complex number $z$ we often denote the real part by $\lambda$ and the imaginary part by $b$. The notation $\|\cdot\|_{L^2_r}$ denotes the $L^2$ norm on the half-line with respect to the measure $\ud r$. We also use the notation $X_0:=L^2(\bbR_{+})$, $\bbR_{+}:=[0,\infty)$, with respect to the measure $\ud r$, and let $X_1:=X_0\times X_0$. For real-valued scalar functions $f$ and $g$ we use the notation
\begin{align*}
    \langle f,g\rangle:=\int_0^\infty f(r)g(r)\ud r.
\end{align*}
For $\bbR^2$-valued functions $\bmf = (f_1,f_2)$ and $\bmg = (g_1,g_2)$, we use the notation $\langle \bmf,\bmg\rangle:=\langle f_1,g_1\rangle+\langle f_2,g_2\rangle$.

\section{Spectral Properties of the Linearized Operator}\label{sec:setup}
In this section we establish basic spectral properties of $i\calL$ and we define the evolution $e^{t\calL}$ using the Hille-Yosida theorem. We also prove the lower bound on the growth of $\|e^{t\calL}\|$ in Subsection~\ref{subsec:L2growth}, and derive a Stone-type formula in Subsection~\ref{sec:stone} 

\subsection{Basics}

Recall that $\rho_1(r)$ denotes the unique solution to the ordinary differential equation \eqref{eq:rhoeq1} and satisfies the asymptotics~\eqref{eq:rhobounds1}.
For the remainder of the paper we adhere to the following notations. 

\begin{defi}\label{def:basic}
We define the following operators in the Hilbert space $$X_0:=L^2(\R_+)$$ relative to the Lebesgue measure:
   \begin{align*}
    L_0 &:= -\partial_r^2 + \frac{3}{4r^2}, \qquad L:= L_0 + 2\rho_1(r)^2 -1 \\
    L_1 &:= L_0 + \rho_1(r)^2 -1,\qquad L_2 := L_0 + 3\rho_1(r)^2 -1
\end{align*}
As domain we choose $\calD_0:= C^2_c(\R_+)$.
We also define the matrix operator
\begin{align*}
    \calL &:= \left[ \begin{matrix}
        0 & L_1 \\ -L_2 & 0 
    \end{matrix}\right],
\end{align*} 
 on the domain $\calD_0\times\calD_0$. 
\end{defi}

We denote fundamental systems of solutions for the scalar operators $L_1$, respectively $L_2$, by $\{f_1, f_2\}$, respectively $\{g_1, g_2\}$. In other words, $L_1 f_j = 0$, respectively $L_2 g_j = 0$, for $j = 1,2$.

Observe that if $L_1 f = 0$, then the possible asymptotics for $f$ are $\{ r^{-\frac12}, r^{\frac32} \}$ near zero and $\{ r^{\frac12}, r^{\frac12} \log(r) \}$ near infinity.
Since $f_1(r) := r^{\frac12} \rho_1(r)$ is one explicit solution, the two possible asymptotics are (with some constants $c\ne0$ that can change from line to line)
\begin{equation}
\label{eq:f1f2}
 \begin{aligned}
  f_1(r) &\sim \left\{ \begin{aligned} c r^{\frac32}, \quad r \to 0+, \\
                                        c r^{\frac12}, \quad r \to \infty, \\
                        \end{aligned} \right. \\
  f_2(r) &\sim \left\{ \begin{aligned} c r^{-\frac12}, \quad r \to 0+, \\
                                        c r^{\frac12} \log(r), \quad r \to \infty, \\
                        \end{aligned} \right.
 \end{aligned}
\end{equation}
Similarly, if $L_2 g = 0$, then the possible asymptotics for $g$ are $\{ r^{-\frac12}, r^{\frac32}\}$ near zero and $\{ e^{\sqrt{2}r}, e^{-\sqrt{2}r} \}$ near infinity. Since zero is not an eigenvalue of $L_2$ we can select a fundamental system such that
\begin{equation}
\label{eq:g1g2}
 \begin{aligned}
  g_1(r) &\sim \left\{ \begin{aligned} c r^{\frac32}, \quad r \to 0+, \\
                                        c e^{\sqrt{2} r}, \quad r \to \infty, \\
                        \end{aligned} \right. \\
  g_2(r) &\sim \left\{ \begin{aligned} c r^{-\frac12}, \quad r \to 0+, \\
                                        c e^{-\sqrt{2} r}, \quad r \to \infty.
                        \end{aligned} \right.
 \end{aligned}
\end{equation}

Define
\begin{equation} \label{eq:F1F2}
 \begin{aligned}
  F_1(r) := \begin{bmatrix} 0 & g_1(r) \\ f_1(r) & 0 \end{bmatrix}, \quad F_2(r) := \begin{bmatrix} 0 & g_2(r) \\ f_2(r) & 0 \end{bmatrix}.
 \end{aligned}
\end{equation}
Then we have $i \calL F_j = 0$, $j = 1, 2$. Note that $F_1\in L^2(0,1)$ and it is the unique fundamental matrix of $\calL$ with this property up to $F_1\mapsto F_1 B$ where $B\in \GL(2,\Cp)$.

\subsection{Determining the spectrum and applying Hille-Yosida}\label{sec:HY}

We begin with a description of the spectrum of the scalar operators $L_1$ and $L_2$ in the definition of $\calL$.
\begin{lemma}
The operators $L_1, L_2$ are essentially self-adjoint, 
    and  self-adjoint and positive with domain 
    \[
    \calD:= \{ g\in H^2_{\loc}(\R_+)\cap X_0\::\: L_0\, g \in X_0\}. 
    \] 
    The spectrum of $L_1$ is $\spec(L_1)=[0,\infty)$,  which is purely absolutely continuous. For some $c_0>0$ one has $\spec(L_2)\subset [c_0,\infty)$ and $\spec(L_2)\cap [2,\infty)=\spec_{\ess}(L_2)$
    is the essential spectrum, which is purely absolutely continous.
\end{lemma}
\begin{proof}
These operators are limit point at both end points $r=0$, and $r=\infty$. Indeed, any solution of $Lf=0$ or $L_jf=0$  is asymptotic to a linear combination of $r^{\frac32}, r^{-\frac12}$ as $r\to0+$, so no boundary condition is needed at $r=0$. The other endpoint $r=\infty$ is standard, see~\cite[Theorem X.8]{RS}. The claim about essential self-adjointness is~\cite[Theorem X.7]{RS}. 
    See \cite[Section 4]{GZ} for the Bessel operator $L_0$ and its domain.    The Kato-Rellich theorem applies to $L,L_1,L_2$, which are relatively bounded perturbations of $L_0$,  see \cite[Section 1.4]{Dav} (in fact, these operators are perturbations of~$L_0$ by bounded operators). The Weyl criterium, see~\cite[Theorem XIII.14]{RS4}, implies that $\spec_{\ess}(L_0)=\spec_{\ess}(L_1)=[0,\infty)$. If $\lambda\in \spec(L_1)$ for some $\lambda<0$, then there exists a ground state of negative energy, i.e., $L_1\phi =\lambda_0\phi$ for some $\lambda_0<0$ and $\phi\in\calD$ with $\phi>0$ (in fact, $\phi$ is smooth and $\phi(r)\sim cr^{\frac32}$ as $r\to0+$ for some constant~$c$). Let $\chi(r)=1$ for $0\le r\le 1$, $\chi\in C^\infty([0,\infty))$ with compact support,  and set $\chi_b(r)=\chi(r/b)$ with $b\ge1$. Then
    \[
    \langle L_1 \phi , \chi_b(r) r^{\frac12}\rho_1\rangle = \lambda_0 \langle \phi, \chi_b(r) r^{\frac12}\rho_1\rangle
    \]
    Integrating by parts on the left-hand side and sending $b\to\infty$ now leads to a contradiction because of the vanishing $L_1(r^{\frac12}\rho_1)=0$. Hence $L_1>0$ as stated (note that $L_1$ cannot have a zero energy eigenfunction because the unique $0$-energy solution is not in~$X_0$). Pure a.c.~spectrum is a consequence of the construction of the Weyl, Titchmarsh $m$-function for these operators, see~\cite{GZ, KST}.

    The essential spectrum of $L_2$ follows from the Weyl criterium as before. Since $L_2=L_1+2\rho_1^2>0$, we conclude that the discrete spectrum  of $L_2$ -- if it exists -- is strictly positive. 
\end{proof}

\begin{rem}
    The exact value of $c_0>0$ is not known, but the approximate value $c_0 \approx 1.3326$ is obtained in \cite{PP24} via a numerically assisted argument. Moreover, \cite{PP24} shows that $L_2$ has infinitely many eigenvalues in $(c_0,2)$ and that $2$ is a resonance.
\end{rem}

Next, we determine the spectrum of $i\calL$ and define the evolution $e^{t\calL}$ using the Hille-Yosida theorem. As we do not have a selfadjoint reference operator available as required for the standard version of Weyl's theorem as in \cite[Theorem XIII.14]{RS4}, we need to proceed differently. 
To this end we first obtain a proper understanding of the resolvent $(i\calL_0-z)^{-1}$ of the \emph{free} operator
\begin{align*}
\begin{split}
\calL_0 := \pmat{0 & -\partial_r^2+\frac{3}{4r^2} \\ -\bigl(-\partial_r^2+\frac{3}{4r^2}+2\bigr) & 0}.
\end{split}
\end{align*}
Let $\tilp_+(\zeta)$ denote the modified Hankel function
\begin{align*}
\begin{split}
\tilp_+(\zeta):=\sqrt{\zeta}H_1^{(1)}(\zeta) =\sqrt{\zeta} \big(J_1(\zeta)+iY_1(\zeta)\big),
\end{split}
\end{align*}
and let $\tilq_+(\zeta)$ denote the modified Bessel function
\begin{align*}
\begin{split}
\tilq_+(\zeta):=\sqrt{\zeta} J_1(\zeta).
\end{split}
\end{align*}
Here $H_1^{(1)}$, $J_1$, and $Y_1$ denote the order one Hankel function, Bessel function of the first kind, and Bessel function of the second kind respectively. $\tilp_{+}$ and $\tilq_{+}$ satisfy the ODE
\begin{align*}
\begin{split}
-\frac{\ud ^2}{\ud\zeta^2} \tilp_{+}(\zeta)+\frac{3}{4\zeta^2}\tilp_{+}(\zeta)=\tilp_{+}(\zeta),\qquad -\frac{\ud ^2}{\ud\zeta^2} \tilq_{+}(\zeta)+\frac{3}{4\zeta^2}\tilq_{+}(\zeta)=\tilq
_{+}(\zeta).
\end{split}
\end{align*}
Then by direct inspection, the vectors
\begin{align*}
\begin{split}
&\tiluppsi_1(r,z):=\pmat{\frac{ik_1(z)^2}{z}\tilp_{+}(k_1(z)r)\\\tilp_{+}(k_1(z)r)},\quad \tiluppsi_2(r,z):=\pmat{\frac{ik_2(z)^2}{z}\tilp_{+}(k_2(z)r)\\\tilp_{+}(k_2(z)r)},\\
&\tiluppsi_3(r,z):=\pmat{\frac{ik_1(z)^2}{z}\tilq_{+}(k_1(z)r)\\\tilq_{+}(k_1(z)r)},\quad \tiluppsi_4(r,z):=\pmat{\frac{ik_2(z)^2}{z}\tilq_{+}(k_2(z)r)\\\tilq_{+}(k_2(z)r)},
\end{split}
\end{align*}
satisfy
\begin{align*}
\begin{split}
(i\calL_0-z)\tiluppsi_j(r,z)=0.
\end{split}
\end{align*}
Let 
\begin{align*}
\begin{split}
\Psi(\cdot,z):=\pmat{\tiluppsi_1(\cdot,z)&\tiluppsi_2(\cdot,z)},\qquad \Theta(\cdot,z):=\pmat{\tiluppsi_3(\cdot,z)&\tiluppsi_4(\cdot,z)}.
\end{split}
\end{align*}
Note that $\Psi(r,z)$ is $L^2$ for large $r$  and $\Theta(r,z)$ is $L^2$ near $r=0$. Let
\begin{align}\label{eq:calGijLAP1}
\begin{split}
\calG_{ij}(r,s,z):=\rc\big(\tiluppsi_i(r,z)\tiluppsi_j^t(s,z)\mathds{1}_{[0 < s \leq r]}+\tiluppsi_j(r,z)\tiluppsi_i^t(s,z)\mathds{1}_{[ s > r]}\big)\sigma_1
\end{split}
\end{align}
and
\begin{align*}
\begin{split}
\tils_{ij}(z):=W[\tiluppsi_i(\cdot,z),\tiluppsi_j(\cdot,z)].
\end{split}
\end{align*}
The following lemma gives a formula for the kernel of the resolvent $(i\calL_0-z)^{-1}$. 
\begin{lemma}\label{lem:calL0res1}
For $0 < |z| \leq \delta_0$ with $\Im z>0$ the kernel of the resolvent $(i\calL_0-z)^{-1}$ is given by
\begin{align}\label{eq:calG0tocalGij1}
\begin{split}
\calG_0(r,s,z)=\frac{1}{\tils_{13}(z)}\calG_{13}(r,s,z)+\frac{1}{\tils_{24}(z)}\calG_{24}(r,s,z).
\end{split}
\end{align}
Moreover, 
\begin{align}\label{eq:tilsij1}
    |\tils_{13}(z)|\simeq |\tils_{24}(z)|\simeq\sqrt{|z|}.
\end{align}
\end{lemma}
The proof of this lemma uses ideas from later sections and is independent of the rest of the material in this section. For this reason we have have deferred the proof to Appendix~\ref{app:calL0}.

Now we are in the position to determine the spectrum of $i\calL$ and define the evolution $e^{t\calL}$. 

\begin{lemma} \label{lem:HYos}
    The matrix operator $\calL$ is closed on $\calD\times\calD\subset X_0\times X_0=: X_1$, with $\spec(i\calL)= \R$. It generates a group of bounded operators on~$X_1$ satisfying 
    \begin{equation}
        \label{eq:HY}
        \|e^{t\calL}\|_{X_1}\le e^t
    \end{equation}  for all $t\ge0$. Moreover, if $\begin{bmatrix} f(t) \\ g(t) \end{bmatrix} = e^{t\calL} \begin{bmatrix} f_0 \\ g_0 \end{bmatrix}$ for $f_0, g_0\in \calD$, then 
    \[
    \langle L_2 f(t), f(t)\rangle + \langle L_1 g(t), g(t)\rangle = \mathrm{const}.
    \]
    In particular, $\|f(t)\|_{X_0}\les C(f_0, g_0)$ for all $t\ge0$. 
\end{lemma}
\begin{proof}
    The closedness is elementary to prove, using the previous lemma and the fact that self-adjoint operators are closed. 
    Let $T(z)=z^2 L_2^{-1}+L_1$, which is closed on $\calD\subset X_0$.
    As we will see below, this operator arises as Schur complement of $\calL-z$. 
    We claim that for all $z\in\Cp\setminus i\R$, there is a bounded inverse $T(z)^{-1}:X_0\to X_0$. To see this, we write
    \[
    T(z)= L_2^{-\frac12}(z^2 + T_0)  L_2^{-\frac12},\qquad T_0=  L_2^{\frac12} L_1  L_2^{\frac12}=L_2^2  - 2L_2^{\frac12} \rho_1^2 L_2^{\frac12}.
    \]
    We first check that $T_0$ is self-adjoint and positive with domain $L_2^{-1}\calD=L_2^{-2}X_0$. It is clear that $T_0$ is symmetric and positive on this domain. Let $g,h\in X_0$ satisfy \[\langle T_0 f,g\rangle = \langle f,h\rangle \qquad \forall\; f\in L_2^{-2}X_0. \] 
    The goal is to show that $g\in L_2^{-2}X_0$. 
    Thus, writing $f=L_2^{-2}\tilde f$ with $\tilde f\in X_0$ we obtain
    \begin{equation}
        \label{eq:adj1}
        \langle \tilde f,g\rangle = 2   \langle L_2^{\frac12} \rho_1^2 L_2^{-\frac32} \tilde f,g\rangle + \langle \tilde f, L_2^{-2}h\rangle.
    \end{equation}
    Next, we verify that $B_{\frac12}:=L_2^{\frac12} \rho_1^2 L_2^{-\frac12}$ is bounded on~$X_0$. By complex interpolation, it suffices to show that $B_s:=L_2^{s} \rho_1^2 L_2^{-s}$ is bounded on the lines $\Re s=0$, respectively $\Re s=1$. Since $L_2^s$ is unitary if $s\in i\R$ it is further enough to verify this property at $s=1$. In other words, for 
    \[
    B_1 = \rho_1^2 - ((\rho_1^2)'' +2 (\rho_1^2)'\partial_r) L_2^{-1}.
    \]
    This further reduces to the showing that $\| r\langle r\rangle^{-4}\partial_r  L_2^{-1} f\|_2\les \|f\|_2$ for all $f\in X_0$. The resolvent $L_2^{-1}$ has the integral representation 
    \[
    (L_2^{-1} f)(r) = \int_0^r \psi_{-}(s)\psi_{+}(r) f(s)\ud s + \int_r^\infty \psi_{-}(r)\psi_{+}(s) f(s)\ud s
    \]
with smooth solutions $L_2 \psi_{\pm}=0$ satisfying  $\psi_{-}(r)\sim a_-\, r^{\frac32}$ and $\psi_{+}(r)\sim a_+\, r^{-\frac12}$ as $r\to 0+$, as well as $\psi_{-}(r)\sim b_-\, e^{\sqrt{2} r}$ and $\psi_{+}(r)\sim b_+\, e^{-\sqrt{2} r} $ as $r\to \infty$ (with the corresponding asymptotic relations for the derivatives). Moreover, the Wronskian $W[\psi_{-}, \psi_{+}]:=\psi_-\psi_+'-\psi_-'\psi_+=1$. Hence, 
\[
    \partial_r (L_2^{-1} f)(r) = \int_0^r \psi_{-}(s)\psi_{+}'(r) f(s)\ud s + \int_r^\infty \psi_{-}'(r)\psi_{+}(s) f(s)\ud s.
\]
By the Schur's test, this is an $X_0$-bounded operator whence our assertion that $B_{\frac12}$ is bounded. Returning to~\eqref{eq:adj1}, we conclude by taking adjoints in the second term that
\[
g= 2 L_2^{-\frac32}  \rho_1^2  L_2^{\frac12} g + L_2^{-2} h \in L_2^{-1} X_0.
\]
    Bootstrapping this relation yields 
   \[
g= L_2^{-2} \big( 2 L_2^{\frac12}  \rho_1^2  L_2^{-\frac12} L_2 g +  h)  \in L_2^{-2} X_0
\] 
as desired. 
    Therefore, we have self-adjointness and for all $z^2\not \in \R_{-}\cup\{0\} $, the bounded inverse $(z^2+T_0)^{-1}:X_0\to X_0$ exists. Furthermore, we claim that 
    \begin{equation}
    \label{eq:reg gain}
        (z^2+T_0)^{-1}: L_2^{\frac12} X_0\to L_2^{-\frac12} X_0
    \end{equation} 
    as a bounded operator, which concludes the proof that~$T(z)^{-1}$ is bounded on~$X_0$. Indeed, we claim that $\|L_2g\|_{X_0}\leq C\|h\|_{X_0}$ whenever 
    \begin{equation}\label{eq:gh}
    (T_0+z^2)g=L_2h,
    \end{equation}
    which suffices by complex interpolation. 
To prove this, we use elliptic estimates to bound the derivatives of $g$ in $X_0$, and the $X_0$ boundedness of $(T_0+z^2)^{-1}$ to bound $g$ itself. First, recalling that~$T_0=L_2^2-2L_2^{\frac{1}{2}}\rho_1^2L_2^{\frac{1}{2}}$ and applying $L_2^{-1}$ to \eqref{eq:gh} we get
\begin{align*}
L_2g= h+2B_{\frac{1}{2}}^\ast g -z^2L_2^{-1}g.
\end{align*}
By the preceding, $B_{\frac12}^*:X_0\to X_0$ is bounded. 
It follows that $\|L_2g\|_{X_0}\leq C(\|h\|_{X_0}+\|g\|_{X_0})$. On the other hand, from \eqref{eq:gh}
\begin{align*}
g&=(T_0+z^2)^{-1}L_2h \\
&= (T_0+z^2)^{-1}(L_2^2-2L_2^{\frac{1}{2}}\rho_1^2L_2^{\frac{1}{2}})L_2^{-1}h+2(T_0+z^2)^{-1}B_{\frac{1}{2}}h\\
&=L_2^{-1}h-z^2(T_0+z^2)^{-1}L_2^{-1}h+2(T_0+z^2)^{-1}B_{\frac{1}{2}}h.
\end{align*}
Using the boundedness of $(T_0+z^2)^{-1}$ in $X_0$ we conclude that $\|g\|_{X_0}\leq C\|h\|_{X_0}$, as desired.
    Returning to the matrix operator, we now write
    \[
    \calL - z   = \left[\begin{matrix}
        1 & zL_2^{-1} \\ 0 & 1 
    \end{matrix}\right] \left[\begin{matrix}
        0 & T(z) \\ -L_2 & -z
    \end{matrix}\right]
    \]
    whence for any $z\in\Cp\setminus i\R$, 
    \begin{align*}
        ( \calL - z)^{-1} &= \left[\begin{matrix}
        -zL_2^{-1} T(z)^{-1} & -L_2^{-1} \\ T(z)^{-1}  & 0
    \end{matrix}\right]  \left[\begin{matrix}
        1 & -zL_2^{-1} \\ 0 & 1 
    \end{matrix}\right] \\
    &= \left[\begin{matrix}
        -zL_2^{-1} T(z)^{-1} & z^2 L_2^{-1} T(z)^{-1} L_2^{-1}-  L_2^{-1} \\ T(z)^{-1}  & -zT(z)^{-1}L_2^{-1}     \end{matrix}\right]
    \end{align*}
    as a bounded operator on $X_0\times X_0$. This shows that $\mathrm{spec}(i\calL)\subseteq \bbR$.

To show that $\bbR\subseteq \mathrm{spec}(i\calL)$ we argue by contradiction. Suppose $\lambda\in \bbR$ is in the resolvent set for $i\calL$. Then a neighborhood $\bfD\subseteq\bbC$ of $\lambda$ is also in the resolvent set. Let $\calV_0:=i\calL-i\calL_0$. Then in view of the decay of $\calV_0$ the operator $\calV_0(i\calL-z)^{-1}$ is compact on $X_0\times X_0$. See for instance \cite{ES1}. By the analytic Fredholm alternative, cf. \cite[Theorem VI.14]{RS1}, either $I-\calV_0(i\calL-z)^{-1}$ is never invertible for $z\in \bfD$ or it is invertible for all $z\in \bfD\backslash{\bfS}$, where $\bfS$ is a (possibly empty) discrete subset of $\bfD$. The former cannot hold because for $\Im z>0$ we have the explicit inverse formula
\begin{align*}
\big(I-\calV_0(i\calL-z)^{-1}\big)^{-1}=(i\calL-z)(i\calL_0-z)^{-1}=I+\calV_0(i\calL_0-z)^{-1}.
\end{align*}
The resolvent identity
\begin{align*}
    (i\calL_0-z)^{-1}=(i\calL-z)^{-1}\big(I-\calV_0(i\calL-z)^{-1}\big)^{-1},
\end{align*}
for $\Im z>0$, then shows that $(i\calL_0-z)^{-1}$ extends to a bounded operator to all $z\in \bfD\backslash\bfS$. But, Lemma~\ref{lem:calL0res1} and the asymptotics for Bessel and Hankel functions, see~\eqref{eq:tilhbigzeta1} below, show that the resolvent $(i\calL_0-\lambda)^{-1}$ is not $L^2$-bounded for $\lambda\in\bbR$. This is a contradiction. 

For the time evolution, it is more convenient to work with the unitarily equivalent operator
\begin{align*}
    i\calH=-U\calL U^{-1}=i\pmat{L&\rho_1^2\\-\rho_1^2&-L},\qquad U=\frac{1}{\sqrt{2}}\pmat{1&i\\1&-i}.
\end{align*}
We apply the Hille-Yosida theorem to $i\calH$. One has the resolvent identity 
    \[
    (i\calH-\lambda)^{-1}=R(\lambda)(\Id + VR(\lambda))^{-1}
    \]
    with
    \[
    R(\lambda)= \left[\begin{matrix}
        (iL-\lambda)^{-1} &0 \\0 & (-iL-\lambda)^{-1}
    \end{matrix}\right],\qquad V = \left[\begin{matrix}
        0&i\rho_1^2 \\  -i \rho_1^2 & 0
    \end{matrix}\right].
    \]   
    Thus, for all $\lambda>1$, 
    \[
    \| (i\calH-\lambda)^{-1}\| \le \lambda^{-1}\big(1-\|V\|_\infty/\lambda\big)^{-1}\le (\lambda-1)^{-1}
    \]
    and the Hille-Yosida theorem guarantees the existence of the group of bounded operators $e^{it\calH}$ as well as the exponential bound $\|e^{it\calH}\|\leq e^t$. Since $i\calH$ and $\calL$ are unitarily equivalent, this implies the desired result for $e^{t\calL}$.  
    
    The final statement of the lemma uses that $\calL=\bfJ\calS$, where 
    \begin{equation}
        \label{eq:JS}
            \calS=\left[\begin{matrix}
        L_2 &0 \\0 & L_1
    \end{matrix}\right],
    \end{equation}
    and amounts to the conservation of the Hamiltonian defined by $\calS$. 
\end{proof}

Next, we invert the Laplace transform and write the group $e^{t\calL}$ as a contour integral of the resolvent, cf.~\cite[Lemma 6.8]{KS}. This leads to an analogue of Stone's formula for this non-selfadjoint matrix operator and its distorted Fourier transform. But first, we note an anomaly at zero energy. 

\subsection{The embedded resonance at zero energy}\label{subsec:L2growth}

By inspection one has $\calL \psi=0$, for $\psi=\binom{0}{\sqrt{r}\rho_1}$. Note that  
$\calL^2 \tilde\psi=0$, $\tilde\psi= \binom{L_2^{-1}(\sqrt{r}\rho_1)}{0}$. In analogy with Jordan forms of matrices this could indicate a nilpotent component of $\calL$ leading to growth of $e^{t\calL}$ of at least~$t$. Since the evolution $e^{t\calL}$ is only defined on $L^2$ at this point, and since these kernels are not in~$L^2$ these considerations are only formal. The following lemma proves that this linear growth does indeed happen on $L^2$. 

\begin{lemma}
    \label{lem:nilpotent}
    One has $\| e^{t\calL}\|\gtrsim t$ for $t\ge1$, where $\| e^{t\calL}\|$ denotes the operator norm on $L^2_r(0,\infty)\times L^2_r(0,\infty)$.
\end{lemma}
\begin{proof}
    Let $\tilde\rho(r)=\sqrt{r}\, \rho_1(r)$ and $\psi_0=\binom{L_2^{-1}\tilde\rho}{0}$. The Green's function for $L_2$ is given by
    \[
    L_2^{-1}(r,s) = W[g_1,g_2]^{-1} \left\{ \begin{array}{cc}
       g_1(r) g_2(s)   & \text{\ \ for\ } r<s \\
         g_2(r) g_1(s)   & \text{\ \ for\ } r>s
    \end{array} \right.
    \]
    where $W[g_1,g_2]=g_1g_2'-g_1'g_2$ is our convention for the Wronskian. 
Next,   
\[
\psi(t,r) := \psi_0 + t\calL\psi_0 = \binom{L_2^{-1}\tilde\rho}{-t\tilde\rho}
\]
solves in the pointwise sense 
\[
\partial_t \psi = \calL\psi, \quad \psi(0)=\psi_0
\]
Let $\chi\in C^\infty([0,\infty))$ satisfy $\chi(u)=1$ if $0\le u\le1$ and $\chi(u)=0$ if $u\ge2$. Then set $\psi^{(N)}_0(r)=\psi_0(r)\chi(r/N)$ and $\psi^{(N)}(t):=e^{t\calL}\, \psi^{(N)}_0$ as defined by Lemma~\ref{lem:HYos}. The difference
\[
\delta\psi^{(N)}(t,r) := \psi^{(N)}(t,r) - \psi(t,r)\chi(r/N)
\]
satisfies the PDE
\[
\left\{ \begin{array}{ll}
  \partial_t \delta\psi^{(N)}(t,r)   &= \calL \delta\psi^{(N)}(t,r) -i\sigma_2 N^{-2} \psi(t,r) \chi''(r/N) \\
  &\qquad - 2i\sigma_2 N^{-1} \partial_r \psi(t,r) \chi'(r/N) \\
   \delta\psi^{(N)}(0,r)  & =0
\end{array}
\right. 
\]
On the one hand, for any $t\ge1$, 
\begin{align*}
    \| \psi^{(N)}_0 \|_{L^2_r} &\simeq N \\
    \| \psi^{(N)}(t) \|_{L^2_r} & \gtrsim tN - \| \delta\psi^{(N)}(t)\|_{L^2_r} 
\end{align*}
    and, on the other hand, by the exponential bound in Lemma~\ref{lem:HYos}
    \begin{align*}
   \delta\psi^{(N)}(t)   &=    -i \sigma_2\int_0^t  e^{(t-s)\calL} \big[\frac{2}{N}\partial_r\psi(s,\cdot)\chi'(\cdot/N) + \psi(s,\cdot) N^{-2}\chi''(\cdot/N)\big] \ud s \\
   \|\delta\psi^{(N)}(t)\|_{L^2_r}   &\les \int_0^t e^{t-s}  \big [ (1+s)N^{-1}\log N + (1+s) N/N^2\big]\ud s \\
   &\les  e^t N^{-1}\log N
    \end{align*}
By the preceding, for any $t\ge1$ we can take $N\ge1$ large enough such that 
\[
\| e^{t\calL}\, \psi^{(N)}_0\|_{L^2_r}\gtrsim t \|\psi^{(N)}_0\|_{L^2_r}
\]
whence the lemma. 
\end{proof}

\subsection{A Stone-type formula} \label{sec:stone}

Next, we derive a Stone-type representation formula for the evolution~$e^{t\calL}$.

\begin{lemma}
    \label{lem:6.8}
    For any $\phi,\psi\in \calD\times\calD$ and any $b>0$ one has 
    \begin{equation}
        \label{eq:evolphipsi}
         \langle e^{t\calL}\phi,\psi\rangle = \frac{1}{2\pi i} \int_{-\infty}^\infty e^{it\lambda} \langle [ e^{-bt} (i\calL-(\lambda +ib))^{-1} - e^{bt} (i\calL-(\lambda -ib))^{-1}] \phi,\psi\rangle\ud \lambda,
    \end{equation}
    where the indefinite integral converges.  
\end{lemma}
\begin{proof}
    The proof is essentially identical to~\cite[Lemma 6.8]{KS}, so we do not carry out all details. First, for $b>1$ the limit exists by an explicit calculation involving the Laplace transform
    \[
    (\calL - z)^{-1} = -\int_0^\infty e^{-zt} e^{t\calL}\ud t
    \]
    valid for $\Re z>1$. The contour can then be shifted to all $b>0$ by Lemma~\ref{lem:HYos} and a limiting absorption principle as in (6.9) of~\cite{KS}. Since some details are different, for completeness we provide a proof of the relevant limiting absorption estimate
\begin{align}\label{eq:LAP1}
\begin{split}
\sup_{\Im z>0,|\Re z|>\lambda_0}|z|^{\frac{1}{2}}\|\langle\cdot\rangle^{-\sigma}\big(i\calL-z\big)^{-1}f\|_{L^2_r(0,\infty)}\lesssim \|\langle\cdot\rangle^{\sigma}f\|_{L^2_r(0,\infty)},\qquad \sigma>\frac{1}{2},
\end{split}
\end{align}
provided $\lambda_0$ is sufficiently large. To simplify notation we will write $\|\cdot\|_{L^2}$ for $\|\cdot\|_{L^2_r(0,\infty)}$ in the remainder of this proof. 
We first prove a limiting absorption estimate for $i\calL_0$. More precisely, we show that for $\sigma>\frac{1}{2}$,
\begin{align}\label{eq:LAPcalL01}
\begin{split}
\sup_{\Im z>0,|\Re z|>1}|z|^{\frac{1}{2}}\|\langle\cdot\rangle^{-\sigma}\big(i\calL_0-z\big)^{-1}f\|_{L^2}\lesssim \|\langle\cdot\rangle^{\sigma}f\|_{L^2}.
\end{split}
\end{align}
In view of~\eqref{eq:calG0tocalGij1} and~\eqref{eq:tilsij1} it suffices to show that
\begin{align*}
\begin{split}
\sup_{\Im>0,|\Re z|>1}|z|^{\frac{1}{2}}\|\langle\cdot\rangle^{-\sigma}G_{ij}f\|_{L^2}\lesssim \|\langle\cdot\rangle^{\sigma}f\|_{L^2}, \qquad (i,j)\in\{(1,3),(2,4)\},
\end{split}
\end{align*}
where
\begin{align*}
\begin{split}
G_{ij}f(r,z):=\int_{0}^\infty \calG_{ij}(r,s,z)f(s)\ud s.
\end{split}
\end{align*}
We present the proof for $(i,j)=(1,3)$, the case $(i,j)=(2,4)$ being similar. Let
\begin{align*}
\begin{split}
\calG_{13}^1(r,s,z):=\rc\,\tiluppsi_1(r,z)\tiluppsi^t_3(s,z)\sigma_1\mathds{1}_{[0 < s \leq r]},\qquad \calG_{13}^2(r,s,z):= \rc\,\tiluppsi_3(r,z)\tiluppsi_1^t(s,z)\sigma_1\mathds{1}_{[ s > r]}.
\end{split}
\end{align*}
To use the appropriate asymptotics for $\tilp(k_j(z)\cdot)$, we divide into the regions where the input and output variables are less than or bigger than $|z|^{-\frac{1}{2}}$. That is, we define
\begin{align*}
\begin{split}
T_\ell^j f(r,z)=\int_0^\infty K_\ell^j(r,s,z)f(s)\ud s, \qquad \ell=1,2,3,\quad j=1,2,
\end{split}
\end{align*}
where
\begin{align*}
\begin{split}
&K_1^j(r,s,z)= \calG_{13}^j(r,s,z)\mathds{1}_{[s\leq r\leq |z|^{-\frac{1}{2}}]},\qquad K_2^j(r,s,z)= \calG_{13}^j(r,s,z)\mathds{1}_{[s\leq |z|^{-\frac{1}{2}}\leq r]},\\
&  K_3^j(r,s,z)= \calG_{13}^j(r,s,z)\mathds{1}_{[|z|^{-\frac{1}{2}}\leq s\leq r]}.
\end{split}
\end{align*}
It then suffices to prove that $\|\langle\cdot\rangle^{-\sigma} T_\ell^jf\|_{L^2}\lesssim \|\langle\cdot\rangle^{\sigma}f\|_{L^2}$ for $\ell=1,2,3$, $j=1,2$. From \cite[(9.2.1),(9.2.3)]{AS}, for large $\zeta$ with $\Im \zeta>0$, and for suitable nonzero constants $C_1$, $C_2$, and $C_3$, the modified Hankel and Bessel functions satisfy the asymptotics
\begin{align}\label{eq:tilhbigzeta1}
\begin{split}
\tilp_{+}(\zeta)\sim C_1e^{i\zeta},\qquad \tilq_{+}(\zeta)\simeq C_2e^{i\zeta} + C_3 e^{-i\zeta}.
\end{split}
\end{align}
It follows from \eqref{eq:tilhsmallzeta1} and~\eqref{eq:tilhbigzeta1} that
\begin{align*}  
&|K_3^1(r,s,z|)+|K_3^2(r,s,z)|\lesssim e^{-\Im k_1(z) |r-s|}\lesssim 1,  \\
&|K_2^1(r,s,z)|+|K_2^2(r,s,z)|\lesssim \mathds{1}_{ [ s\leq  |z|^{-\frac{1}{2}} \leq r] } (|z|^{\frac{1}{2}}s)^{\frac{3}{2}}e^{-\Im k1(z)r}+\mathds{1}_{ [ r\leq  |z|^{-\frac{1}{2}} \leq s] } (|z|^{\frac{1}{2}}r)^{\frac{3}{2}}e^{-\Im k1(z)s}\lesssim 1,\\
&|K_1^1(r,s,z)|+|K_1^2(r,s,z)|\lesssim \mathds{1}_{ [ s\leq r\leq |z|^{-\frac{1}{2}} ]} |z|^{\frac{1}{2}}s^{\frac{3}{2}}r^{-\frac{1}{2}}+\mathds{1}_{ [ r\leq s\leq |z|^{-\frac{1}{2}} ]} |z|^{\frac{1}{2}}r^{\frac{3}{2}}s^{-\frac{1}{2}}\lesssim 1.
\end{align*}
It follows that for $\ell=1,2,3$ and $j=1,2$,
\begin{align*}
\begin{split}
\|\langle\cdot\rangle^{-\sigma}T_\ell^jf\|_{L^2}^2 \lesssim \int_{0}^\infty \jap{r}^{-2\sigma}\Big(\int_{0}^\infty|f(s)|\ud s\Big)^2\ud r \lesssim \|\langle\cdot\rangle^\sigma f\|_{L^2},
\end{split}
\end{align*}
proving~\eqref{eq:LAPcalL01}.
To prove~\eqref{eq:LAP1} we introduce the the matrix potential $\calV_0$ defined by the relation
\begin{equation*}
    i\calL=: i\calL_0+\calV_0.
\end{equation*}
Since $|\calV_0(r)|\lesssim \langle r\rangle^{-2}$, by~\eqref{eq:LAPcalL01} if $\lambda_0$ is sufficiently large then
\begin{equation*}
    \sup_{\Im z>0, |\Re z|>\lambda_0}\|\calV_0(i\calL-z)^{-1}\|\leq\frac{1}{2}.
\end{equation*}
Here $\|\cdot\|$ denotes the operator norm for a map from $\langle\cdot\rangle^{-\sigma}L^2\times \langle\cdot\rangle^{-\sigma}L^2$ to $\langle\cdot\rangle^{\sigma}L^2\times \langle\cdot\rangle^{\sigma}L^2$. This proves the existence of $\big(I+\calV_0(i\calL_0-z)\big)^{-1}$ as a bounded map from $\langle\cdot\rangle^{\sigma}L^2\times \langle\cdot\rangle^{\sigma}L^2$ to $\langle\cdot\rangle^{-\sigma}L^2\times \langle\cdot\rangle^{-\sigma}L^2$. Estimate \eqref{eq:LAP1} is then an immediate consequence of the resolvent identity
\begin{equation*}
    (i\calL-z)^{-1}=(i\calL_0-z)^{-1}\big(I+\calV_0(i\calL_0-z)^{-1}\big)^{-1},
\end{equation*}
which holds as an identity of operators from $\langle\cdot\rangle^{\sigma}L^2\times \langle\cdot\rangle^{\sigma}L^2$ to $\langle\cdot\rangle^{-\sigma}L^2\times \langle\cdot\rangle^{-\sigma}L^2$.
\end{proof}

The next goal is to extract a distorted Fourier transform from Lemma~\ref{lem:6.8} by computing the resolvent kernel and passing to the limit $b\to0+$. We will do this for small energies $\lambda\in [-\delta_0,\delta_0]$ and prove that the limit $b\to0+$ can be moved inside. In particular, we define the localized evolution as follows. The existence of the limit in the following definition and the equality of the two lines in~\eqref{eq:localevol} below require careful justification. This will be done in Subsection~\ref{sec:thejump}.

\begin{defi} \label{def:PIevol}
    For any interval $I \subseteq [-\delta_0,\delta_0]$ set 
    \begin{equation}
    \label{eq:localevol}
    \begin{aligned}
        e^{t\calL} P_I &:= \lim_{b\to0+} \frac{1}{2\pi i} \int_I e^{it\lambda} \bigl[ e^{-bt} (i\calL-(\lambda +ib))^{-1} - e^{bt} (i\calL-(\lambda -ib))^{-1} \bigr]  \ud \lambda \\
        &=  \frac{1}{2\pi i} \int_I e^{it\lambda} \bigl[ (i\calL-(\lambda +i0^{+}))^{-1}- (i\calL-(\lambda -i0^{+}))^{-1} \bigr]  \ud \lambda,
    \end{aligned}
    \end{equation}
    where the limit exists in the weak sense of~\eqref{eq:evolphipsi} for $\phi,\psi \in L^2(\bbR_+) \times L^2(\bbR_+)$ with compact support.
\end{defi}

In Appendix~\ref{sec:PI} we discuss how to achieve this  localization via suitably defined spectral projection operators.

\section{Solving from $r=0$} \label{sec:Fs}

In this section we begin in earnest with the development of the distorted Fourier transform associated with the linearized operator $\calL$. Throughout the remainder of this paper we work with a string of fixed small absolute constants
\begin{equation} \label{equ:small_constants}
    0 < \delta_0 \ll \epsilon_\infty \ll \epsilon_0 \ll 1.
\end{equation}
The main goal of this section is to construct a fundamental system of matrix solutions to $i\calL F = z F$ for small energies $0 < |z| \leq \delta_0 \ll 1$ by perturbing around the zero-energy solutions \eqref{eq:F1F2}.

We start with the construction of the fundamental (matrix) solution $F_1(\cdot,z) \in L^2_{\mathrm{loc}}([0,\infty))$ to
\begin{equation} \label{eq:F1ODE}
    i\calL F_1(\cdot,z) = z F_1(\cdot,z),
\end{equation}
which is unique up to invertible linear combinations of the columns.
Let $\calG_1$ and $\calG_2$ be the Green's functions for $iL_1$, respectively $-iL_2$, i.e.,
\begin{equation}
\label{eq:calG1G2}
\begin{aligned}
 \calG_1\phi(r) &= \frac{1}{iW[f_1,f_2]}\int_0^r \bigl( f_1(r)f_2(s)-f_1(s)f_2(r) \bigr) \phi(s) \ud s,\\
 \calG_2\phi(r) &= \frac{1}{iW[g_1,g_2]} \biggl( \int_0^rg_1(s)g_2(r)\phi(s)\ud s + \int_r^{\epsilon_0|z|^{-1}}g_1(r)g_2(s)\phi(s)\ud s \biggr).
\end{aligned}
\end{equation}
More precisely, $\calG_2=\calG_2(z)$ is the local Green's function of $L_2$ on $(0,\epsilon_0|z|^{-1}]$. We work with this local Green's function rather than the global one since it avoids factors of $j!$ in the series expansion~\eqref{eq:F1series1} below.
Note that
\begin{equation}
\label{eq:L1G1L2G2}
    iL_1\calG_1\phi= -iL_2\calG_2\phi=\phi.
\end{equation}
We seek to derive a convergent power series representation for $F_1(r,z)$ in the interval $r\in [0,\epsilon_0/ |z|]$ for small $0 < |z| \leq \delta_0 \ll 1$.
For this purpose, we will more precise asymptotics for $f_j(r)$ and $g_j(r)$, $1 \leq j \leq 2$. For small $r$ such asymptotics were given in~\eqref{eq:f1f2}, \eqref{eq:g1g2}. The next lemma provides the necessary asymptotics for large $r$.
\begin{lemma}\label{lem:fjgj_1}
$f_1$, $f_2$, $g_1$ and $g_2$ satisfy the following expansions for large $r$:
\begin{align*}
\begin{split}
&f_1(r)= r^{\frac{1}{2}}(1-\frac{1}{2r^2}-\frac{9}{8r^4}+\calO(r^{-6})),\qquad f_2(r)= r^{\frac{1}{2}}\log r (1+\calO(r^{-2})),\\
&g_1(r)=ce^{\sqrt{2}r}(1+\calO(r^{-1})),\qquad g_2(r)=e^{-\sqrt{2}r}(1+\calO(r^{-1}))
\end{split}
\end{align*}
and
\begin{align*}
\begin{split}
&f_1'(r)=\frac{1}{2}r^{-\frac{1}{2}}(1+\calO(r^{-2})),\qquad f_2'(r)=\frac{1}{2}r^{-\frac{1}{2}}\log r + r^{-\frac{1}{2}}+\calO(r^{-\frac{5}{2}}\log r)\\
&g_1'(r)=c\sqrt{2}e^{\sqrt{2}r}(1+\calO(r^{-1}),\qquad g_2'(r)=-\sqrt{2}e^{\sqrt{2}r}(1+\calO(r^{-1})).
\end{split}
\end{align*}
\end{lemma}
\begin{proof}
For $f_1$ this follows from the asymptotics of $\rho_1$ which can be found for instance in~\cite[Theorem 1.1]{CEQ94}. For $f_2$ we write $f_2(r)=r^{\frac{1}{2}}x(r)$ and notice that $x$ satisfies
\begin{align*}
\begin{split}
x''+\frac{1}{r}x'+Vx=0
\end{split}
\end{align*}
where $V=\calO(r^{-4})$ for large $r$. The desired estimate follows from the leading order bound $|x(r)|\lesssim \log r$ and the Volterra integral representation
\begin{align*}
\begin{split}
x(r)=c\log r -\int_r^\infty (\log s-\log r)s V(s) x(s) \ud s.
\end{split}
\end{align*}
The derivative bound for $f_2$ is obtained by differentiating this relation. Similarly, the estimates for $g_1$, $g_2$ and their derivatives follow from plugging in the leading order asymptotics in the Volterra and Lyapunov-Perron integral representation formulas, respectively.
\end{proof}

\subsection{$L^2$ solution near $r=0$}\label{sec:F1}

We can now formulate the main result on the construction of the $L^2$ solution branch near $r=0$. 
\begin{prop} \label{lem:F1_1}
Let $0 < |z| \leq \delta_0$. There exists a solution $F_1(\cdot,z)\in L^2_{\mathrm{loc}}([0,\infty))$ to~\eqref{eq:F1ODE} given by
\begin{align}\label{eq:F1series1}
\begin{split}
F_1(\cdot,z) := \pmat{z\sum_{j=0}^\infty z^{2j}\calG_2(\calG_1\calG_2)^jf_1&\sum_{j=0}^\infty z^{2j}(\calG_2\calG_1)^jg_1\\\sum_{j=0}^\infty z^{2j}(\calG_1\calG_2)^jf_1&z\sum_{j=0}^\infty z^{2j}\calG_1(\calG_2\calG_1)^jg_1 },
\end{split}
\end{align}
where the sums are absolutely and uniformly convergent for $r\in[0,r_0(z)]$ with 
\begin{equation}
    r_0(z) := \frac{\epsilon_0}{|z|}.
\end{equation}
There exist constants $C, C_0, c>0$ (independent of $\epsilon_0$ and $|z|$) such that for all integers $j\ge0$ 
\begin{align}
|(\calG_1\calG_2)^jf_1(r)|+|\calG_2(\calG_1\calG_2)^jf_1(r)| &\leq C_0 (C r_0(z))^{2j} \jap{r}^{\frac{1}{2}}, \quad \quad \quad r\in[0,r_0(z)], \label{eq:F1bound1}\\
|(\calG_2\calG_1)^jg_1(r)| + |\calG_1(\calG_2\calG_1)^jg_1(r)| 
&\leq C_0 (C r_0(z))^{j} e^{\sqrt{2}r}, \, \, \quad \quad \quad r\in[0,r_0(z)],\label{eq:F1bound2}\\
|\partial_r(\calG_1\calG_2)^jf_1(r)|+|\partial_r\calG_2(\calG_1\calG_2)^jf_1(r)|  &\leq C_0 (C r_0(z))^{2j}  \jap{r}^{-\frac{1}{2}}, \, \, \, \qquad r\in[1,r_0(z)/2],\label{eq:F1bound3}
\\
 |\partial_r(\calG_2\calG_1)^jg_1(r)| + |\partial_r\calG_1(\calG_2\calG_1)^jg_1(r)| &\leq C_0  (C r_0(z))^{j} e^{\sqrt{2}r}, \, \, \, \quad \qquad r\in[1,r_0(z)/2]. \label{eq:F1bound4}
\end{align}
Moreover, for $j=0,1$, and $A \gg 1$ sufficiently large (independent of $\epsilon_0$ and $|z|$), we have
\begin{align}
|(r\partial_r)^jf_1(r)|, |(r\partial_r)^j\calG_2f_1(r)| &\geq c r^{\frac{1}{2}}, &r\in[A,r_0(z)/2],\label{eq:F1bound9}\\
|\partial_r^jg_1(r)|, |\partial_r^j\calG_1 g_1(r)| &\geq c e^{\sqrt{2}r}, &r\in[A,r_0(z)/2].\label{eq:F1bound10}
\end{align}
Finally, $F_1(\cdot,z)$ is the unique non-trivial solution in $L^2_{\mathrm{loc}}([0,\infty))$ up to multiplying from the right by some $B(z) \in GL(2,\Cp)$. 
\end{prop}

Before we turn to the proof of Proposition~\ref{lem:F1_1}, we introduce some notation that will be useful later on. Let $\upfy_1$ and $\upfy_2$ denote the columns of $F_1$, that is,
\begin{align*}
\begin{split}
F_1=\pmat{\upfy_1&\upfy_2}.
\end{split}
\end{align*}
Then in view of Proposition~\ref{lem:F1_1}, upon defining
\begin{equation} \label{equ:definition_f_and_g}
\begin{split}
    f(r,z) := \sum_{j=0}^\infty z^{2j} (\calG_1\calG_2)^j f_1(r),\qquad g(r,z) := \sum_{j=0}^\infty z^{2j} (\calG_2\calG_1)^j g_1(r),
\end{split}
\end{equation}
we can write
\begin{equation} \label{equ:definition_upfy_one_and_two}
\begin{split}
    \upfy_1=\pmat{z\calG_2f\\f},\quad \upfy_2=\pmat{g\\z\calG_1g}.
\end{split}
\end{equation}
Note that by inspection $f(r,z)$ and $g(r,z)$ are even functions of $z$, while $z\calG_2f(r,z)$ and $z\calG_1g(r,z)$ are odd functions of $z$. Therefore,
\begin{equation} \label{eq:upfy12evenodd1}
\begin{split}
    \upfy_1(r,-z)=\sigma_3\upfy_1(r,z),\qquad \upfy_2(r,-z)=-\sigma_3\upfy_2(r,z).
\end{split}
\end{equation}
In view of the estimates \eqref{eq:F1bound1}--\eqref{eq:F1bound10}, we invite the reader to keep the following decompositions in mind
\begin{equation}
    f(r,z) = f_1(r) + \bigl\{\text{remainder}\bigr\}, \quad g(r,z) = g_1(r) + \bigl\{\text{remainder}\bigr\}.
\end{equation}

\begin{proof}
The formulas for the columns of $F_1(r,z)$ are obtained by making a power series ansatz. Indeed, formally applying $i\calL$ to the series representation \eqref{eq:F1series1} we see that $$i\calL F_1(r,z)=zF_1(r,z),$$ so we only need to consider the convergence and size estimates. Starting from the upper bound 
\begin{align*}
\begin{split}
|f_1(r)|\leq C r^{\frac{1}{2}},
\end{split}
\end{align*}
to prove \eqref{eq:F1bound1}  we assume inductively that 
\begin{align*}
\begin{split}
|(\calG_1\calG_2)^jf_1(r)|\leq C^{2j}|z|^{-2j}\epsilon_0^{2j}\jap{r}^{\frac{1}{2}}.
\end{split}
\end{align*}
We will now prove that
\begin{align*}
\begin{split}
|\calG_2(\calG_1\calG_2)^jf_1(r)| &\leq C^{2j+1}|z|^{-2j}\epsilon_0^{2j}\jap{r}^{\frac{1}{2}} \\
|(\calG_1\calG_2)^{j+1}f_1(r)| &\leq C^{2(j+1)}|z|^{-2(j+1)}\epsilon_0^{2(j+1)}\jap{r}^{\frac{1}{2}}.
\end{split}
\end{align*}
First, for $r\geq 1$
\begin{align*}
\begin{split}
|\calG_2(\calG_1\calG_2)^jf_1(r)| &\lesssim  C^{2j}|z|^{-2j}\epsilon_0^{2j}\Big(e^{-\sqrt{2}r}\int_0^re^{\sqrt{2}s}\jap{s}^{\frac{1}{2}}\ud s+ e^{\sqrt{2}r}\int_r^{\epsilon_0|z|^{-1}}e^{-\sqrt{2}s}s^{\frac{1}{2}}\ud s\Big)\\
&\lesssim C^{2j}|z|^{-2j}\epsilon_0^{2j}\jap{r}^{\frac{1}{2}},
\end{split}
\end{align*}
which is bounded by $C^{2j+1}|z|^{-2j}\epsilon_0^{2j}\jap{r}^{\frac{1}{2}}$ if $C$ is sufficiently large. For $0<r\leq 1$,
\begin{align*}
\begin{split}
|\calG_2(\calG_1\calG_2)^jf_1(r)| &\lesssim C^{2j}|z|^{-2j}\epsilon_0^{2j} \Big(r^{-\frac{1}{2}}\int_0^r|g_1(s)|\ud s+ r^{\frac{3}{2}}\int_r^{\epsilon_0|z|^{-1}}|g_2(s)|\jap{s}^{\frac{1}{2}}\ud s\Big)\\
&\lesssim C^{2j}|z|^{-2j}\epsilon_0^{2j},
\end{split}
\end{align*}
which is again bounded by $C^{2j+1}|z|^{-2j}\epsilon_0^{2j}$. This establishes the upper bound on $\calG_2(\calG_1\calG_2)^jf_1$. Turning to $(\calG_1\calG_2)^{j+1}f_1=\calG_1\big[\calG_2(\calG_1\calG_2)^{j}f_1\big]$, we use the relation
\begin{align}\label{eq:logint1}
\begin{split}
\int x^p\log x \,\ud x = \frac{1}{p+1}x^{p+1}\log x -\frac{1}{(p+1)^2}x^{p+1} + \mathrm{const},
\end{split}
\end{align}
and Lemma~\ref{lem:fjgj_1} to conclude that for $r\geq 1$ 
\begin{align}\label{eq:remtemp1}
\begin{split}
|(\calG_1\calG_2)^{j+1}f_1(r)|&\lesssim C^{2j+1}(r_0(z))^{2j}\Big(r^{\frac{1}{2}}\log r+ \int_1^r (\log r - \log s)(rs)^{\frac{1}{2}}s^{\frac{1}{2}}\ud s +r^{\frac{1}{2}}\log r \int_1^r\ud s\Big) \\
&\lesssim C^{2j+1}(r_0(z))^{2j}\jap{r}^{\frac{5}{2}}.
\end{split}
\end{align}
Here the last term on the first line comes from the error in replacing $f_1$ and $f_2$ by their leading order expressions from Lemma~\ref{lem:fjgj_1}. Notice that if $1\leq r\leq \epsilon_0|z|^{-1}$ then the right-hand side is bounded by~$(C\epsilon_0/|z|)^{2(j+1)}\jap{r}^{\frac{1}{2}}$ as desired. For $0<r\leq 1$,
\begin{align*}
\begin{split}
|(\calG_1\calG_2)^{j+1}f_1(r)|\lesssim C^{2j+1}(r_0(z))^{2j}\int_0^r (r^{\frac{3}{2}}s^{-\frac{1}{2}}+s^{\frac{3}{2}}r^{-\frac{1}{2}})\ud s\lesssim C^{2j+1}(r_0(z))^{2j},
\end{split}
\end{align*}
which is again bounded by $(C\epsilon_0/|z|)^{2(j+1)}$. This completes the estimate of~\eqref{eq:F1bound1}. For the derivative bounds \eqref{eq:F1bound3} for $r\geq1$ we differentiate the integral representation formulas and use Lemma~\ref{lem:fjgj_1}. More precisely, when $j=0$ we know that $(\partial_rf_1)(r)$ satisfies the desired bound. By induction we assume that 
\begin{align*}
\begin{split}
|\partial_r(\calG_1\calG_2)^jf_1(r)|\leq \tilC^{2j}|z|^{-2j}\epsilon_0^{2j}\jap{r}^{-\frac{1}{2}},\qquad r\in[1,r_0(z)/2],
\end{split}
\end{align*}
and prove that for $r\in[1,r_0(z)/2]$
\begin{align}\label{eq:F1dertemp0}
\begin{split}
|\partial_r\calG_2(\calG_1\calG_2)^jf_1(r)| &\leq \tilC^{2j+1}|z|^{-2j}\epsilon_0^{2j}\jap{r}^{-\frac{1}{2}}, \\
|\partial_r(\calG_1\calG_2)^{j+1}f_1(r)| &\leq \tilC^{2(j+1)}|z|^{-2(j+1)}\epsilon_0^{2(j+1)}\jap{r}^{-\frac{1}{2}}.
\end{split}
\end{align}
For $\partial_r\calG_2(\calG_1\calG_2)^jf_1(r)$ we need to estimate
\begin{align*}
\begin{split}
&g_2'(r)\int_0^1g_1(s)(\calG_1\calG_2)^{j}f_1(s)\ud s+g_2'(r)\int_1^rg_1(s)(\calG_1\calG_2)^{j}f_1(s)\ud s+g_1'(r)\int_r^{\epsilon_0 |z|^{-1}}g_2(s)(\calG_1\calG_2)^{j}f_1(s)\ud s.
\end{split}
\end{align*}
The first term satisfies better bounds than we need. For the second two integrals we use Lemma~\ref{lem:fjgj_1} to write them as (up to an overall constant)
\begin{align}\label{eq:F1dertemp1}
\begin{split}
&-\sqrt{2}\int_1^r e^{\sqrt{2}(s-r)}(\calG_1\calG_2)^{j}f_1(s)\ud s+\sqrt{2}\int_r^{\epsilon_0 |z|^{-1}}e^{\sqrt{2}(r-s)}(\calG_1\calG_2)^{j}\ud s+E\\
&=-\int_1^r\frac{\ud}{\ud s}e^{\sqrt{2}(s-r)}(\calG_1\calG_2)^{j}f_1(s)\ud s -\int_r^{\epsilon_0 |z|^{-1}}\frac{\ud}{\ud s}e^{\sqrt{2}(r-s)}(\calG_1\calG_2)^{j}f_1(s)\ud s +E\\
&=\int_1^re^{\sqrt{2}(s-r)}\frac{\ud}{\ud s}(\calG_1\calG_2)^{j}f_1(s)\ud s +\int_r^{\epsilon_0 |z|^{-1}}e^{\sqrt{2}(r-s)}\frac{\ud}{\ud s}(\calG_1\calG_2)^{j}f_1(s)\ud s\\
&\quad+e^{\sqrt{2}(1-r)}(\calG_1\calG_2)^{j}f_1(1)-e^{\sqrt{2}(r-\epsilon_0|z|^{-1})}(\calG_1\calG_2)^{j}f_1(\epsilon_0|z|^{-1})+E,
\end{split}
\end{align}
where the error $E$, which comes from replacing $f_1$ and $f_2$ by their leading order terms in Lemma~\ref{lem:fjgj_1}, satisfies
\begin{align}\label{eq:F1dertemp2}
\begin{split}
|E| &\lesssim (Cr_0(z))^{2j}\Big(e^{-\sqrt{2}r}\int_1^r e^{\sqrt{2}s}s^{-\frac{1}{2}}\ud s+e^{\sqrt{2}r}r^{-1}\int_r^{\epsilon_0|z|^{-1}}e^{-\sqrt{2}s}s^{\frac{1}{2}}\ud s\Big)\\
&\lesssim (Cr_0(z))^{2j} r^{-\frac{1}{2}}.
\end{split}
\end{align}
Since $1\leq r\leq \frac{1}{2} \epsilon_0|z|^{-1}$, the first estimate in \eqref{eq:F1dertemp0} follows from \eqref{eq:F1dertemp1}, \eqref{eq:F1dertemp2} and the induction hypothesis. Similarly, for $\partial_r(\calG_1\calG_2)^{j+1}f_1(r)$ we need to estimate
\begin{align*}
\begin{split}
&\int_0^1(f_1(s)f_2'(r)-f_1'(r)f_2(s))\calG_2(\calG_1\calG_2)^jf_1(s)\ud s  +\int_1^r(f_1(s)f_2'(r)-f_1'(r)f_2(s))\calG_2(\calG_1\calG_2)^jf_1(s)\ud s.
\end{split}
\end{align*}
The first integral satisfies the desired bound, while, by Lemma~\ref{lem:fjgj_1}, the second integral can be written as (up to an overall constant)
\begin{align*}
\begin{split}
&r^{-\frac{1}{2}}\int_1^r(\log r-\log s)s^{\frac{1}{2}}\calG_2(\calG_1\calG_2)^jf_1(s)\ud s+E,
\end{split}
\end{align*}
where
\begin{align*}
\begin{split}
|E|\lesssim r^{-\frac{1}{2}}\int_1^rs^{\frac{1}{2}}|\calG_2(\calG_1\calG_2)^jf_1(s)|\ud s.
\end{split}
\end{align*}
The desired estimate \eqref{eq:F1bound3} for $1\leq r\leq \frac{1}{2}\epsilon_0 |z|^{-1}$ then follows using \eqref{eq:logint1} as above, completing the proof of \eqref{eq:F1bound3}.

Turning to estimate \eqref{eq:F1bound2} for the second column of $F_1(r,z)$ in \eqref{eq:F1series1}, by induction we assume that 
\begin{align*}
\begin{split}
|(\calG_2\calG_1)^jg_1(r)|\leq C^{2j}|z|^{-j}\epsilon_0^{j}e^{\sqrt{2}r},
\end{split}
\end{align*}
and prove that
\begin{align*}
\begin{split}
|\calG_1(\calG_2\calG_1)^jg_1(r)| &\leq C^{2j+1}|z|^{-j}\epsilon_0^{j}e^{\sqrt{2}r},\\
|(\calG_2\calG_1)^{j+1}g_1(r)| &\leq C^{2(j+1)}|z|^{-j-1}\epsilon_0^{j+1}e^{\sqrt{2}r}.
\end{split}
\end{align*}
For $r\leq 1$
\begin{align*}
\begin{split}
|\calG_1(\calG_2\calG_1)^jg_1(r)|\lesssim C^{2j}r_0(z)^{j} \int_0^r(|f_1(s)f_2(r)|+|f_1(r)f_2(s)|)\ud s\lesssim C^{2j}r_0(z)^{j}.
\end{split}
\end{align*}
For $r\geq 1$, by Lemma~\ref{lem:fjgj_1},
\begin{align*}
\begin{split}
|\calG_1(\calG_2\calG_1)^jg_1(r)|&\lesssim C^{2j}r_0(z)^{j}  e^{\sqrt{2}r}+ C^{2j}r_0(z)^{j}\int_{\frac{r}{2}}^r(\log r-\log s)(sr)^{\frac{1}{2}}e^{\sqrt{2}s}\ud s\\
&\quad+C^{2j}r_0(z)^{j}r^{\frac{1}{2}}\log r\int_{\frac{r}{2}}^r  s^{-\frac{3}{2}}e^{\sqrt{2}s}\ud s.
\end{split}
\end{align*}
The first integral is bounded, after one integration by parts, by a multiple of
\begin{align*}
\begin{split}
& C^{2j}r_0(z)^{j} e^{\sqrt{2}r} + C^{2j}r_0(z)^{j}\int_{\frac{r}{2}}^r s^{-\frac{1}{2}}r^{\frac{1}{2}}(1+\log(r/s))e^{\sqrt{2}s}\ud s  \lesssim C^{2j}r_0(z)^{j}e^{\sqrt{2}r}.
\end{split}
\end{align*}
Similarly, the second integral is bounded by a multiple of $$C^{2j}r_0(z)^{j}(r^{-1}\log r) e^{\sqrt{2}r}.$$ This proves the desired estimate \eqref{eq:F1bound2} for $\calG_1(\calG_2\calG_1)^jg_1$. For 
\[(\calG_2\calG_1)^{j+1}g_1=\calG_2[\calG_1(\calG_2\calG_1)^jg_1]\] with $r\geq 1$ we have
\begin{align*}
\begin{split}
 |(\calG_2\calG_1)^{j+1}g_1(r)| &\lesssim C^{2j+1}\,r_0(z)^{j} \Big(e^{-\sqrt{2}r}\int_0^re^{2\sqrt{2}s}\ud s+e^{\sqrt{2}r}\int_r^{\epsilon_0|z|^{-1}}\ud s\Big)\\
 & \lesssim C^{2j+1} \,r_0(z)^{j+1}e^{\sqrt{2}r},
\end{split}
\end{align*}
as desired. For $r\leq 1$,
\begin{align*}
\begin{split}
|(\calG_2\calG_1)^{j+1}g_1(r)| &\lesssim C^{2j+1}r_0(z)^{j}\Big( r^{-\frac{1}{2}}\int_0^rs^{\frac{3}{2}}\ud s+r^{\frac{3}{2}}\int_r^{\epsilon_0|z|^{-1}}|g_2(s)|e^{\sqrt{2}s}\ud s\Big)\\
& \lesssim C^{2j+1}\, r_0(z)^{j+1}.
\end{split}
\end{align*}
This completes the proof of the induction for \eqref{eq:F1bound2}. The proof of the derivative estimates is similar to the proofs of the estimates we have already established. 

It remains to prove the lower bounds \eqref{eq:F1bound9} and \eqref{eq:F1bound10}. The estimates for $(r\partial_r)^jf_1$ and $\partial_r^jg_1$ follow from Lemma~\ref{lem:fjgj_1}. For $\calG_2f_1$, again by Lemma~\ref{lem:fjgj_1}, there are constants $c_\ell>0$ such that 
\begin{align*}
\begin{split}
|\calG_2f_1(r)| &\geq c_1e^{-\sqrt{2}r}\int_{\frac{r}{2}}^{r}e^{\sqrt{2}s} s^{\frac{1}{2}}\ud s+c_2e^{\sqrt{2}r}\int_r^{\epsilon_0|z|^{-1}}e^{-\sqrt{2}s}s^{\frac{1}{2}}\ud s - c_3r^{\frac{1}{2}}e^{-\frac{\sqrt{2}}{2}r} \geq c_4 r^{\frac{1}{2}},
\end{split}
\end{align*}
if $A\leq r\leq \frac{1}{2}\epsilon_0|z|^{-1}$ and $A$ is sufficiently large. Similarly,
\begin{align*}
\begin{split}
\bigl|r\partial_r\calG_2f_1(r)\bigr| &\geq c_5r\Bigg|\int_{\frac{r}{2}}^r\frac{\ud}{\ud s}e^{\sqrt{2}(s-r)}f_1(s)\ud s+\int_r^{\epsilon_0|z|^{-1}}\frac{\ud}{\ud s}e^{\sqrt{2}(r-s)}f_1(s)\ud s\Bigg|\\
&\quad -c_6r^{\frac{3}{2}}e^{-\frac{\sqrt{2}}{2}r}-c_7r^{-\frac{1}{2}} \biggl(\int_{\frac{r}{2}}^re^{\sqrt{2}(s-r)}\ud s+\int_r^{\epsilon_0|z|^{-1}}e^{\sqrt{2}(r-s)}\ud s \biggr) \\ 
&\geq c_8 r^{\frac{1}{2}},
\end{split}
\end{align*}
for $A\leq r\leq \frac{1}{2}\epsilon_0|z|^{-1}$. This completes the proof of \eqref{eq:F1bound9}. The proof of \eqref{eq:F1bound10} is similar.
\end{proof}

\subsection{The non-$L^2$ solutions near $r=0$}\label{sec:F2}

Next, we construct another linearly independent matrix solution $F_2(\cdot,z)$ to $i\calL F = z F$ by perturbing around $F_2(r)$ from \eqref{eq:F1F2}. This requires modifying the Green's function $\calG_1$ in~\eqref{eq:calG1G2} as follows,
\begin{equation}
\label{eq:calG1G2*}
\begin{aligned}
&\calG_1^\dagger\phi(r)=\frac{-1}{iW[f_1,f_2]}\int_r^{\infty}\big(f_1(r)f_2(s)-f_1(s)f_2(r)\big)\phi(s)\ud s,\\
&\wt\calG_1 \phi(r)=\frac{1}{iW[f_1,f_2]}\int_1^{r}\big(f_1(r)f_2(s) -f_1(s)f_2(r)\big)\phi(s)\ud s.
\end{aligned}
\end{equation}
Then \eqref{eq:L1G1L2G2} continues to hold.

\begin{prop}\label{lem:F2_1}
Let $0 < |z| \leq \delta_0$. There exists a solution $F_2(\cdot,z)$ to~\eqref{eq:F1ODE} given by
\begin{align}\label{eq:F2series1}
\begin{split}
F_2(\cdot,z) := \pmat{z\sum_{j=0}^\infty z^{2j} \calG_2(\wt\calG_1\calG_2)^jf_2&\sum_{j=0}^\infty z^{2j}(\calG_2\calG^\dagger_1)^jg_2\\ \sum_{j=0}^\infty z^{2j}(\wt\calG_1\calG_2)^jf_2&z\sum_{j=0}^\infty z^{2j}\calG^\dagger_1(\calG_2\calG^\dagger_1)^jg_2},
\end{split}
\end{align}
where the sums are absolutely convergent for $r \in (0,r_0(z)]$ with 
\begin{equation}
    r_0(z) := \frac{\epsilon_0}{|z|}.
\end{equation}
There exist constants $C, C_0, c > 0$ (independent of $\epsilon_0$ and $|z|$) such that for all $j \geq 0$ 
\begin{align}
|(\wt\calG_1\calG_2)^jf_2(r)|+|\calG_2(\wt\calG_1\calG_2)^jf_2(r)| &\leq C_0 (C r_0(z))^{2j} r^{\frac{1}{2}}\log r, \qquad r\in[2,r_0(z)], \label{eq:F2bound1}\\
|(\calG_2\calG^\dagger_1)^jg_2(r)| + |\calG^\dagger_1(\calG_2\calG^\dagger_1)^jg_2(r)| 
&\leq C_0 (C r_0(z))^{j} e^{-\sqrt{2}r}, \qquad \quad r\in[2,r_0(z)],\label{eq:F2bound2}\\
|\partial_r(\wt\calG_1\calG_2)^jf_2(r)|+|\partial_r\calG_2(\wt\calG_1\calG_2)^jf_2(r)|  &\leq C_0 (C r_0(z))^{2j}  r^{-\frac{1}{2}}\log r, \quad \, \, \, r\in[2,r_0(z)/2],\label{eq:F2bound3}
\\
|\partial_r(\calG_2\calG^\dagger_1)^jg_2(r)| + |\partial_r\calG^\dagger_1(\calG_2\calG^\dagger_1)^jg_2(r)| &\leq C_0 (C r_0(z))^{j} e^{-\sqrt{2}r} ,\qquad \quad \, r\in[2,r_0(z)/2]. \label{eq:F2bound4}
\end{align}
Moreover, for $j=0,1$, and $A \gg 1$ sufficiently large (independent of $\epsilon_0$ and $|z|$),
\begin{align}
|(r\partial_r)^jf_2(r)|, |(r\partial_r)^j\calG_2f_2(r)| &\geq c r^{\frac{1}{2}} \log r, \qquad r\in[A,r_0(z)/2], \label{eq:F2bound9} \\ 
|\partial_r^jg_2(r)|, |\partial_r^j\calG^\dagger_1 g_2(r)| &\geq c  e^{-\sqrt{2}r}, \qquad \, \, \, \, r \in [A,r_0(z)/2]. \label{eq:F2bound10}
\end{align}
\end{prop}
\begin{remark}
    The proof of Proposition~\ref{lem:F2_1} also gives bounds on the entries of $F_2(r,z)$ for $r\leq 2$. See~\eqref{eq:F2temp1}, \eqref{eq:F2temp2}, \eqref{eq:F2temp3}, and \eqref{eq:F2temp4}. But the large $r$ estimates in the statement of the proposition are the only ones we will need going forward.
\end{remark}
For future reference we introduce the notation $\upfy_3$ and $\upfy_4$ to denote the columns of $F_2$, that is,
\begin{align*}
\begin{split}
F_2=\pmat{\upfy_3&\upfy_4}.
\end{split}
\end{align*}
Then in view of Proposition~\ref{lem:F1_1}, defining
\begin{align*}
\begin{split}
\tilf=\sum_{j=0}^\infty z^{2j}(\wt\calG_1\calG_2)^jf_2, \qquad \tilg=\sum_{j=0}^\infty z^{2j}(\calG_2\calG^\dagger_1)^jg_2,
\end{split}
\end{align*}
we can write
\begin{align*}
\begin{split}
\upfy_3=\pmat{z\calG_2\tilf\\\tilf},\qquad \upfy_4=\pmat{\tilg\\z\calG_1^\dagger\tilg}.
\end{split}
\end{align*}
Note that by inspection $\tilf(r,z)$ and $\tilg(r,z)$ are even functions of $z$, while $z\calG_2\tilf(r,z)$ and $z\calG_1^\dagger\tilg(r,z)$ are odd functions of $z$. Therefore,
\begin{equation} \label{eq:upfy34evenodd1}
\begin{split}
    \upfy_3(r,-z)=\sigma_3\upfy_3(r,z),\qquad  \upfy_4(r,-z)=-\sigma_3 \upfy_4(r,z).
\end{split}
\end{equation}

\begin{proof}
The proof is similar to that of Proposition~\ref{lem:F1_1} so we will be brief. The fact that $F_2$ is a solution follows formally by differentiating term by term, so we only need to prove the claimed estimates. Assume inductively that
\begin{align*}
\begin{split}
|(\wt\calG_1\calG_2)^jf_2(r)|\leq \begin{cases} (Cr_0)^{2j}r^{-\frac{1}{2}},\quad &0\leq r\leq 1\\ (Cr_0)^{2j}r^{\frac{1}{2}}\log r,\quad& 1\leq r\leq r_0\end{cases}.
\end{split}
\end{align*}
We will prove that
\begin{align}\label{eq:F2temp1}
\begin{split}
|\calG_2(\wt\calG_1\calG_2)^jf_2(r)|\leq \begin{cases} C(Cr_0)^{2j}r^{\frac{3}{2}}\log r,\quad &0\leq r\leq 1\\ C(Cr_0)^{2j}r^{\frac{1}{2}}\log r,\quad& 1\leq r\leq r_0\end{cases},
\end{split}
\end{align}
and
\begin{align}\label{eq:F2temp2}
\begin{split}
|(\wt\calG_1\calG_2)^{j+1}f_2(r)|\leq \begin{cases} (Cr_0)^{2j+2}r^{-\frac{1}{2}},\quad &0\leq r\leq 1\\ (Cr_0)^{2j+2}r^{\frac{1}{2}}\log r,\quad& 1\leq r\leq r_0\end{cases}.
\end{split}
\end{align}
For $r\leq 1$ by the induction assumption
\begin{align*}
\begin{split}
|\calG_2(\wt\calG_1\calG_2)^jf_2(r)|&\lesssim (Cr_0)^{2j}\int_0^r r^{-\frac{1}{2}}s\ud s+(Cr_0)^{2j}\int_r^1r^{\frac{3}{2}}s^{-1}\ud s+(Cr_0)^{2j}r^{\frac{3}{2}}\int_1^{r_0}e^{-\sqrt{2}s}s^{\frac{1}{2}}\log s\ud s\\
&\lesssim (Cr_0)^{2j}r^{\frac{3}{2}}\log r.
\end{split}
\end{align*}
For $1\leq r\leq r_0$
\begin{align*}
\begin{split}
|\calG_2(\wt\calG_1\calG_2)^jf_2(r)|&\lesssim (Cr_0)^{2j}\int_0^1e^{-\sqrt{2}r}s\ud s+(Cr_0)^{2j}\int_1^re^{-\sqrt{2}(r-s)}s^{\frac{1}{2}}\log s\ud s\\
&\quad+(Cr_0)^{2j}\int_r^{r_0}e^{-\sqrt{2}(s-r)}s^{\frac{1}{2}}\log s\ud s\\
&\lesssim (Cr_0)^{2j}r^{\frac{1}{2}}\log r.
\end{split}
\end{align*}
This proves \eqref{eq:F2temp1}.  For \eqref{eq:F2temp2}, using \eqref{eq:F2temp1} for $r\leq 1$ we have
\begin{align*}
\begin{split}
|(\wt\calG_1\calG_2)^{j+1}f_2(r)|\lesssim C(Cr_0)^{2j}r^{\frac{3}{2}}\int_r^1s\log s\ud s+C(Cr_0)^{2j}r^{-\frac{1}{2}}\int_r^1s^3\log s\ud s\lesssim C(Cr_0)^{2j}r^{-\frac{1}{2}}.
\end{split}
\end{align*}
For $1\leq r\leq r_0$
\begin{align*}
\begin{split}
|(\wt\calG_1\calG_2)^{j+1}f_2(r)|&\lesssim C(Cr_0)^{2j}\int_1^r(rs)^{\frac{1}{2}}\log (\frac{r}{s})s^{\frac{1}{2}}\log s\ud s+C(Cr_0)^{2j}  r^{\frac{1}{2}}\log r\int_1^r s^{-1}\log s \ud s\\
&\lesssim Cr_0^2(Cr_0)^{2j}r^{\frac{1}{2}}\log r,
\end{split}
\end{align*}
where we have used Lemma~\ref{lem:fjgj_1}, equation~\ref{eq:logint1}, and integration by parts. This completes the proof of \eqref{eq:F2temp2} and hence of \eqref{eq:F2bound1}. Turning to \eqref{eq:F2bound2} again by induction we assume that
\begin{align*}
\begin{split}
|(\calG_1\calG_1^\dagger)^jg_2(r)|\leq \begin{cases} (Cr_0)^{j}r^{-\frac{1}{2}},\quad &0\leq r\leq 1\\ (Cr_0)^{j}e^{-\sqrt{2}r},\quad& 1\leq r\leq r_0\end{cases}.
\end{split}
\end{align*}
Our goal is then to prove that
\begin{align}\label{eq:F2temp3}
\begin{split}
|\calG_1^\dagger(\calG_2\calG_1^\dagger)^jg_2(r)|\leq \begin{cases} C^{\frac{1}{2}}(Cr_0)^{j}r^{-\frac{1}{2}},\quad &0\leq r\leq 1\\ C^{\frac{1}{2}}(Cr_0)^{j}e^{-\sqrt{2}r},\quad& 1\leq r\leq r_0\end{cases},
\end{split}
\end{align}
and
\begin{align}\label{eq:F2temp4}
\begin{split}
|(\calG_2\calG_1^\dagger)^{j+1}g_2(r)|\leq \begin{cases} (Cr_0)^{j+1}r^{-\frac{1}{2}},\quad &0\leq r\leq 1\\ (Cr_0)^{j+1}e^{-\sqrt{2}r},\quad& 1\leq r\leq r_0\end{cases}.
\end{split}
\end{align}
For $r\leq1$ by the induction assumption
\begin{align*}
\begin{split}
|\calG_1^\dagger(\calG_2\calG_1^\dagger)^jg_2(r)|&\lesssim (Cr_0)^j\int_r^1(r^{\frac{3}{2}}s^{-1}+r^{-\frac{1}{2}}s)\ud s+(Cr_0)^j\int_1^\infty (r^{\frac{3}{2}}s^{\frac{1}{2}}\log s+r^{-\frac{1}{2}}s^{\frac{1}{2}})e^{\sqrt{2}s}\ud s\\
&\lesssim (Cr_0)^jr^{-\frac{1}{2}}.
\end{split}
\end{align*}
For $1\leq r\leq r_0$ by Lemma~\ref{lem:fjgj_1}
\begin{align*}
\begin{split}
|\calG_1^\dagger(\calG_2\calG_1^\dagger)^jg_2(r)|&\lesssim (C r_0)^j\int_r^\infty (sr)^{\frac{1}{2}}\log (\frac{s}{r})e^{-\sqrt{2}s}\ud s+(C r_0)^jr^{-\frac{3}{2}}\int_r^\infty s^{\frac{1}{2}}\log(\frac{s}{r})e^{-\sqrt{2}s}\ud s\\
&\lesssim (Cr_0)^je^{-\sqrt{2}r}.
\end{split}
\end{align*}
This proves \eqref{eq:F2temp3}. Using \eqref{eq:F2temp3}, for $r\leq1$
\begin{align*}
\begin{split}
|(\calG_2\calG_1^\dagger)^{j+1}g_2(r)|&\lesssim C^{\frac{1}{2}}(Cr_0)^j\int_0^rr^{-\frac{1}{2}}s\ud s+C^{\frac{1}{2}}(Cr_0)^j\int_r^1r^{\frac{3}{2}}s^{-1}\ud s+C^{\frac{1}{2}}(Cr_0)^{j}\int_1^rr^{\frac{3}{2}}e^{-2\sqrt{2}s}\ud s\\
&\lesssim C(Cr_0)^jr^{-\frac{1}{2}}.
\end{split}
\end{align*}
For $1\leq r\leq r_0$
\begin{align*}
\begin{split}
|(\calG_2\calG_1^\dagger)^{j+1}g_2(r)|&\lesssim C^{\frac{1}{2}}(Cr_0)^je^{-\sqrt{2}r}\int_0^1s\ud s+C^{\frac{1}{2}}(Cr_0)^je^{-\sqrt{2}r}\int_1^r \ud s+C^{\frac{1}{2}}(Cr_0)^je^{\sqrt{2}r}\int_r^{r_0}e^{-2\sqrt{2}s}\ud s\\
&\lesssim C^{\frac{1}{2}}(Cr_0)^jr_0e^{-\sqrt{2}r}.
\end{split}
\end{align*}
This completes the proof of \eqref{eq:F2temp4} and hence of \eqref{eq:F2bound2}. The derivative bounds \eqref{eq:F2bound3} and \eqref{eq:F2bound4} are proved in a similar way. Estimates~\eqref{eq:F2bound9} and~\eqref{eq:F2bound10} for $(r\partial_r)^jf_2$ and $\partial_r^j g_2$ follow from Lemma~\ref{lem:fjgj_1}. The arguments for  $(r\partial_r)^j\calG_2f_2$ and $\partial_r^j\calG_1^\dagger g_2$ are similar to the corresponding ones, \eqref{eq:F1bound9} and~\eqref{eq:F1bound10}, in the proof of Proposition~\ref{lem:F1_1}.
\end{proof}

\section{The $L^2$ solution near $r=\infty$}\label{sec:Psis}

In this section, we construct the Weyl-Titchmarsh matrix solutions. 
Specifically, for small $z \in \bbC \backslash \bbR$ we devise matrix solutions $\Psi_{+}(r,z)$, respectively $\Psi_{-}(r,z)$, to
\begin{equation}
 i \calL \Psi_{\pm}(r,z) = z \Psi_{\pm}(r,z)
\end{equation}
so that $\Psi_{+}(\cdot,z) \in L^2((1,\infty))$ when $\Im(z) > 0$, respectively $\Psi_{-}(\cdot, z) \in L^2((1,\infty))$ when $\Im(z) < 0$.
We recall from \eqref{equ:small_constants} that throughout this paper we work with a string of fixed small absolute constants
\begin{equation} \label{equ:small_constants_section_large_r}
    0 < \delta_0 \ll \epsilon_\infty \ll \epsilon_0 \ll 1.
\end{equation}

We seek to construct $\Psi_{\pm}(r,z)$ so that as $r\to\infty$ it approaches the solution $\Psi_\pm^{(0)}(\cdot,z) \in L^2((1,\infty))$ (unique up to normalizations) to
\begin{equation} \label{equ:calLinfty_leading_order}
 i \calL_\infty \Psi_\pm^{(0)}(r,z) = z \Psi_\pm^{(0)}(r,z)
\end{equation}
with the reference operator
\begin{equation}
 \calL_\infty := \begin{bmatrix} 0 & \calH_0 \\ -(\calH_0 + 2) & 0 \end{bmatrix}, \quad \calH_0 := -\partial_r^2 - \frac{1}{4r^2}.
\end{equation}
Then we can write
\begin{equation*}
 i\calL = i\calL_\infty + i \calV(r), \quad \calV(r) := \begin{bmatrix} 0 & V_1 \\ -V_2 & 0 \end{bmatrix},
\end{equation*}
with
\begin{equation}
\begin{aligned}
 V_1(r) := \rho_1(r)^2 - 1 + \frac{1}{r^2}, \quad V_2(r) := 2\bigl( \rho_1(r)^2-1 \bigr) + \frac{1}{r^2}.
\end{aligned}
\end{equation}
Observe that in view of the asymptotics
\begin{equation}
 \rho_1(r) = 1 - \frac{1}{2r^2} - \frac{9}{8r^4} + \calO\Bigl(\frac{1}{r^6}\Bigr), \quad r \gg 1,
\end{equation}
we have that
\begin{equation}
 |V_1(r)| \lesssim \frac{1}{r^4}, \quad |V_2(r)| \lesssim \frac{1}{r^2}, \quad r \gg 1.
\end{equation}
Our choice of reference operator $\calL_\infty$ instead of a simpler candidate for a reference operator at spatial infinity such as the constant coefficient one is related to dealing with the asymmetry between the faster decay of $V_1(r)$ and the slow decay of $V_2(r)$ as $r\to\infty$.

\subsection{Bessel functions and the Stokes phenomenon}

In this subsection we recall the properties of the Jost solutions for the scalar Bessel operator $\calH_0$ from which our reference operator $\calL_\infty$ is built.
This requires that we revisit the classical topic of the asymptotics of Bessel functions in the complex domain. 

\begin{lemma} \label{lem:calH0_ode}
 For $\zeta \in \bbC \backslash \{0\}$, $\Im \zeta\ge 0$, the problem
 \begin{equation} \label{equ:ode_calH0}
  \left\{ \begin{aligned}
   &\calH_0 h = \zeta^2 h, \quad 0 < r < \infty, \\
   &\lim_{r\to\infty} e^{-i\zeta r} h = 1,
 \end{aligned} \right.
 \end{equation}
 has a unique solution denoted by $h(r;\zeta)$ satisfying
 \begin{equation*}
  \lim_{r\to\infty} e^{-i\zeta r} h'(r;\zeta) = i\zeta.
 \end{equation*}
\end{lemma}

Lemma~\ref{lem:calH0_ode}  will be a simple consequence of the following lemma about the solutions of the ODE \eqref{eq:bessel0} below. The ODE \eqref{eq:bessel0}  has the Hankel functions $\{ \sqrt{z}H_0^{(1)}(z), \sqrt{z}H_0^{(2)}(z)\}$ as a fundamental system. Our goal is to derive uniform asymptotics for these functions in the closed upper half-plane. The delicate part is the derivative bound in~\eqref{eq:H0dom}, which cannot hold as we include the negative real axis. This is due to the Stokes phenomenon, which manifests itself through a singularity at $p=-2i$ in the Borel plane, see below. Questions of this type are of course classical, see~\cite[Chapter 11]{Olv}, but it is not   a simple matter to extract the needed information from such sources. We therefore give a self-contained and simple presentation. 

\begin{lemma}\label{lem:Hankel0}
    The  equation 
   \begin{equation}\label{eq:bessel0}
       h''(z) + \big( 1+\frac{1}{4z^2}\big)h(z)=0
   \end{equation}
    admits a holomorphic  fundamental system $\{ h_+(z), h_-(z)\}$ in the upper half-plane $\Im z>0$. The recessive branch $h_+(z)$ satisfies
    \begin{equation}
        \begin{aligned}
        \label{eq:H0rec}
    h_+ (z) &\sim e^{ iz} \text{\ as\ }|z|\to\infty\\
    \big| e^{-iz} h_+ (z)\big| &\les 1, \quad \big| {z^2} (e^{- iz} h_+ (z))'\big| &\les 1, \quad \big| {z^3} (e^{- iz} h_+ (z))''\big| \les 1,
    \end{aligned}
    \end{equation}
    uniformly in $z\in\Omega_+:=\{ z\in\Cp\::\: \Im z\ge0,\; |z|\ge1\}$.
   The dominant branch satisfies  
 \begin{equation}
        \begin{aligned}
        \label{eq:H0dom}
    h_- (z) &\sim e^{ -iz} \text{\ as\ }|z|\to\infty\\
    \big| e^{iz} h_- (z)\big| &\le C,\quad |e^{iz}h_{-}'(z)|\le C, \quad \big| {z^2} (e^{ iz} h_- (z))'\big| &\le C_\delta
    \end{aligned}
    \end{equation}
    uniformly in 
    \[ z\in\Omega_{+,\delta}:=\{ z\in\Cp\::\: \Im z\ge0,\; |z|\ge1,\; \Im z \ge -\delta\Re z\}\]
    where $\delta\in(0,1)$ is arbitrary but fixed. Here $C$ is an absolute constant, while $C_\delta\to\infty$ as $\delta\to 0$. In fact, 
   \begin{equation}
       \label{eq:stokes}
       h_-(x) \sim e^{-ix} + 2i e^{ix}\text{\ \ as\ \ } x\to-\infty
   \end{equation}
     An analogous statement holds in the lower half-plane.  
    \end{lemma}

First we present the proof of Lemma~\ref{lem:calH0_ode} assuming Lemma~\ref{lem:Hankel0}.
\begin{proof}[Proof of Lemma~\ref{lem:calH0_ode}]
The assertion reduces to studying the following ordinary differential equation for holomorphic functions on the punctured plane
\begin{equation} \label{equ:fuchsian_ode}
 \left\{ \begin{aligned}
  &-g'' - \frac{1}{4\xi^2} g = g, \quad \xi \in \bbC \backslash \{0\},\;\Im\xi\ge0 \\
  &\lim_{|\xi|\to\infty} e^{-i\xi} g(\xi) = 1.
 \end{aligned} \right.
\end{equation}
Setting $h(r;\zeta) := h_+(\zeta r)$ gives the (unique) solution to \eqref{equ:ode_calH0}, where $h_+$ is the recessive branch Lemma~\ref{lem:Hankel0}. 
\end{proof}
 Next we prove Lemma~\ref{lem:Hankel0}.
\begin{proof}[Proof of Lemma~\ref{lem:Hankel0}]
We seek a dominant solution in the form $h_-(z)=e^{-iz} z^{\frac12} g(z)$. Then one checks that 
\[
g''(z) + (\frac{1}{z}-2i)g'(z) -\frac{i}{z}g(z)=0
\]
We solve this ODE in the right-half plane $\Re z>0$ via the Laplace transform
\[
g(z) = \int_0^\infty H(p) e^{-pz}\, dp 
\]
Then one computes that $H$ satisfies
\[
-\int_0^p H(q)q\, dq - i\int_0^p H(q)\, dq + H(p)p(p+2i)=0
\]
After passing to a derivative we obtain the ODE
\[
H'(p) p (2i+p) + H(p) (i+p)=0
\]
with solution
\[
H(p)= \frac{w_0}{\sqrt{p(2i+p)}}
\]
and some complex $w_0\ne0$. 
Thus, we obtain a dominant branch $h_-(z)$ in $\Re z>0$ of the form
\begin{equation}
    \label{eq:dom}
    h_-(z)=w_0 e^{-iz} z^{\frac12} \int_0^\infty \frac{e^{-pz}}{\sqrt{p(2i+p)}}\, dp
\end{equation}
Expanding around $p=0$ yields
\[
\int_0^\infty \frac{e^{-pz}}{\sqrt{p(2i+p)}}\, dp  = \sqrt{\frac{\pi}{2iz}}(1+\calO(z^{-1}))\qquad z\to \infty
\]
which means that $w_0=\sqrt{\frac{2}{\pi}}e^{i\frac{\pi}{4}}$. This   defines a solution of~\eqref{eq:bessel0} analytic in the right-half plane and such that $h_-(z)\sim e^{-iz}$ as $z\to\infty$. The function $\Phi:p\mapsto \sqrt{p(2i+p)} $ is holomorphic on the domain $U:=\bbC\setminus [0,i\infty)\cup [-2i,-i\infty)$, which is the complex plane slit from $-2i$ vertically down, and from $0$ vertically up. By our normalizations,  $\Phi(x)\sim x$ as $x\to\infty$ and 
\begin{equation}
    \label{eq:Phi_jump}
    \Phi(-i\sigma+0)=-\Phi(-i\sigma-0) \text{\ \ for all\ \ }\sigma>2.
\end{equation}
The domain $U$ is invariant under the reflection $-i+\zeta\mapsto -i-\zeta$ for all $\zeta\in i+U$.  Moreover, \begin{equation}
    \label{eq:Phi_symm}
    \Phi(-i+\zeta)=\sqrt{(-i+\zeta)(i+\zeta)}=\sqrt{1+\zeta^2}=\Phi(-i-\zeta) 
\end{equation}
for all such~$\zeta$. In particular, $\Phi(x)\sim -x$ as $x\to-\infty$.

Turning the contour to $e^{-i\theta}(0,\infty)$ for $0\le\theta<\frac{\pi}{2}$, we analytically continue $h_-(z)$ to $\Re (e^{-i\theta} z)>0$.
In fact, for $z\in \Omega_{+,\delta}$ as above 
\[
h_-(z)= w_0 e^{-iz}   \int_{e^{i\beta}(0,\infty)} \frac{e^{-t}}{\sqrt{t(2i+t/z)}}\, dt
\]
provided $0<\beta<\frac{\pi}{2}$ is sufficiently close to $\frac{\pi}{2}$. The bounds~\eqref{eq:H0dom} hold by inspection.

To compute the $\lim_{z\to x} h_-(z)$ for $x\le -1$ from the upper half-plane, we rotate the $p$-contour in~\eqref{eq:dom} from the positive half-axis to the negative one in the clockwise sense. In the process we cross the branch cut $[-2i,-i\infty)$. 
To be more specific, 
by Cauchy's theorem we can write for all $z\in \Omega_+$ with $\frac{3\pi}{4}<\arg(z)<\pi$
\[
h_-(z)= w_0 e^{-iz} z^{\frac12} \int_\gamma \frac{e^{-pz}}{\sqrt{p(2i+p)}}\, dp  - w_0 e^{-iz} z^{\frac12} \int_\Gamma \frac{e^{-pz}}{\sqrt{p(2i+p)}}\, dp =: h_{-,1}(z)+h_{-,2}(z)
\]
where $\Gamma $ is a Hankel contour going from $-i\infty$ circling around $-2i$ counter-clockwise, and then returning to $-i\infty$. Thus, $\Gamma$ runs around the vertical slit starting at~$-2i$. The first integral runs along $\gamma=e^{-\theta i}(0,\infty)$, with $\theta=\pi-\delta$, $\delta>0$ small,  which we can then rotate into the negative real axis by sending $\delta\to0+$, and converges absolutely.  See the picture below. The contribution of the dashed circular arcs vanishes as $R\to\infty$ in the picture, because of the exponentially decaying factor in the corresponding integrals.
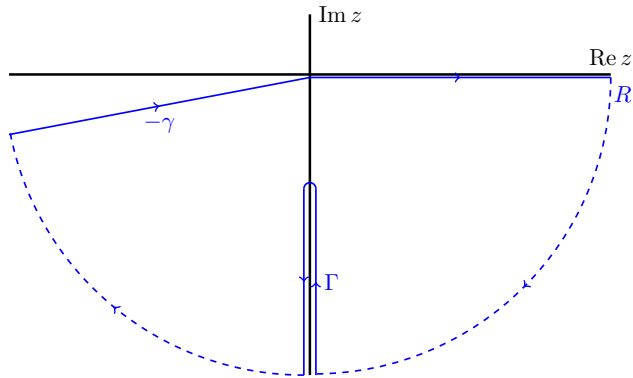
\begin{center}
\scalebox{0.8}{
\begin{tikzpicture}
\begin{scope}[thick,decoration={
    markings,
    mark=at position 0.5 with {\arrow{>}}}]
\draw[black, very thick] (-5,0) -- (5,0) node[right,above]{$\Re z$};
\draw[black, very thick] (0,1) node[right]{$\Im z$} -- (0,-5) ;
\draw[blue,postaction={decorate}] (-5,-1)-- node[below]{$-\gamma$} (0,-0.05) ;
\draw[blue,postaction={decorate}] (0,-0.05) --(5,-0.05) node[right= 0.2cm,below] {$R$};
\draw[blue,postaction={decorate}] (0.1,-5)-- node [right]{$\Gamma$} (0.1,-1.9);
\draw[blue,postaction={decorate}] (-0.1,-1.9)--(-0.1,-5);
\draw [blue] (-0.1,-1.9) arc  (180:0:0.1) (0.1,-1.9);
\draw [blue, dashed,postaction={decorate}] (5,-0.05) arc  (0:-90:4.93) (0.1,-5);
\draw [blue, dashed,postaction={decorate}] (-0.1,-5) arc  (-90:-169:4.97) (-5,-1);
\end{scope}
\end{tikzpicture}}
\end{center}
The roots in the integrand are given by the holomorphic function $\Phi$ defined above.   Thus, on the one hand, 
\begin{equation}
\label{eq:h-1}
h_{-,1}(z)=-w_0 e^{-iz} z^{\frac12} \int_0^\infty \frac{e^{sz}}{\sqrt{s(s-2i)}}\, ds \sim e^{-iz}\qquad z\to -\infty,\; \Im z>0
\end{equation}
and on the other hand, 
\begin{align*}
   h_{-,2}(z) &= -2 i w_0 e^{-iz} z^{\frac12} \int_0^\infty \frac{e^{-z(-2i-is)}}{\Phi(-2i-is+0)}\, ds = -2 i w_0 e^{iz} z^{\frac12} \int_0^\infty \frac{e^{isz}}{\sqrt{(-2i-is)(-is)}}\, ds \\ 
   &= -2 i w_0 e^{iz} z^{\frac12} \int_0^\infty \frac{e^{isz}}{-i\sqrt{s(2+s)}}\, ds = 2i  e^{iz} (1+\calO(z^{-1}))\qquad z\to -\infty
\end{align*}
We arrive at the choice of sign in \eqref{eq:h-1} as follows: parametrizing the negative axis by $\gamma(s)=-s$, with $s>0$, we have
\[
h_{-,1}(z)=-w_0 e^{-iz} z^{\frac12} \int_0^\infty \frac{e^{sz}}{\Phi(-s)}\, ds 
\]
Now $\Phi(-s)=\sqrt{s(s-2i)}\sim s$ as $s\to\infty$. The sign here comes from~\eqref{eq:Phi_symm}, and proves the first equality sign in~\eqref{eq:h-1}. The final statement in~\eqref{eq:h-1} holds either in the form $\sim e^{-iz}$ or $\sim -e^{-iz}$. To determine which of these it is, we calculate the integral with a specific choice of $z=-L+i0+$ with $L>0$ and large.   Then 
\[
 h_{-,1}(z)=-\sqrt{\frac{2}{\pi}} e^{iL} i L^{\frac12} \int_0^\infty \frac{e^{-sL}}{\sqrt{s(s-2i)}}\, ds \sim -\sqrt{\frac{2}{\pi}} e^{i\frac{\pi}{4}} e^{iL} i L^{\frac12} \sqrt{\frac{\pi}{2}} e^{i\frac{\pi}{4}} L^{-\frac12} \sim e^{iL}
\]
as $L\to\infty$.

A similar Laplace transform analysis yields a recessive 
branch $h_-(z)$ in $\Re z>0$ of the form
\begin{equation}
    \label{eq:rec}
    \begin{aligned}
         h_+(z) &=\sqrt{\frac{2}{\pi}} e^{-i\frac{\pi}{4}}  e^{iz} z^{\frac12} \int_0^\infty \frac{e^{-pz}}{\sqrt{p(p-2i)}}\, dp 
    \end{aligned}
\end{equation}
$\Psi(p):=\sqrt{p(p-2i)} $ can be defined as a holomorphic function in the entire lower half-plane in such a way that $\Psi(p)\sim p$ as $p\to\infty$ in that region. By Cauchy's theorem, we can thus deform the contour for every $z\in\Omega_+$ in such a way that~\eqref{eq:rec} turns into 
\begin{equation}\label{eq:h+stokestemp1}
 h_+(z) =\sqrt{\frac{2}{\pi}} e^{-i\frac{\pi}{4}}  e^{iz}   \int_0^\infty \frac{e^{-t}}{\sqrt{t(t/z-2i)}}\, dt,
\end{equation}
where the square root denotes the principal branch. This is due to 
$$|2i-t/z|=|2i-se^{-i\theta}| \geq 2$$ for all $0\le\theta\le\pi$ and $s\ge0$. The estimates in~\eqref{eq:H0rec} follow from~\eqref{eq:h+stokestemp1}
\end{proof}
\begin{remark}
    In principle, the coefficients of the expansion $h_{-}(x)\sim a e^{-ix}+b e^{ix}$ as $x\to-\infty$ can also be computed by expanding $h_{-}(x)$ in terms of the fundamental system $h_{+}(\pm x)$. The coefficients can then be computed by taking Wronskians with $h_{+}(\pm x)$ using the small $x$ asymptotics of $h_\pm(x)$. See~\eqref{eq:Hankelsmall1} below. However, we find that the proof above using analytic continuation is more direct and sheds light on the source of the Stokes phenomenon.
\end{remark}
In view of the expansions
\begin{align*}
    \sqrt{x}\, H_0^{(1)}(x) \sim \sqrt{\frac{2}{\pi}}e^{i(x-\frac{\pi}{4})}, \quad 
    \sqrt{x}\, H_0^{(2)}(x) \sim \sqrt{\frac{2}{\pi}}e^{-i(x-\frac{\pi}{4})} \quad \text{as } x \to \infty,
\end{align*}
we can identify $h_\pm(z)$ with the functions on the left. In fact, it holds 
\begin{equation}
    \label{eq:hpmHankel}
    h_+(z)= \sqrt{\frac{\pi}{2}}e^{i\frac{\pi}{4}}  \sqrt{z}\, H_0^{(1)}(z),\qquad h_-(z)= \sqrt{\frac{\pi}{2}}e^{-i\frac{\pi}{4}}  \sqrt{z}\, H_0^{(2)}(z) \quad \text{for all} \quad \Im z > 0.
\end{equation}
In particular, we have for some universal constant $C > 0$ that
\begin{equation}
    \label{eq:hpmbds}
    \sup_{0<\Im \zeta<1} |e^{\pm\Im \zeta}h_{\pm}(\zeta)|\le C.
\end{equation}
The importance of~\eqref{eq:hpmHankel} lies in the well-known asymptotic expansion of the Hankel functions as $z\to0$. This automatically solves the delicate problem of finding the behavior of   $h_{\pm}(z)$ for small~$z$.

\subsection{Construction of the solutions $\Psi_{\pm}^{(0)}(r,z)$ at infinity}

We now seek to construct solutions $\Psi_\pm^{(0)}(r,z)$ to $(i\calL_\infty-z) \Psi_\pm^{(0)} = 0$ for $\pm \Im(z) > 0$ with $\Psi_\pm^{(0)}(\cdot,z) \in L^2((1,\infty))$.
Suppose $\Phi = \begin{bmatrix} p \\ q \end{bmatrix}$ solves $i \calL_\infty \Phi = z \Phi$. Then $p(r,z)$ and $q(r,z)$ satisfy
\begin{equation} \label{equ:calLinfty_leading_order_coupled_system}
 \left\{ \begin{aligned}
  i \calH_0 q &= z p, \\
  (-i) (\calH_0 + 2) p &= z q.
 \end{aligned} \right.
\end{equation}
Thus, $q(r,z)$ must satisfy
\begin{equation}\nonumber
 \begin{aligned}
  (-i) (\calH_0 + 2) i \calH_0 q &= (-i) (\calH_0 + 2) z p = z (-i) (\calH_0 + 2) p = z^2 q,
 \end{aligned}
\end{equation}
whence
\begin{equation}\nonumber
 (\calH_0 + 2) \calH_0 q = z^2 q.
\end{equation}
Once we have determined $q(r,z)$, we obtain the corresponding choice for $p(r,z)$ from
\begin{equation}\nonumber
 p(r,z) = \frac{i}{z} (\calH_0 q)(r,z).
\end{equation}
The ansatz $q(r,z) = h_+(k(z)r)$ with $k \equiv k(z)$ leads to the equation
\begin{equation} \label{equ:Pkz}
 P(k,z) := k^4 + 2 k^2 - z^2 = 0
\end{equation}
and its Riemann surface which we now determine. 
Observe that
\begin{equation}\nonumber
 \partial_k P(k, z) = 4 k^3 + 4 k = 4k (k^2 + 1).
\end{equation}
Hence, $P(k,z) = \partial_k P(k,z) = 0$ for
\begin{equation}\nonumber
 (k,z) \in \bigl\{ (0,0), (i, i), (i, -i), (-i, i), (-i,-i) \bigr\}.
\end{equation}
It follows from the implicit function theorem that for $z \notin \{0, i, -i\}$ the equation \eqref{equ:Pkz} has four distinct roots, which we denote by $k_1(z)$, $k_2(z)$, $k_3(z)$, and $k_4(z)$.

For all $x > 0$ define
\begin{equation}\label{eq:four roots}
 \begin{aligned}
  f_1(x) &:= \sqrt{-1+\sqrt{1+x^2}}, \\
  f_2(x) &:= i \sqrt{1+\sqrt{1+x^2}}, \\
  f_3(x) &:= -\sqrt{-1+\sqrt{1+x^2}}, \\
  f_4(x) &:= -i \sqrt{1+\sqrt{1+x^2}},
 \end{aligned}
\end{equation}
where we use the convention that $\sqrt{a} > 0$ for positive real numbers $a > 0$.
Then we have $P(f_j(x), x) = 0$ for all $x > 0$ for $1 \leq j \leq 4$. Define the simply-connected domain
\begin{equation}\label{eq:Omega}
 \Omega := \bbC \, \backslash \, \bigl(  [i, i\infty) \cup (-i\infty, -i] \bigr).
\end{equation}
Then analytic continuation of these functions to $\Omega$ leads to the four distinct roots. Note that $z=0$ is not a branch point of any of these roots. For $1 \leq j \leq 4$, we denote by $k_j(z)$ the analytic continuation of $f_j(x)$ to $\Omega$.

\begin{lemma} \label{lem:positive_imaginary_parts}
 We have that
 \begin{equation}\nonumber
  \Im \, \bigl( k_1(z) \bigr) > 0 \quad \text{and} \quad \Im \, \bigl( k_2(z) \bigr) > 0 \quad \text{for all } z \in \Omega \text{ with } \Im(z) > 0.
 \end{equation}
 Similarly, we have that
 \begin{equation}\nonumber
  \Im \, \bigl( k_3(z) \bigr) > 0 \quad \text{and} \quad \Im \, \bigl( k_2(z) \bigr) > 0 \quad \text{for all } z \in \Omega \text{ with } \Im(z) < 0.
 \end{equation}
 Moreover, for $z\in\Omega$
\begin{align*}
\begin{split}
k_1(z)=-k_1(-z)=k_3(-z)=-k_3(z),\qquad k_2(z)=k_2(-z)=-k_4(z)=-k_4(-z).
\end{split}
\end{align*}
\end{lemma}
\begin{proof}
 We begin with the case of positive imaginary part $\Im(z) > 0$.
 Observe that for $0 < y < 1$,
 \begin{equation}\nonumber
  k_1(iy) = i \sqrt{1-\sqrt{1-y^2}}, \quad k_2(iy) = i \sqrt{1 + \sqrt{1-y^2}},
 \end{equation}
 whence we have for $0 < y < 1$,
 \begin{equation}\nonumber
  \Im\bigl( k_1(iy) \bigr) > 0 \quad \text{and} \quad \Im \bigl( k_2(iy) \bigr) > 0.
 \end{equation}
 Now suppose that there exists $z_0 \in \Omega$ with $\Im(z_0) > 0$ such that say $\Im(k_1(z_0)) = 0$. Then by definition
 \begin{equation}
 \nonumber
  z_0^2 = k_1(z_0)^4 + 2 k_1(z_0)^2 \in (0,\infty),
 \end{equation}
 which is a contradiction to $\Im(z_0) > 0$. The assertion follows.

 The case of negative imaginary part $\Im(z) < 0$ can be handled analogously starting with the observation that for $0 < y < 1$,
 \begin{equation}\nonumber
  k_3(-iy) = i \sqrt{1-\sqrt{1-y^2}}, \quad k_2(-iy) = i \sqrt{1 + \sqrt{1-y^2}},
 \end{equation}
 whence we have for $0 < y < 1$,
 \begin{equation}\nonumber
  \Im\bigl( k_3(-iy) \bigr) > 0 \quad \text{and} \quad \Im \bigl( k_2(-iy) \bigr) > 0,
 \end{equation}
 as desired.
 
 Finally, for the relations amongst $k_j$, note that $k_1(x)=-k_3(x)$ and $k_2(x)=-k_4(x)$ hold for $x\in(0,\infty)$ and hence in $\Omega$ by analytic continuation. For the statements $k_1(z)=-k_1(-z)$ and $k_2(z)=k_2(-z)$, again by analytic continuation it suffices to show these relations near $z=0$. The latter holds by expanding $k_1(x)$ and $k_2(x)$ in power series near $x=0$ and observing that only odd powers of $x$ appear for $k_1$ and only even powers of $x$ appear for $k_2$.
\end{proof}

\medskip

For $q(r,z) := h_{+}(rk(z)) \sim e^{i k(z) r}$ (as $r\to\infty$) to be square-integrable near infinity, we need $\Im k(z) > 0$.
In view of Lemma~\ref{lem:positive_imaginary_parts} we arrive at the following definitions
\begin{equation}\label{eq:Psipm0+1}
   \Psi_+^{(0)}(r,z) := \begin{bmatrix}
                      \frac{i k_1(z)^2}{z} h_+(r k_1(z)) & \frac{i k_2(z)^2}{z} h_+(r k_2(z)) \\
                      h(_+r k_1(z)) & h_+(r k_2(z))
                     \end{bmatrix}, \quad \Im(z) > 0,
\end{equation}
and
\begin{equation}\label{eq:Psipm0-1}
 \begin{aligned}
   \Psi_-^{(0)}(r,z) &:= \begin{bmatrix}
                      \frac{i k_3(z)^2}{z} h_{+}(r k_3(z)) & \frac{i k_2(z)^2}{z} h_{+}(r k_2(z)) \\
                      h_{+}(r k_3(z)) & h_{+}(r k_2(z))
                     \end{bmatrix}, \quad \Im(z) < 0.
 \end{aligned}
\end{equation}
These are the Weyl-Titchmarsh matrix solutions of the reference operator $\calL_\infty$.

\subsection{Construction of $\Psi_+(r,z)$ via Lyapunov-Perron}\label{sec:Upsilon1}
Next, we establish the existence of a solution $\Psi_+(r,z)$ to $(i\calL - z ) \Psi_+(r,z) = 0$ for $\Im(z) > 0$ with prescribed behavior $\Psi_+^{(0)}(r,z)$ at $r=\infty$. We also seek to establish error bounds how well $\Psi_+(r,z)$ is approximated by $\Psi_+^{(0)}(r,z)$. 

For $z \in \Omega$ with $\Im(z) > 0$ we rewrite the system of second-order ODEs
\begin{equation}\nonumber
 i \calL \Psi = z \Psi, \quad \Psi(r,z) = \begin{bmatrix} p(r,z) \\ q(r,z) \end{bmatrix},
\end{equation}
as a system of first-order ODEs
\begin{equation} \label{equ:first_order_system_infinity}
 \begin{aligned}
  \partial_r \Upsilon(r;z) = A(r;z) \Upsilon(r;z) + R(r) \Upsilon(r;z), \quad \Upsilon(r;z) := \begin{bmatrix} q(r,z) \\ \partial_r q(r,z) \\ p(r,z) \\ \partial_r p(r,z) \end{bmatrix},
 \end{aligned}
\end{equation}
where
\begin{equation}\label{eq:ARdef1}
 \begin{aligned}
  A(r;z) := \begin{bmatrix} 0 & 1 & 0 & 0 \\ -\frac{1}{4r^2} & 0 & iz & 0 \\ 0 & 0 & 0 & 1 \\ -iz & 0 & 2-\frac{1}{4r^2} & 0 \end{bmatrix}, \quad R(r) := \begin{bmatrix} 0 & 0 & 0 & 0 \\ V_1(r) & 0 & 0 & 0 \\ 0 & 0 & 0 & 0 \\ 0 & 0 & V_2(r) & 0 \end{bmatrix}.
 \end{aligned}
\end{equation}
For $j = 1, \ldots, 4$ define 
\begin{align}\label{eq:wjdef}
\begin{split}
&w^+_1(r;z):=\pmat{h_{+}(k_1(z)r)\\k_1(z)h_{+}'(k_1(z)r)\\\frac{ik_1^2(z)}{z}h_{+}(k_1(z)r)\\\frac{ik_1(z)^3}{z}h_{+}'(k_1(z)r)},\quad w^+_2(r;z):=\pmat{h_{+}(k_2(z)r)\\k_2(z)h_{+}'(k_2(z)r)\\\frac{ik_2^2(z)}{z}h_{+}(k_2(z)r)\\\frac{ik_2(z)^3}{z}h_{+}'(k_2(z)r)},\\
&w^+_3(r;z):=\pmat{h_{-}(k_1(z)r)\\k_1(z)h_{-}'(k_1(z)r)\\\frac{ik_1^2(z)}{z}h_{-}(k_1(z)r)\\\frac{ik_1(z)^3}{z}h_{-}'(k_1(z)r)},\quad w^+_4(r;z):=\pmat{h_{-}(k_2(z)r)\\k_2(z)h_{-}'(k_2(z)r)\\\frac{ik_2^2(z)}{z}h_{-}(k_2(z)r)\\\frac{ik_2(z)^3}{z}h_{-}'(k_2(z)r)}.
\end{split}
\end{align}
Then we have for $j = 1, \ldots, 4$ that
\begin{equation}\nonumber
 (\partial_r - A(r;z)) w_j^+(r;z) = 0.
\end{equation}
We want to approximate the solutions to $i\calL\Psi =z\Psi$ and their derivatives by the corresponding components of $w_j(r;z)$ in a suitable sense to be made precise below. This will be done for $r\geq r_\infty$ where
\begin{align*}
\begin{split}
r_\infty(z) := \frac{\epsilon_\infty}{|z|}
\end{split}
\end{align*}  
with $0 < |z| \leq \delta_0 \ll \epsilon_\infty \ll 1$ as in \eqref{equ:small_constants_section_large_r}.
The norms $X_j$ below are meant to capture the behavior of $w_j^+$ as $r\to\infty$.
\begin{defi}\label{def:Xnorms}
    For an $\bbR^4$ valued function $F=(F_1,\dots,F_4)$ let 
\begin{align*}
    \|F\|_{X_1} &:=\sup_{r\geq r_\infty}\big(|e^{-ik_1(z)r}F_1(r)|+|z|^{-1}|e^{-ik_1(z)r}F_2(r)|\\ &
    \qquad\qquad +|z|^{-1}|e^{-ik_1(z)r}F_3(r)|+|z|^{-2}|e^{-ik_1(z)r}F_4(r)|\big),\\
    \|F\|_{X_2} &:=\sup_{r\geq r_\infty}\big(|e^{-ik_2(z)r}F_1(r)|+|e^{-ik_2(z)r}F_2(r)|\\
    & \qquad\qquad +|z||e^{-ik_2(z)r}F_3(r)|+|z||e^{-ik_2(z)r}F_4(r)|\big),\\
    \|F\|_{X_3} & :=\sup_{r\geq r_\infty}\big(|e^{ik_1(z)r}F_1(r)|+|z|^{-1}|e^{ik_1(z)r}F_2(r)|\\
    &\qquad\qquad +|z|^{-1}|e^{ik_1(z)r}F_3(r)|+|z|^{-2}|e^{ik_1(z)r}F_4(r)|\big),\\
    \|F\|_{X_4} &:=\sup_{r\geq r_\infty}\big(|e^{ik_2(z)r}F_1(r)|+|e^{ik_2(z)r}F_2(r)|\\
    & \qquad\qquad +|z||e^{ik_2(z)r}F_3(r)|+|z||e^{ik_2(z)r}F_4(r)|\big).
\end{align*}
We will use the same notation when $F$ depends on $z$ as well as $r$.
\end{defi}
The choice of these norms is motivated by the formulae \eqref{eq:wjdef}
as explained by the following lemma. 
\begin{lemma} \label{lem:wjXj}
    There exists an absolute constant $C > 0$ such that for $0 < |z| \leq \delta_0 \ll \epsilon_\infty$ with $\Im(z) > 0$ and $r_\infty := \epsilon_\infty |z|^{-1}$, 
    \[
    \max_{j=1,3} \, \|w^+_j(\cdot; z)\|_{X_j} \leq C\epsilon_\infty^{-\frac{1}{2}} \log\epsilon_\infty^{-1}, \qquad  \max_{j=2,4} \, \|w^+_j(\cdot; z)\|_{X_j} \leq C.
    \]
\end{lemma}
\begin{proof}
This follows from \eqref{eq:hpmHankel}, Lemma~\ref{lem:Hankel0}, and the following small $\zeta$ asymptotics 
\begin{align}\label{eq:Hankelsmall1}
\begin{split}
h_{\pm}(\zed) &= c_1^\pm\sqrt{\zed}\log\zed+c_2^\pm\sqrt{\zed}+\calO(\zed^{\frac{5}{2}}\log\zed),\\
h_{\pm}'(\zed) &=\frac{c_1^\pm}{2}\zed^{-\frac{1}{2}}\log\zed+( c_1^\pm+\frac{c_2^\pm}{2})\zed^{-\frac{1}{2}} + \calO(\zed^{\frac{3}{2}}\log\zed).
\end{split}
\end{align}
See \cite[(9.1.12)--(9.1.13)]{AS}. 
\end{proof}
In the next proposition we construct solutions to $i\calL\Psi=z\Psi$ for $\Im z>0$ with leading order given by $w_j^+(\cdot,z)$.
\begin{prop}
\label{prop:Upsilon}
 Let $0 < |z| \leq \delta_0 \ll \epsilon_\infty \ll 1$ with $\Im(z) > 0$. Then the first-order system
 \begin{equation}\nonumber
  \partial_r \Upsilon(r;z) = A(r;z) \Upsilon(r;z) + R(r) \Upsilon(r;z)
 \end{equation}
 has a fundamental set of solutions $\Upsilon_j^+(r,z)$, $1 \leq j \leq 4$, satisfying
 \begin{equation}\nonumber
    \lim_{r\to\infty} e^{-ik_j(z) r} \bigl( \Upsilon_j^+(r;z) - w_j^+(r;z) \bigr) = 0.
 \end{equation}
 In fact there exists a constant $C_{\epsilon_\infty} > 0$, depending on $\epsilon_\infty$ but not $z$, such that
\begin{align}\label{eq:Upsilon+lead1}
\begin{split}
\|\Upsilon_j^{+}(\cdot;z)-w_j^{+}(\cdot;z)\|_{X_j}\leq C_{\epsilon_\infty}|z|.
\end{split}
\end{align}
 The solution $\Upsilon^+_2$ is the most recessive branch and thus unique, while the other 
  $\Upsilon^+_j$ are unique up to adding subordinate branches: to $\Upsilon^+_j$ for $j=3,4$ we can add multiples of $\Upsilon^+_i$ for all $1\le i<j$. To $\Upsilon^+_1$ we can add a multiple of $\Upsilon^+_2$. 
\end{prop}
\begin{rem}
Estimate \eqref{eq:Upsilon+lead1}  is especially important in Section~\ref{sec:Greens} where we compute the Wronskians between the $\Psi^{\pm}(r,z)$ and $F_1(r,z)$, $F_2(r,z)$. Estimate \eqref{eq:Upsilon+lead1}  will allow us to replace  $\Upsilon_j^{+}$ by $w_j^{+}$ in those computations. See Lemma~\ref{lem:bellj1} for more details.
\end{rem}
We prove this proposition by a contraction mapping argument. To set the stage, we define the relevant linear operators. That these are the relevant operators will become clear in the proof of Proposition~\ref{prop:Upsilon} below.  Given
\begin{equation}\nonumber
 \begin{aligned}
  \Upsilon^+_\ell(r;z) = \begin{bmatrix} \Upsilon^+_{\ell,1}(r;z) \\ \Upsilon^+_{\ell;2}(r;z) \\ \Upsilon^+_{\ell,3}(r;z) \\ \Upsilon^+_{\ell,4}(r;z) \end{bmatrix},\quad \ell=1,2,3,4,
 \end{aligned}
\end{equation}
set
\begin{equation} \label{eq:alphabeaell1}
    \begin{aligned}
    &\alpha_\ell(r;z) := \frac{1}{2iz}\frac{1}{k_2^2(z)-k_1^2(z)}\big(-ik_2^3(z)V_1(r)\Upsilon^+_{\ell,1}(r;z)+zk_2(z)V_2(r)\Upsilon^+_{\ell,3}(r;z)\big),\\
    &\beta_\ell(r;z) := \frac{1}{2iz}\frac{1}{k_2^2(z)-k_1^2(z)}\big(ik_1^3(z)V_1(r)\Upsilon^+_{\ell,1}(r;z)-zk_1(z)V_2(r)\Upsilon^+_{\ell,3}(r;z)\big).
    \end{aligned}
\end{equation}
Then for $1 \leq j \leq 4$, define $\calF_j(r,z,\Upsilon^+_j(r;z))$ by
\begin{flalign} 
\calF_1(r,z,\Upsilon^+_1(r;z)) &:= \int_r^\infty \big(w^+_3(r;z)h_{+}(k_1(z)s)-w^+_1(r;z)h_{-}(k_1(z)s)\big)\alpha_1(s;z)\ud s&&\label{eq:Upsilonapp1}\\
&\quad+ \int_{r_\infty}^r w^+_2(r;z)h_{-}(k_2(z)s)\beta_1(s;z)\ud s+\int_r^\infty w^+_4(r;z)h_{+}(k_2(z)s)\beta_1(s;z)\ud s,&&
\end{flalign}
\begin{flalign} 
\calF_2(r,z,\Upsilon^+_2(r;z))&:=\int_r^\infty \big(w^+_3(r;z)h_{+}(k_1(z)s)-w^+_1(r;z)h_{-}(k_1(z)s)\big)\alpha_2(s;z)\ud s&&\label{eq:Upsilonapp2}\\
&\phantom{:=}+\int_r^\infty \big(w^+_4(r;z)h_{+}(k_2(z)s)-w^+_2(r;z)h_{-}(k_2(z)s)\big)\beta_2(s;z)\ud s,&&
\end{flalign}
\begin{flalign}
    \calF_3(r,z,\Upsilon^+_3(r;z))&:=\int_{r_\infty}^{r}w^+_1(r;z)h_{-}(k_1(z)s)\alpha_3(s;z)\ud s+\int_{r_\infty}^rw^+_2(r;z)h_{-}(k_2(z)s)\beta_3(s;z)\ud s&&\label{eq:Upsilonapp3}\\
    &\quad+\int_{r}^\infty w^+_3(r;z)h_{+}(k_1(z)s)\alpha_3(s;z)\ud s+\int_r^{\infty}w^+_4(r;z)h_{+}(k_2(z)s)\beta_3(s;z)\ud s,&&
\end{flalign}
\begin{flalign}
    \calF_4(r,z,\Upsilon^+_4(r;z))&:=\int_{r_\infty}^{r}w^+_1(r;z)h_{-}(k_1(z)s)\alpha_4(s;z)\ud s+\int_{r_\infty}^rw^+_2(r;z)h_{-}(k_2(z)s)\beta_4(s;z)\ud s&&\label{eq:Upsilonapp4} \\
    &\quad-\int_{r_\infty}^rw^+_3(r;z)h_{+}(k_1(z)s)\alpha_4(s;z)\ud s+\int_r^{\infty}w^+_4(r;z)h_{+}(k_2(z)s)\beta_4(s;z)\ud s.&&
\end{flalign}
The main technical ingredient of the proof of Proposition~\ref{prop:Upsilon} is the following lemma.
\begin{lemma} \label{lem:contr}
    Suppose $0 < |z| \leq \delta_0 \ll \epsilon_\infty \ll 1$ and $r_\infty := \epsilon_\infty |z|^{-1}$. Then $\calF_j$ is a contraction on~$X_j$. More precisely, $\|\calF_j\|_{X_j}\lesssim_{\epsilon_\infty} |z|\|\Upsilon_j\|_{X_j}$. 
\end{lemma}
Lemma~\ref{lem:contr} will be proved in Subsection~\ref{sec:Upsilon2} below. Assuming Lemma~\ref{lem:contr} for now, we present the proof of Proposition~\ref{prop:Upsilon}.
\begin{proof}[Proof of Proposition~\ref{prop:Upsilon}]
We begin with the case $j=1$. Write
\begin{equation}\nonumber
 \begin{aligned}
  \Upsilon^+_1(r;z) = \begin{bmatrix} \Upsilon^+_{1,1}(r;z) \\ \Upsilon^+_{1;2}(r,z) \\ \Upsilon^+_{1,3}(r;z) \\ \Upsilon^+_{1,4}(r;z) \end{bmatrix}.
 \end{aligned}
\end{equation}
Observe that
\begin{equation}\nonumber
 \begin{aligned}
  R(r) \Upsilon^+_1(r;z) = \begin{bmatrix} 0 \\ V_1(r) \Upsilon^+_{1,1}(r;z) \\ 0 \\ V_2(r) \Upsilon^+_{1,3}(r;z) \end{bmatrix}.
 \end{aligned}
\end{equation}
We can therefore write $R(r) \Upsilon^+_1(r;z)$ as a linear combination
\begin{equation} \label{equ:source_term_in_terms_of_alpha_beta}
 \begin{aligned}
  R(r) \Upsilon^+_1(r;z) = \begin{bmatrix} 0 \\ 2 i k_1(z) \\ 0 \\ \frac{-2 k_1(z)^3}{z} \end{bmatrix} \alpha(r;z) + \begin{bmatrix} 0 \\ 2 i k_2(z) \\ 0 \\ \frac{-2 k_2(z)^3}{z} \end{bmatrix} \beta(r;z),
 \end{aligned}
\end{equation}
where the coefficients $\alpha(r;z)$ and $\beta(r;z)$ are the solutions to the linear system
\begin{equation}\nonumber
 \begin{aligned}
  \begin{bmatrix} 2 i k_1(z) & 2 i k_2(z) \\ \frac{-2 k_1(z)^3}{z} & \frac{-2 k_2(z)^3}{z} \end{bmatrix}
  \begin{bmatrix} \alpha(r;z) \\ \beta(r;z) \end{bmatrix}
  = \begin{bmatrix} V_1(r) \Upsilon^+_{1,1}(r;z) \\ V_2(r) \Upsilon^+_{1,3}(r;z) \end{bmatrix}.
 \end{aligned}
\end{equation}
Using that $k_1(z) k_2(z) = i z$ for all $z \in \Omega$ (which follows by analytic continuation from the corresponding identity $k_1(x) k_2(x) = i x$ for $x > 0$), we compute
\begin{equation}\nonumber
 \begin{aligned}
  \begin{bmatrix} 2 i k_1(z) & 2 i k_2(z) \\ \frac{-2 k_1(z)^3}{z} & \frac{-2 k_2(z)^3}{z} \end{bmatrix}^{-1} = \frac{1}{2iz} \frac{1}{k_2(z)^2-k_1(z)^2} \begin{bmatrix} -ik_2(z)^3 & z k_2(z) \\ i k_1(z)^3 & -z k_1(z) \end{bmatrix}.
 \end{aligned}
\end{equation}
Thus, we obtain
\begin{equation} \label{equ:definition_alpha_beta}
 \begin{aligned}
  \alpha(r;z) &= \frac{1}{2iz} \frac{1}{k_2(z)^2-k_1(z)^2} \Bigl( -i k_2(z)^3 V_1(r) \Upsilon^+_{1,1}(r;z) + z k_2(z) V_2(r) \Upsilon^+_{1;3}(r;z) \Bigr), \\
  \beta(r;z) &= \frac{1}{2iz} \frac{1}{k_2(z)^2-k_1(z)^2} \Bigl( i k_1(z)^3 V_1(r) \Upsilon^+_{1,1}(r;z) - z k_1(z) V_2(r) \Upsilon^+_{1;3}(r;z) \Bigr),
 \end{aligned}
\end{equation}
which agrees with \eqref{eq:alphabeaell1} with $\ell=1$ above. We will write $\alpha_1$ and $\beta_1$ for $\alpha$ and $\beta$ from now on. Note that $h_\pm(k_1(z)r)$ are two (linearly independent) solutions to the same ODE
\begin{equation}\nonumber
 -u''-\frac{1}{4r^2} u = k_1(z)^2 u ,
\end{equation}
and analogously for $h_\pm(k_2(z)r)$. Their Wronskians must therefore be independent of $r$, and can be evaluated from their asymptotic behavior
\begin{equation}\nonumber
 \begin{aligned}
  W\bigl[ h_+(\cdot k_1(z)), h_-(\cdot k_1(z)) \bigr] &= k_1(z)h_+( k_1(z)r) h_{-}'( k_1(z)r) - k_1(z)h_{+}'(k_1(z)r) h_{-}(k_3(z)r) \\
  & = -2ik_1(z),
 \end{aligned}
\end{equation}
and analogously for $h_\pm(\cdot k_2(z))$,
\begin{equation}\nonumber
 \begin{aligned}
  W\bigl[ h_+(\cdot k_2(z)), h_-(\cdot k_2(z)) \bigr] &= k_2(z)h_{+}( k_2(z)r) h_{-}'(k_2(z)r) -k_2(z) h_{+}'(k_2(z)r) h_{-} (k_4(z)r)\\ &= -2ik_2(z).
 \end{aligned}
\end{equation}
Following the Lyapunov-Perron approach, we will construct $\Upsilon^+_1(r;z)$ on an interval $[r_\infty, \infty)$ as a solution to the following integral equation
\begin{equation} \label{equ:Upsilon1_integral_equation} 
\begin{aligned} 
\Upsilon^+_1(r;z) &= w^+_1(r;z) 
+ \int_r^\infty \big(w^+_3(r;z)h_{+}(k_1(z)s)-w^+_1(r;z)h_{-}(k_1(z)s)\big)\alpha_1(s;z) \ud s \\
&\quad + \int_{r_\infty}^r w^+_2(r;z)h_{-}(k_2(z)s)\beta_1(s;z)\ud s+\int_r^\infty w^+_4(r;z)h_{+}(k_2(z)s)\beta_1(s;z)\ud s,
\end{aligned} 
\end{equation}
where $\alpha_1(s;z)$ and $\beta_1(s;z)$ are defined as in \eqref{equ:definition_alpha_beta}. Note that the integrals on the right-hand side in \eqref{equ:Upsilon1_integral_equation} are precisely $\calF_1$ defined in \eqref{eq:Upsilonapp1}. That is, we can write \eqref{equ:Upsilon1_integral_equation} as
\begin{align}\label{eq:Upsilonapp1def}
    \Upsilon^+_1(r;z)= w^+_1(r;z)+\calF_1(r,z,\Upsilon^+_1(r;z)),
\end{align}
Let us verify that a solution $\Upsilon^+_1(r;z)$ to \eqref{equ:Upsilon1_integral_equation} satisfies $(\partial_r - A(r;z))\Upsilon^+_1(r;z) = R(r) \Upsilon^+_1(r;z)$. To this end we first compute 
\begin{equation}\nonumber
 \begin{aligned}
  &(\partial_r - A(r;z)) \biggl(  \int_r^\infty w^+_3(r;z) h_+(k_1(z)s) \alpha_1(s;z) \, \ud s -\int_r^\infty w^+_1(r;z) h_{-}(k_1(z)s) \alpha_1(s;z) \, \ud s\biggr) \\
  &= \bigl( w^+_1(r;z) h_{-}(k_1(z)r) - w^+_3(r;z) h_{+}(k_1(z)r) \bigr) \alpha_1(r;z) \\
  &= \begin{bmatrix} 1 \\ 0 \\ \frac{i k_1(z)^2}{z} \\ 0 \end{bmatrix} \underbrace{\bigl( h_{+}(k_1(z)r) h_{-}(k_1(z)r) - h_{-}(k_1(z)r) h_{+}(k_1(z)r) \bigr)}_{= \, 0} \alpha_1(r;z) \\
  &\quad + \begin{bmatrix} 0 \\ k_1(z) \\ 0 \\ \frac{i k_1(z)^3}{z} \end{bmatrix} \bigl( h_{+}'(k_1(z)r) h_{-}(k_1(z)r) - h'_{-}(k_1(z)r) h_{+}(k_1(z)r) \bigr) \alpha_1(r;z) \\
  &= \begin{bmatrix} 0 \\ k_1(z) \\ 0 \\ \frac{i k_1(z)^3}{z} \end{bmatrix} \underbrace{W\bigl[h_{-}(\cdot), h_{+}(\cdot)\bigr]}_{= \,2i} \alpha_1(r;z) = \begin{bmatrix} 0 \\ 2ik_1(z) \\ 0 \\ \frac{-2 k_1(z)^3}{z} \end{bmatrix} \alpha_1(r;z).
 \end{aligned}
\end{equation}
An analogous computation yields that
\begin{equation}\nonumber
 \begin{aligned}
  &(\partial_r - A(r;z)) \biggl( \int_{r_\infty}^r w^+_2(r;z) h_{-}(k_2(z)s) \alpha_1(s;z) \, \ud s + \int_r^\infty w^+_4(r;z) h_{+}(k_2(z)s) \beta_1(s;z) \, \ud s \biggr) \\
  &= \begin{bmatrix} 0 \\ 2ik_2(z) \\ 0 \\ \frac{-2 k_2(z)^3}{z} \end{bmatrix} \beta_1(r;z).
 \end{aligned}
\end{equation}
By \eqref{equ:source_term_in_terms_of_alpha_beta} it follows that $(\partial_r - A(r;z))$ applied to the right-hand side of \eqref{equ:Upsilon1_integral_equation} gives $R(r) \Upsilon^+_1(r;z)$, as desired.  Similarly, we can set up the Lyapunov-Perron integral equations for the other branches $\Upsilon^+_j$, $2\le j\le4$, that is,
\begin{align}
    &\Upsilon^+_2(r;z)=w^+_2(r,z)+\calF_2(r,z,\Upsilon^+_2(r;z)),\label{eq:Upsilonapp2def}\\
    & \Upsilon^+_3(r;z)=w^+_3(r;z)+\calF_3(r,z,\Upsilon^+_3(r;z)), \label{eq:Upsilonapp3def}\\
    &\Upsilon^+_4(r;z)=w^+_4(r;z)+\calF_4(r,z,\Upsilon^+_4(r;z)).\label{eq:Upsilonapp4def}\\
\end{align}
The proposition then follows from Lemma~\ref{lem:contr} applied to \eqref{eq:Upsilonapp1def}, \eqref{eq:Upsilonapp2def}, \eqref{eq:Upsilonapp3def}, \eqref{eq:Upsilonapp4def}.
\end{proof}

\subsubsection{Fixed point argument for $\Upsilon^+$}\label{sec:Upsilon2}

Here we prove Lemma~\ref{lem:contr}.
\begin{proof}[Proof of Lemma~\ref{lem:contr}]
 Starting with $\calF_1$ we   estimate the $X_1$ norm of each integral in \eqref{eq:Upsilonapp1}. 
We denote by $C > 0$ a constant that is independent of $r_\infty$ and $z$, but whose value may change from line to line. First, in view of~\eqref{eq:hpmbds}, using Lemma~\ref{lem:wjXj} we have 
\begin{align*}
    & \Big\|\int_{r}^\infty w^+_1(r;z)h_{-}(k_1(z)s)\alpha_1(s;z)\ud s\Big\|_{X_1} \leq C|z|^{-1}\|\Upsilon^+_1\|_{X_1}\|w^+_1\|_{X_1} \int_{r_\infty}^\infty(s^{-4}+|z|^2s^{-2})\ud s\\
    &\leq C \|w^+_1\|_{X_1} (r_\infty^{-3}|z|^{-1}+|z|r_\infty^{-1})\|\Upsilon^+_1\|_{X_1} \les_{\epsilon_\infty} |z|^2 \; \|\Upsilon^+_1\|_{X_1}. 
\end{align*}
Next, using that $\Im k_2(z)-\Im k_1(z)$ is strictly positive (uniformly for $|z|$ small),
\begin{align*}
    &\Big\|\int_{r_\infty}^r w^+_2(r;z)h_{-}(k_2(z)s)\beta_1(s;z)\ud s\Big\|_{X_1} \\&\leq C|z|^{-1} \|w^+_2\|_{X_2}\|\Upsilon^+_1\|_{X_1}\sup_{r\geq r_\infty}\int_{r_\infty}^r e^{-(\Im k_2(z)-\Im k_1(z))(r-s)}s^{-2}\ud s\\
    &\leq C  \|w^+_2\|_{X_2} |z|^{-1}r_\infty^{-2} \|\Upsilon^+_1\|_{X_1} \les_{\epsilon_\infty} |z| \|\Upsilon^+_1\|_{X_1}.
\end{align*}
as well as, 
\begin{align*}
    &\Big\|\int_r^\infty w^+_3(r;z)h_{+}(k_1(z)s)\alpha_1(s;z)\ud s\Big\|_{X_1}\\
    &\leq C|z|^{-1}\|\Upsilon^+_1\|_{X_1}\|w^+_3\|_{X_3} \sup_{r\geq r_\infty}\int_{r}^\infty e^{-2\Im k_1(z)(s-r)}(s^{-4}+|z|^2s^{-2})\ud s\\
    &\leq C \|w^+_3\|_{X_3} (r_\infty^{-3}|z|^{-1}+|z|r_\infty^{-1})\|\Upsilon^+_1\|_{X_1} \les_{\epsilon_\infty} |z|^2 \; \|\Upsilon^+_1\|_{X_1} 
\end{align*}
For the final integral in \eqref{eq:Upsilonapp1},
\begin{align*}
    &\Big\|\int_{r}^\infty w^+_4(r;z)h_{+}(k_2(z)s)\beta_1(s;z)\ud s\Big\|_{X_1}\\
    &\leq C|z|^{-1} \|w^+_4\|_{X_4}\|\Upsilon^+_1\|_{X_1}\sup_{r\geq r_\infty}\int_{r}^\infty e^{-(\Im k_2(z)+\Im k_1(z))(s-r)}s^{-2}\ud s\\
    &\leq C  \|w^+_4\|_{X_4} |z|^{-1}r_\infty^{-2} \|\Upsilon^+_1\|_{X_1} \les_{\epsilon_\infty}|z| \|\Upsilon^+_1\|_{X_1}.
\end{align*}
In conclusion,
\[
    \|\calF_1(r,z,\Upsilon^+_1(r;z))\|_{X_1}\les_{\epsilon_\infty}|z|\|\Upsilon^+_1\|_{X_1}.
\]
We proceed to estimate $\|\calF_2\|_{X_2}$, where for the reader's convenience we recall from \eqref{eq:Upsilonapp2} that $\calF_2$ is given by 
\begin{align*} 
\begin{split}
\calF_2(r,z,\Upsilon^+_2(r;z))&:=\int_r^\infty \big(w^+_3(r;z)h_{+}(k_1(z)s)-w^+_1(r;z)h_{-}(k_1(z)s)\big)\alpha_2(s;z)\ud s\\
&\phantom{:=}+\int_r^\infty \big(w^+_4(r;z)h_{+}(k_2(z)s)-w^+_2(r;z)h_{-}(k_2(z)s)\big)\beta_2(s;z)\ud s.
\end{split}
\end{align*}
For the first integral,
\begin{align*}
    &\Big\|\int_r^\infty \big(w^+_3(r;z)h_{+}(k_1(z)s)-w^+_1(r;z)h_{-}(k_1(z)s)\big)\alpha_2(s;z)\ud s\Big\|_{X_2}\\
    &\leq C|z|^{-1} \|w^+_3\|_{X_3}\|\Upsilon^+_2\|_{X_2}\sup_{r\geq r_\infty}\int_r^\infty e^{-(\Im k_1(z)+\Im k_2(z))(s-r)} s^{-2}\ud s\\
    &\quad+C|z|^{-1}\|w^+_1\|_{X_1}\|\Upsilon^+_2\|_{X_2}\sup_{r\geq r_\infty}\int_r^\infty e^{-(\Im k_2(z)-\Im k_1(z))(s-r)}s^{-2}\ud s\\
    &\les_{\epsilon_\infty}|z| \|\Upsilon^+_2\|_{X_2}.
\end{align*}
Similarly, for the second integral,
\begin{align*}
    &\Big\|\int_r^\infty \big(w^+_4(r;z)h_{+}(k_2(z)s)-w^+_2(r;z)h_{-}(k_2(z)s)\big)\beta_2(s;z)\ud s\Big\|_{X_2}\\
    &\leq C  \|w^+_4\|_{X_4}\|\Upsilon^+_2\|_{X_2}\sup_{r\geq r_\infty}\int_r^\infty e^{-2\Im k_2(z)(s-r)}s^{-2}\ud s\\
    &\quad+ C \|w^+_2\|_{X_2}\|\Upsilon^+_2\|_{X_2}\sup_{r\geq r_\infty}\int_r^\infty s^{-2}\ud s\\
    &\les_{\epsilon_\infty}|z|^2 \|\Upsilon^+_2\|_{X_2}.
\end{align*}
In conclusion,
\[ \|\calF_2(r,z,\Upsilon^+_2(r;z))\|_{X_2}\les_{\epsilon_\infty}|z|\|\Upsilon^+_2\|_{X_2}.
\]
Next we turn to $\calF_4$
where  from \eqref{eq:Upsilonapp4}
\begin{align*}
\begin{split}
    \calF_4(r,z,\Upsilon^+_4(r;z))&:=\int_{r_\infty}^{r}w^+_1(r;z)h_{-}(k_1(z)s)\alpha_4(s;z)\ud s+\int_{r_\infty}^rw^+_2(r;z)h_{-}(k_2(z)s)\beta_4(s;z)\ud s\\
    &\quad-\int_{r_\infty}^rw^+_3(r;z)h_{+}(k_1(z)s)\alpha_4(s;z)\ud s+\int_r^{\infty}w^+_4(r;z)h_{+}(k_2(z)s)\beta_4(s;z)\ud s.
\end{split}
\end{align*}
Then the first integral in $\calF_4$ is bounded in $\|\cdot\|_{X_4}$ by
\begin{align*}
    & C \|w^+_1\|_{X_1}\|\Upsilon^+_4\|_{X_4}|z|^{-1}\sup_{r\geq r_\infty}\int_{r_\infty}^r e^{-(\Im k_1(z)+\Im k_2(z))(r-s)}s^{-2}\ud s \\
    & \les \|w^+_1\|_{X_1} \|\Upsilon^+_4\|_{X_4} r_\infty^{-2}|z|^{-1} \les_{\epsilon_\infty}|z| \|\Upsilon^+_4\|_{X_4} 
\end{align*}
The second integral in $\calF_4$ is bounded in $\|\cdot\|_{X_4}$ by
\begin{align*}
    &C \|w^+_2\|_{X_2}\|\Upsilon^+_4\|_{X_4}\sup_{r\geq r_\infty}\int_{r_\infty}^r e^{-2\Im k_2(z)(r-s)}s^{-2}\ud s\\
    & \les  \|w^+_2\|_{X_2}\|\Upsilon^+_4\|_{X_4} r_\infty^{-2}\les_{\epsilon_\infty}|z|^2 \|\Upsilon^+_4\|_{X_4}.
\end{align*}
The third integral in $\calF_4$ is bounded in $\|\cdot\|_{X_4}$ by
\begin{align*}
    &C \|w^+_3\|_{X_3}\|\Upsilon^+_4\|_{X_4}|z|^{-1}\sup_{r\geq r_\infty}\int_{r_\infty}^r e^{-(\Im k_2(z)-\Im  k_1(z))(r-s)}s^{-2}\ud s\\
    &\leq C \|w^+_3\|_{X_3}  \|\Upsilon^+_4\|_{X_4} r_\infty^{-2}|z|^{-1}\les_{\epsilon_\infty}|z| \|\Upsilon^+_4\|_{X_4}
\end{align*}
The fourth integral in $\calF_4$ is bounded in $\|\cdot\|_{X_4}$ by
\begin{align*}
    &C \|w^+_4\|_{X_4}\|\Upsilon^+_4\|_{X_4}\sup_{r\geq r_\infty}\int_r^\infty s^{-2}\ud s \leq C  \|w^+_4\|_{X_4} r_\infty^{-1}\|\Upsilon^+_4\|_{X_4} \les_{\epsilon_\infty}|z| \|\Upsilon^+_4\|_{X_4}
\end{align*}
In conclusion,
\[
    \|\calF_4(r,z,\Upsilon^+_4(r;z))\|_{X_4}\les_{\epsilon_\infty}|z|\|\Upsilon^+_4\|_{X_4}.
\]
Finally, we consider $\calF_3$
where by \eqref{eq:Upsilonapp3}
\begin{align*}
    \calF_3(r,z,\Upsilon^+_3(r;z))&:=\int_{r_\infty}^{r}w^+_1(r;z)h_{-}(k_1(z)s)\alpha_3(s;z)\ud s+\int_{r_\infty}^rw^+_2(r;z)h_{-}(k_2(z)s)\beta_3(s;z)\ud s\\
    &\quad+\int_{r}^\infty w^+_3(r;z)h_{+}(k_1(z)s)\alpha_3(s;z)\ud s+\int_r^{\infty}w^+_4(r;z)h_{+}(k_2(z)s)\beta_3(s;z)\ud s.
\end{align*}
The first integral in $\calF_3$ is bounded in $\|\cdot\|_{X_3}$ by
\begin{align*}
    &C \|w^+_1\|_{X_1}\|\Upsilon^+_3\|_{X_3}|z|^{-1}\sup_{r\geq r_\infty}\int_{r_\infty}^r e^{-2\Im k_1(z)(r-s)}(s^{-4}+|z|^2s^{-2})\ud s\\
    &\leq C \|w^+_1\|_{X_1}(|z|^{-1}r_\infty^{-3}+|z|r_\infty^{-1})\|\Upsilon^+_3\|_{X_3}\les_{\epsilon_\infty}|z|^2 \|\Upsilon^+_3\|_{X_3}
\end{align*}
The second integral in $\calF_3$ is bounded in $\|\cdot\|_{X_3}$ by
\begin{align*}
    &C\|w^+_2\|_{X_2} \|\Upsilon^+_3\|_{X_3}|z|^{-1}\sup_{r\geq r_\infty}\int_{r_\infty}^r e^{-(\Im k_1(z)+\Im k_2(z))(r-s)}s^{-2}\ud s\\
    &\leq C  \|w^+_2\|_{X_2} |z|^{-1}r_\infty^{-2} \|\Upsilon^+_3\|_{X_3} \les_{\epsilon_\infty}|z| \|\Upsilon^+_3\|_{X_3}.
\end{align*}
The third integral in $\calF_3$ is bounded in $\|\cdot\|_{X_3}$ by
\begin{align*}
    &C \|w^+_3\|_{X_3}\|\Upsilon^+_3\|_{X_3}|z|^{-1}\int_{r_\infty}^\infty(s^{-4}+|z|^2s^{-2})\ud s\\
    & \leq C \|w^+_3\|_{X_3}(|z|^{-1}r_\infty^{-3}+|z|r_\infty^{-1})\|\Upsilon^+_3\|_{X_3}\les_{\epsilon_\infty}|z|^2 \|\Upsilon^+_3\|_{X_3}.
\end{align*}
The fourth integral in $\calF_3$ is bounded in $\|\cdot\|_{X_3}$ by
\begin{align*}
    & C \|w^+_4\|_{X_4}\|\Upsilon^+_3\|_{X_3}|z|^{-1}\sup_{r\geq r_\infty}\int_r^\infty e^{-(\Im k_2(z)-\Im k_1(z))(s-r)}s^{-2}\ud s\\
    & \les \|w^+_4\|_{X_4} |z|^{-1}r_\infty^{-2} \|\Upsilon^+_3\|_{X_3}\les_{\epsilon_\infty}|z| \|\Upsilon^+_3\|_{X_3}.
\end{align*}
In conclusion
\[
    \|\calF_3(r,z,\Upsilon^+_3(r;z))\|_{X_3}\les_{\epsilon_\infty}|z| \|\Upsilon^+_3\|_{X_3}, 
\] 
and we are done. 
\end{proof}
 Before turning to the construction of $\Psi_{-}(r,z)$ in the lower half plane, we record a result about the parity of $\Upsilon_2^+(r;\lambda)$ in $\lambda\in\bbR\backslash\{0\}$ which will be needed in Section~\ref{sec:thejump}. Here and below by~$\Upsilon_j^+(r;\lambda)$ we mean $\lim_{b\to0}\Upsilon_j^+(r;\lambda+ib)$, with similar definitions for $w_j^+(r,\lambda)$.
\begin{cor}\label{cor:Upsilon2}
For $\lambda\in\bbR\backslash\{0\}$
\begin{align*}
    \Upsilon_2^{+,1}(r;-\lambda)=\Upsilon_2^{+,1}(r;\lambda),\qquad \Upsilon_2^{+,3}(r;-\lambda)=-\Upsilon_2^{+,3}(r;-\lambda).
\end{align*}
\end{cor}
\begin{proof}
    We suppress $+$ from the notation in the rest of this proof. From Proposition~\ref{prop:Upsilon} we know that
    \begin{align*}   \Upsilon_2(r;z)=w_2(r;z)+\calF_2(r,z,\Upsilon_2(r;z)).
    \end{align*}
    The desired parity holds for $w_2(r;z)$ because $k_2(z)=k_2(-z)$ by Lemma~\ref{lem:positive_imaginary_parts}. In view of the definitions \eqref{eq:Upsilonapp2} and \eqref{eq:alphabeaell1} of $\calF_2$ and $\alpha_2$, $\beta_2$, we can write 
    \begin{align*}    &\Upsilon_2^1(r;\lambda)=w_2^1(r;\lambda)+\int_r^\infty K_{1,1}(r,s,z)\Upsilon_2^1(s;\lambda)\ud s+\int_r^\infty K_{1,3}(r,s,z)\Upsilon_2^3(s;\lambda)\ud s,\\
    &\Upsilon_2^3(r;\lambda)=w_2^3(r;\lambda)+\int_r^\infty K_{3,1}(r,s,z)\Upsilon_2^1(s;\lambda)\ud s+\int_r^\infty K_{3,3}(r,s,z)\Upsilon_2^3(s;\lambda)\ud s,
    \end{align*}
    where
    \begin{align*}
        K_{1,1}(r,s,z)&=\frac{k_2^3(\lambda)}{2\lambda}\frac{V_1(s)}{k_1^2(\lambda)-k_2^2(\lambda)}\big(h_{-}(k_1(\lambda)r)h_{+}(k_1(\lambda)s)-h_{+}(k_1(\lambda)r)h_{-}(k_1(\lambda)s)\big)\\
        &\quad+\frac{k_1^3(\lambda)}{2\lambda}\frac{V_1(s)}{k_2^2(\lambda)-k_1^2(\lambda)}\big(h_{-}(k_2(\lambda)r)h_{+}(k_2(\lambda)s)-h_{+}(k_2(\lambda)r)h_{-}(k_2(\lambda)s)\big),
    \end{align*}
    \begin{align*}
        K_{1,3}(r,s,z)&=\frac{k_2(\lambda)}{2i}\frac{V_2(s)}{k_2^2(\lambda)-k_1^2(\lambda)}\big(h_{-}(k_1(\lambda)r)h_{+}(k_1(\lambda)s)-h_{+}(k_1(\lambda)r)h_{-}(k_1(\lambda)s)\big)\\
        &\quad+\frac{k_1(\lambda)}{2i}\frac{V_2(s)}{k_1^2(\lambda)-k_2^2(\lambda)}\big(h_{-}(k_2(\lambda)r)h_{+}(k_2(\lambda)s)-h_{+}(k_2(\lambda)r)h_{-}(k_2(\lambda)s)\big),
    \end{align*}
    \begin{align*}
        K_{3,1}(r,s,z)&=\frac{ik_1^2(\lambda)k_2^3(\lambda)}{2\lambda^2}\frac{V_1(s)}{k_1^2(\lambda)-k_2^2(\lambda)}\big(h_{-}(k_1(\lambda)r)h_{+}(k_1(\lambda)s)-h_{+}(k_1(\lambda)r)h_{-}(k_1(\lambda)s)\big)\\
        &\quad+\frac{ik_2^2(\lambda)k_1^3(\lambda)}{2\lambda^2}\frac{V_1(s)}{k_2^2(\lambda)-k_1^2(\lambda)}\big(h_{-}(k_2(\lambda)r)h_{+}(k_2(\lambda)s)-h_{+}(k_2(\lambda)r)h_{-}(k_2(\lambda)s)\big),
    \end{align*}
    \begin{align*}
        K_{3,3}(r,s,z)&=\frac{k_1^2(\lambda)k_2(\lambda)}{2\lambda}\frac{V_2(s)}{k_2^2(\lambda)-k_1^2(\lambda)}\big(h_{-}(k_1(\lambda)r)h_{+}(k_1(\lambda)s)-h_{+}(k_1(\lambda)r)h_{-}(k_1(\lambda)s)\big)\\
        &\quad+\frac{k_2^2(\lambda)k_1(\lambda)}{2\lambda}\frac{V_2(s)}{k_1^2(\lambda)-k_2^2(\lambda)}\big(h_{-}(k_2(\lambda)r)h_{+}(k_2(\lambda)s)-h_{+}(k_2(\lambda)r)h_{-}(k_2(\lambda)s)\big).
    \end{align*}
    Since $\Upsilon_2$ is a solution of a fixed point problem, it can be written as an iteration starting with $w_2$. Since as observed above $w_2$ satisfies the desired parity in $\lambda$, it suffices to prove that $K_{1,1}(r,s,\lambda)$ and $K_{3,3}(r,s,\lambda)$ are even in $\lambda$ and $K_{1,3}(r,s,\lambda)$ and $K_{3,1}(r,s,\lambda)$ are odd in $\lambda$. This follows, if we can show that $h_{\mathrm{odd}}(r,s,\lambda)$ and $h_{\mathrm{even}}(r,s,\lambda)$ are odd, respectively even, in $\lambda$, where
    \begin{align*}
        &h_{\mathrm{odd}}(r,s,\lambda):=h_{-}(k_1(\lambda)r)h_{+}(k_1(\lambda)s)-h_{+}(k_1(\lambda)r)h_{-}(k_1(\lambda)s),\\
        &h_{\mathrm{even}}(r,s,\lambda):=h_{-}(k_2(\lambda)r)h_{+}(k_2(\lambda)s)-h_{+}(k_2(\lambda)r)h_{-}(k_2(\lambda)s).
    \end{align*}
    That $h_{\mathrm{even}}$ is even in $\lambda$ follows from the fact $k_2(-\lambda)=k_2(\lambda)$. For $h_{\mathrm{odd}}$ we use the fact that $h_{+}(k_1(\lambda)r)$ and $h_{-}(k_1(\lambda)r)$ form a fundamental system for the ODE 
\begin{align*}
    \bmh''(r)+\frac{1}{4r^2}\bmh(r)=-k_1^2(\lambda)\bmh(r).
\end{align*}
We can therefore write $h_{+}(k_1(-\lambda)r)$ and $h_{-}(k_1(-\lambda)r)$ as linear combinations of $h_{+}(k_1(\lambda)r)$ and $h_{-}(k_1(\lambda)r)$. To compute the coefficients of these linear combinations we use Lemma~\ref{lem:Hankel0} which shows that with $c=2i$, for $\lambda>0$
\begin{align*}
    h_{+}(k_1(-\lambda)r)=h_{-}(k_1(\lambda)r),\qquad h_{-}(k_1(-\lambda)r)=h_{+}(k_1(\lambda)r)+ch_{-}(k_1(\lambda)r),
\end{align*}
and for $\lambda<0$
\begin{align*}
    h_{+}(k_1(-\lambda)r)=h_{-}(k_1(\lambda)r)-ch_{+}(k_1(\lambda)r),\qquad h_{-}(k_1(-\lambda)r)=h_{+}(k_1(\lambda)r).
\end{align*}
It follows from these expansions and a direct computation that $h_{\mathrm{odd}}$ is odd in $\lambda$ as desired.
\end{proof}
\subsection{Construction of $\Psi_-(r,z)$ via Lyapunov-Perron}\label{sec:Upsilonminus} 
The construction of $\Psi_{-}(r,z)$ is similar to that of $\Psi_+(r,z)$ and can in fact be worked out from that of $\Psi_{+}(r,z)$ and symmetry considerations. To start, recall from \eqref{eq:Psipm0-1} that
\begin{align*}
\begin{split}
  \Psi_-^{(0)}(r,z) &:= \begin{bmatrix}
                   \frac{i k_3(z)^2}{z} h_{+}(k_3(z)r) & \frac{i k_2(z)^2}{z} h_{+}(k_2(z)r) \\
                      h_{+}(k_3(z)r) & h_{+}(k_2(z)r)
                     \end{bmatrix}, \quad \Im(z) < 0.
\end{split}
\end{align*}
Recalling \eqref{equ:first_order_system_infinity} and \eqref{eq:ARdef1} we seek solutions of the equation 
\begin{align*}
  \partial_r\Upsilon(r;z)=\big(A(r;z)+R(r)\big)\Upsilon(r;z),\qquad \Im z<0.
\end{align*}
As in \eqref{eq:wjdef}, for $\Im (z)<0$ we let
\begin{align*}
\begin{split}
&w^-_1(r;z):=\pmat{h_{+}(k_3(z)r)\\k_3(z)h_{+}'(k_3(z)r)\\\frac{ik_3^2(z)}{z}h_{+}(k_3(z)r)\\\frac{ik_3(z)^3}{z}h_{+}'(k_3(z)r)},\quad w^-_2(r;z):=\pmat{h_{+}(k_2(z)r)\\k_2(z)h_{+}'(k_2(z)r)\\\frac{ik_2^2(z)}{z}h_{+}(k_2(z)r)\\\frac{ik_2(z)^3}{z}h_{+}'(k_2(z)r)},\\
&w^-_3(r;z):=\pmat{h_{-}(k_3(z)r)\\k_3(z)h_{-}'(k_3(z)r)\\\frac{ik_3^2(z)}{z}h_{-}(k_3(z)r)\\\frac{ik_3(z)^3}{z}h_{-}'(k_3(z)r)},\quad w^-_4(r;z):=\pmat{h_{-}(k_2(z)r)\\k_2(z)h_{-}'(k_2(z)r)\\\frac{ik_2^2(z)}{z}h_{-}(k_2(z)r)\\\frac{ik_2(z)^3}{z}h_{-}'(k_2(z)r)}.
\end{split}
\end{align*}
Since by Lemma~\ref{lem:positive_imaginary_parts} we have  $k_3(z)=k_1(-z)$, $k_2(z)=k_2(-z)$, for $\Im z<0$
\begin{align*}
    w_j^{-}(r;z)=\upiota w_j^{+}(r;-z),\qquad \upiota:=\pmat{1&0&0&0\\0&1&0&0\\0&0&-1&0\\0&0&0&-1}.
\end{align*}
By inspection
\begin{align*}
  \upiota A(r;-z)\upiota=A(r;z)
\end{align*}
and hence
\begin{align*}
    \partial_rw_j^{-}(r;z)=A(r;z)w_j^{-}(r;z),\qquad \Im z<0.
\end{align*}
We then define $\Upsilon_j^{-}$, $j=1,2,3,4$, for $\Im z<0$ by
\begin{align}\label{eq:Upsilonpmdef1}
    \Upsilon_j^{-}(r;z)=\upiota\Upsilon_j^{+}(r;-z).
\end{align}
It follows that for $\Im z<0$
\begin{align*}
    \lim_{r\to\infty}e^{-ik_j(-z)r}\big(\Upsilon_j^{-}(r;z)-w_j^{-}(r;z)\big)=0.
\end{align*}
Moreover, by inspection
\begin{align*}
    \upiota R(r)\upiota=R(r).
\end{align*}
It follows that
\begin{align*}
    \partial_r\Upsilon^{-}(r;z)=\big(A(r;z)+R(r)\big)\Upsilon_j^{-}(r;z)
\end{align*}
as desired.

\section{Green's kernel of $i\calL-z$}\label{sec:Greens}
The analysis of the previous sections puts us in the position to compute the distorted Fourier transform for $\calL$.  This is the ultimate goal of the current section. The most delicate step is justifying that the limit $b\to0^+$ can be taken inside the $\lambda$-integral in~\eqref{eq:localevol}. This is proved in Subsection~\ref{sec:thejump} and relies on the detailed analysis of the resolvent kernel $(i\calL-z)^{-1}$ for $\pm \Im z>0$ in the earlier subsections.

\subsection{Basic ansatz for the kernel}
To solve $(i\calL - z) \psi = \phi$ for $z \in \Omega$ with $\pm \Im(z) > 0$ we define the Green's functions
\begin{equation} 
 \calG_\pm(r,s; z) := \Psi_\pm(r,z) S(s,z) \mathds{1}_{[0 < s \leq r]} + F_1(r,z) T(s,z) \mathds{1}_{[r \leq s < \infty]},
\end{equation}
where we require the matrices $S(r,z)$ and $T(r,z)$ to satisfy
\begin{equation}\nonumber
 \begin{aligned}
  \Psi_\pm(r,z) S(r,z) - F_1(r,z) T(r,z) &= 0, \\
  (\partial_r \Psi_\pm)(r,z) S(r,z) - (\partial_r F_1)(r,z) T(r,z) &= \sigma_2,
 \end{aligned}
\end{equation}
with $\sigma_2$ one of the Pauli matrices defined in \eqref{eq:pauli}.
Then a solution to $(i\calL - z) \psi = \phi$ for $z \in \Omega$ with $\pm \Im(z) > 0$ is given by
\begin{equation}\nonumber
 \psi(r) := \int_0^\infty \calG_\pm(r,s;z) \phi(s) \, \ud s.
\end{equation}
Equivalently, we require the matrices $S(r,z)$ and $T(r,z)$ to satisfy
\begin{equation} \label{equ:4times4system_Green_ST}
 \begin{aligned}
  \begin{bmatrix} \Psi_\pm(r,z) & F_1(r,z) \\ (\partial_r \Psi_\pm)(r,z) & (\partial_r F_1)(r,z) \end{bmatrix} \begin{bmatrix} S(r,z) \\ - T(r,z) \end{bmatrix} = \begin{bmatrix} 0 \\ \sigma_2 \end{bmatrix}.
 \end{aligned}
\end{equation}
A decisive question is the invertibility of the $4\times4$ matrix on the left-hand side of \eqref{equ:4times4system_Green_ST}. To investigate this question we first set up some basic but important formalism in Subsections~\ref{subsec:matrixwronsks} and~\ref{subsec:calLGreen1}. This largely proceeds analogously to~\cite[Section 6]{KS}. We then  compute the relevant Wronskians in Subsection~\ref{sec:connection}. As already emphasized in Section~\ref{sec:intro}, a crucial difference with~\cite{KS} is that here we need to work with non-real $z$.
\subsection{Matrix Wronskians}\label{subsec:matrixwronsks}
Define the (real) vector inner product
\begin{equation}\nonumber
 \langle \bm{v}, \bm{w} \rangle := v_1 w_1 + v_2 w_2, \quad \bm{v} = \begin{bmatrix} v_1 \\ v_2 \end{bmatrix}, \quad \bm{w} = \begin{bmatrix} w_1 \\ w_2 \end{bmatrix}.
\end{equation}
We introduce the Wronskian
\begin{equation}\nonumber
 W[\bmf, \bmg] := \langle \bm{f}, \sigma_3 \bm{g}' \rangle - \langle \bm{f}', \sigma_3 \bm{g} \rangle, \quad \bmf(r) = \begin{bmatrix} f_1(r) \\ f_2(r) \end{bmatrix}, \quad \bmg(r) = \begin{bmatrix} g_1(r) \\ g_2(r) \end{bmatrix}.
\end{equation}
Note that $W[\bmf, \bmf]=0$ for any $\bmf$. We also introduce the following short-hand notation for our linearized operator
\begin{equation}\nonumber
 \begin{aligned}
  i\calL =  \sigma_2 \partial_r^2 + i \calV(r), \quad \calV(r) := \begin{bmatrix} 0 & \frac{3}{4r^2} + \rho_1^2 - 1 \\ - \bigl( \frac{3}{4r^2} + 3 \rho_1^2 - 1 \bigr) \end{bmatrix}.
 \end{aligned}
\end{equation}

\begin{lemma} \label{lem:constant_wronskian_vectorial}
 Let $z \in \bbC$. Suppose $(i\calL - z) \bmf = (i\calL - z) \bmg = 0$. Then we have
 \begin{equation}\nonumber
  \frac{\ud}{\ud r} W[\bmf, \bmg] = 0.
 \end{equation}
\end{lemma}
\begin{proof}
 Using the identity $\sigma_3 \sigma_2 \calV = \calV^t \sigma_2^t \sigma_3$, the claim follows by direct computation.
\end{proof}

Next, we introduce the matrix Wronskian
\begin{equation}\nonumber
 \begin{aligned}
  \calW[F,G] := F^t \sigma_3 G' - F'^t \sigma_3 G
 \end{aligned}
\end{equation}
for $2\times2$ matrices $F(r)$ and $G(r)$.

\begin{lemma} \label{lem:constant_wronskian_matrix}
 Let $z \in \bbC$. Suppose $(i\calL - z)F = (i\calL-z)G = 0$ for  $2\times2$ matrices $F(r)$ and $G(r)$. Then we have
 \begin{equation}\nonumber
  \frac{\ud}{\ud r} \calW[F,G] = 0.
 \end{equation}
\end{lemma}
\begin{proof}
 Denote by $\bmf_1(r)$ and $\bmf_2(r)$ the columns of $F(r)$, and by $\bmg_1(r)$ and $\bmg_2(r)$ the columns of $G(r)$. Then we have
 \begin{equation}\nonumber
  \calW[F,G] = \begin{bmatrix} W[\bmf_1, \bmg_1] & W[\bmf_1, \bmg_2] \\ W[\bmf_2, \bmg_1] & W[\bmf_2, \bmg_2] \end{bmatrix},
 \end{equation}
 and the assertion is a direct consequence of Lemma~\ref{lem:constant_wronskian_vectorial}.
\end{proof}

\begin{remark} \label{rem:analogues_wronskian_calL_infty}
 Analogues of Lemma~\ref{lem:constant_wronskian_vectorial} and Lemma~\ref{lem:constant_wronskian_matrix} also hold for our limiting operator $i \calL_\infty$ (at $r=\infty$).
\end{remark}

Next, we point out the following inversion identity (compare with \cite[Lemma~6.5]{KS}).
\begin{lemma} \label{lem:inverse_4times4}
 Let $F(r)$ and $G(r)$ be two ($r$-dependent) $2\times2$ matrices. Suppose that $\calW[F,F] = \calW[G,G] = 0$. Moreover, suppose that $D := \calW[F,G]$ is invertible. Then we have
 \begin{equation}\nonumber
 \begin{aligned}
  \begin{bmatrix} F & G \\ F' & G' \end{bmatrix}^{-1} = \begin{bmatrix} (D^t)^{-1} & 0 \\ 0 & D^{-1} \end{bmatrix} \begin{bmatrix} 0 & - I \\ I & 0 \end{bmatrix} \begin{bmatrix} F^t & F'^t \\ G^t & G'^t \end{bmatrix} \begin{bmatrix} 0 & \sigma_3 \\ -\sigma_3 & 0 \end{bmatrix}.
 \end{aligned}
\end{equation}
\end{lemma}
\begin{proof}
 Under the assumption that $\calW[F,F] = \calW[G,G] = 0$, we have that
 \begin{equation} \label{equ:4times4_inversion_observation}
 \begin{aligned}
  \begin{bmatrix} 0 & - I \\ I & 0 \end{bmatrix} \begin{bmatrix} F^t & F'^t \\ G^t & G'^t \end{bmatrix} \begin{bmatrix} 0 & \sigma_3 \\ -\sigma_3 & 0 \end{bmatrix} \begin{bmatrix} F & G \\ F' & G' \end{bmatrix} = \begin{bmatrix} - \calW[G,F] & 0 \\ 0 & \calW[F,G] \end{bmatrix}.
 \end{aligned}
 \end{equation}
 Noting that $D^t = -\calW[G,F]$, the assertion follows from the preceding identity \eqref{equ:4times4_inversion_observation}.
\end{proof}

Let us now return to the question of the invertibility of the $4\times4$ matrix on the left-hand side of \eqref{equ:4times4system_Green_ST}, say for $z \in \Omega$ with $\Im(z) > 0$.
Here we have that $(i\calL - z) \Psi_+(\cdot,z) = (i\calL - z) F_1(\cdot,z) = 0$.
Thus, by Lemma~\ref{lem:constant_wronskian_matrix} the matrix Wronskians
\begin{equation}
 \calW[\Psi_+(\cdot,z), \Psi_+(\cdot, z)], \quad \calW[F_1(\cdot,z), F_1(\cdot, z)], \quad \calW[\Psi_+(\cdot,z), F_1(\cdot, z)]
\end{equation}
are constant in $r$.
Since all entries of $\Psi_+^{(0)}(r,z)$ are exponentially decaying as $r\to\infty$, one has $\calW[\Psi_+(\cdot,z), \Psi_+(\cdot, z)] = 0$.

 Since $F_1(r) = \calO(r^{\frac32})$ as $r \to 0$ and $F'(r) = \calO(r^{\frac12})$ as $r \to 0$, we have $\calW[F_1(\cdot,z), F_1(\cdot, z)] = \calO(r^2)$ as $r \to 0$. Letting $r \to 0$ and keeping in mind that $\calW[F_1(\cdot,z), F_1(\cdot, z)]$ must be constant in $r$, it follows that $\calW[F_1(\cdot,z), F_1(\cdot, z)] = 0$.

\subsection{Solving for the Green's kernel}\label{subsec:calLGreen1}

By the preceding considerations, for $z \in \Omega$ with $\pm \Im(z) > 0$, the Green's function for $i\calL-z$ is given by
\begin{equation}
 \calG_\pm(r,s; z) := \Psi_\pm(r,z) S(s,z) \mathds{1}_{[0 < s \leq r]} + F_1(r,z) T(s,z) \mathds{1}_{[r \leq s < \infty]},
\end{equation}
with the matrices $S(r,z)$ and $T(r,z)$ determined by
\begin{equation}
 \begin{aligned}
  \begin{bmatrix} \Psi_\pm(r,z) & F_1(r,z) \\ \Psi_\pm'(r,z) & F_1'(r,z) \end{bmatrix} \begin{bmatrix} S(r,z) \\ - T(r,z) \end{bmatrix} = \begin{bmatrix} 0 \\ \sigma_2 \end{bmatrix}.
 \end{aligned}
\end{equation}
For $z \in \Omega$ with $\pm \Im(z) > 0$ we introduce the following short-hand notation for the matrix Wronskians (that are constant in $r$)
\begin{equation}
 D_\pm(z) := \calW\bigl[ \Psi_\pm(\cdot, z), F_1(\cdot, z) \bigr].
\end{equation}
Suppose $D_\pm(z)$ is invertible. Then by Lemma~\ref{lem:inverse_4times4} and the preceding observations,
\begin{equation}
 \begin{aligned}
  \begin{bmatrix} S(r,z) \\ - T(r,z) \end{bmatrix} &= \begin{bmatrix} \bigl( D_\pm(z)^{-1} \bigr)^t & 0 \\ 0 & D_\pm(z)^{-1} \end{bmatrix} \begin{bmatrix} 0 & -I \\ I & 0 \end{bmatrix} \begin{bmatrix} \Psi_\pm(r,z)^t & \Psi_\pm'(r,z)^t \\ F_1(r,z)^t & F_1'(r,z)^t \end{bmatrix} \begin{bmatrix} 0 & \sigma_3 \\ -\sigma_3 & 0 \end{bmatrix} \begin{bmatrix} 0 \\ \sigma_2 \end{bmatrix} \\
  &= \begin{bmatrix} - \bigl( D_\pm(z)^{-1} \bigr)^t F_1(r,z)^t \sigma_3 \sigma_2 \\ D_\pm(z)^{-1} \Psi_\pm(r,z)^t \sigma_3 \sigma_2 \end{bmatrix},
 \end{aligned}
\end{equation}
whence
\begin{equation}
 \begin{aligned}
  S(r,z) &= - \bigl( D_\pm(z)^{-1} \bigr)^t F_1(r,z)^t \sigma_3 \sigma_2, \\
  T(r,z) &= - D_\pm(z)^{-1} \Psi_\pm(r,z)^t \sigma_3 \sigma_2.
 \end{aligned}
\end{equation}
We arrive at the following expressions for the Green's functions
\begin{equation}\label{eq:G+}
 \begin{aligned}
  \calG_+(r,s;z) = \left\{ \begin{aligned} i \Psi_+(r,z) \bigl( D_+(z)^{-1} \bigr)^t F_1(s,z)^t \sigma_1,& \quad 0 < s \leq r, \\
  i F_1(r,z) D_+(z)^{-1} \Psi_+(s,z)^t \sigma_1,& \quad r \leq s < \infty, \end{aligned} \right.
 \end{aligned}
\end{equation}
and
\begin{equation}\label{eq:G-}
 \begin{aligned}
  \calG_-(r,s;z) = \left\{ \begin{aligned} i \Psi_-(r,z) \bigl( D_-(z)^{-1} \bigr)^t F_1(s,z)^t \sigma_1,& \quad 0 < s \leq r, \\
  i F_1(r,z) D_-(z)^{-1} \Psi_-(s,z)^t \sigma_1,& \quad r \leq s < \infty. \end{aligned} \right.
 \end{aligned}
\end{equation}

\subsection{The connection problem}\label{sec:connection}

The goal of the current subsection is to solve the connection problem. This refers to the evaluation of the Wronskians between the solutions to $(i\calL-z)\Psi=0$ constructed from $r=\infty$ and those constructed from $r=0$. The careful work in Sections~\ref{sec:Fs} and~\ref{sec:Psis} for $\pm\Im z>0$ starts to play  crucial at this point in the analysis. 

Recall from Section~\ref{sec:Fs} that the columns of $F_1$ and $F_2$ are denoted by $\upfy_j$, $1 \leq j \leq 4$, i.e.,
\begin{equation*}
    F_1=\pmat{\upfy_1&\upfy_2},\quad F_2=\pmat{\upfy_3&\upfy_4}.
\end{equation*}
Using the notation of Sections~\ref{sec:Upsilon1} and~\ref{sec:Upsilon2}, we set
\begin{equation} \label{equ:uppsi_definition}
    \uppsi_j^\pm = \pmat{\Upsilon_{j,3}^{\pm}\\\Upsilon_{j,1}^{\pm}},\qquad 1 \leq j \leq 4.
\end{equation}
Then by definition
\begin{align*}
\begin{split}
\partial_r\uppsi_j^\pm=\pmat{\Upsilon_{j,4}^{\pm}\\\Upsilon_{j,2}^{\pm}}.
\end{split}
\end{align*}
Moreover, the columns of $\Psi_\pm$ are given by $\uppsi_1^\pm$ and $\uppsi_2^\pm$, i.e.,
\begin{align*}
\begin{split}
    \Psi_{\pm}=\pmat{\uppsi_1^\pm&\uppsi_2^\pm}.
\end{split}
\end{align*}

The Wronskians between $\uppsi_i^\pm$ and $\upfy_j$ will play an important role and will be denoted by
\begin{align}\label{eq:omegacij1}
\begin{split}
c_{ij}(\lambda) &:= W[\upfy_i(\cdot,\lambda),\upfy_j(\cdot,\lambda)], \\ 
\omega_{ij}^\pm(\lambda)&:= W[\uppsi_i^\pm(\cdot,\lambda),\upfy_j(\cdot,\lambda)], \\
s_{ij}^{\pm}(\lambda) &:= W[\uppsi_i^\pm(\cdot,\lambda),\uppsi^\pm_j(\cdot,\lambda)],
\end{split}
\end{align}
where by definition $\uppsi_j^\pm(\cdot,\lambda):= \lim_{b\to0^+}\uppsi_j^\pm(\cdot,\lambda\pm i b)$. Since $\{\upfy_1,\dots,\upfy_4\}$ is a fundamental system, we can find constants $b_\ell^{\pm,j}(z)$ such that
\begin{align*}
\begin{split}
\uppsi_\ell^\pm(r,z)=\sum_{j=1}^4 b_\ell^{\pm,j}(z)\upfy_j(r,z),\quad \ell=1,\dots,4.
\end{split}
\end{align*}
For $D_\pm(z)$ we use the notation
\begin{align*}
\begin{split}
 D_\pm(z) := \calW\bigl[ \Psi_\pm(\cdot, z), F_1(\cdot, z) \bigr]=\pmat{\delta_\pm&-\gamma_\pm\\-\beta_\pm&\alpha_\pm},
\end{split}
\end{align*}
so that with $d_\pm:=\det D_\pm=\alpha_\pm\delta_\pm-\beta_\pm\gamma_\pm$,
\begin{align*}
\begin{split}
 D_\pm^{-1}(z)=\frac{1}{d_\pm}\pmat{\alpha_\pm&\gamma_\pm\\ \beta_\pm&\delta_\pm}.
\end{split}
\end{align*}
It follows from \eqref{eq:G+} and \eqref{eq:G-} that
\begin{align}\label{eq:G+2}
\begin{split}
\calG_{+}(r,s;z)&=\frac{i}{d_{+}}\sum_{j=1}^4(\alpha_+ b_1^{+,j}+\gamma_+ b_2^{+,j})(\upfy_j(r,z)\upfy_1^t(s,z)\sigma_1\mathds{1}_{\{0 < s \leq r\}}+\upfy_1(r,z)\upfy_j^t(s,z)\sigma_1\mathds{1}_{\{r < s\}})\\
&\quad+\frac{i}{d_{+}}\sum_{j=1}^4(\beta_+ b_1^{+,j}+\delta_+ b_2^{+,j})(\upfy_j(r,z)\upfy_2^t(s,z)\sigma_1\mathds{1}_{\{0 < s \leq r\}}+\upfy_2(r,z)\upfy_j^t(s,z)\sigma_1\mathds{1}_{\{r < s\}})
\end{split}
\end{align}
and
\begin{align}\label{eq:G-2}
\begin{split}
\calG_{-}(r,s;z)&=\frac{i}{d_{-}}\sum_{j=1}^4(\alpha_- b_1^{-,j}+\gamma_- b_2^{-,j})(\upfy_j(r)\upfy_1^t(s)\sigma_1\mathds{1}_{\{0 < s \leq r\}}+\upfy_1(r,z)\upfy_j^t(s,z)\sigma_1\mathds{1}_{\{r < s\}})\\
&\quad+\frac{i}{d_{-}}\sum_{j=1}^4(\beta_- b_1^{-,j}+\delta_- b_2^{-,j})(\upfy_j(r,z)\upfy_2^t(s,z)\sigma_1\mathds{1}_{\{0 < s \leq r\}}+\upfy_2(r,z)\upfy_j^t(s,z)\sigma_1\mathds{1}_{\{r < s\}}).
\end{split}
\end{align}
For simplicity, we suppressed the $z$-dependence from the notation on the right-hand side. 
We now set out to estimate the coefficients $b_\ell^{\pm,j}(z)$ and $\alpha_\pm(z),\beta_\pm(z),\gamma_\pm(z),\delta_\pm(z)$.  First we prove a lemma that states the relation between $\omega_{ij}^+$ and $\omega_{ij}^-$.
\begin{lemma}\label{lem:omegapm}
    For $\lambda\in\bbR\backslash\{0\}$
    \begin{align*}
        \omega_{i1}^{-}(\lambda)=\omega_{i1}^{+}(-\lambda), \quad \omega_{i2}^{-}(\lambda)=-\omega_{i2}^{+}(-\lambda),\quad \omega_{i3}^{-}(\lambda)=\omega_{i3}^{+}(-\lambda),\quad \omega_{i4}^{-}(\lambda)=-\omega_{i4}^{+}(-\lambda).
    \end{align*}
    Moreover,
    \begin{align*}
        \omega_{21}^+(-\lambda)=\omega_{21}^+(\lambda),\qquad \omega_{22}^+(-\lambda)=-\omega_{22}^+(\lambda).
    \end{align*}
\end{lemma}
\begin{proof}
    The relation between $\omega_{ij}^+$ and $\omega_{ij}^-$ is a result of the identities
    \begin{align*}
        &\uppsi_j^-(r,\lambda)=-\sigma_3\uppsi_j^+(r,-\lambda),\qquad \upfy_1(r,-\lambda)=\sigma_3\upfy_1(r,\lambda),\qquad \upfy_2(r,-\lambda)=-\sigma_3\upfy_2(r,\lambda),\\&\upfy_3(r,-\lambda)=\sigma_3\upfy_3(r,\lambda),\qquad \upfy_4(r,-\lambda)=-\sigma_3\upfy_4(r,\lambda).
    \end{align*}
    These identities follow from \eqref{eq:upfy12evenodd1}, \eqref{eq:upfy34evenodd1}, and \eqref{eq:Upsilonpmdef1}. For the claims about $\omega_{12}^+$ and $\omega_{22}^+$ we also use Corollary~\ref{cor:Upsilon2} which shows that $\uppsi_2^+(r,-\lambda)=\sigma_3\uppsi_2^+(r,\lambda)$.
\end{proof}
We recall that throughout this paper we work with a fixed string of small absolute constants $0 < \delta_0 \ll \epsilon_\infty \ll \epsilon_0 \ll 1$.
Now we fix $\epsilon_\infty \ll \epsilon \ll \epsilon_0$. For $0 < |z| \leq \delta_0$ we then define
\begin{equation} \label{equ:repsilon_definition}
    r_\epsilon(z) := \frac{\epsilon}{|z|}.
\end{equation}
Note that then all the estimates on $\upfy_j(r,z)$ and $\uppsi_j^\pm(r,z)$ from Sections~\ref{sec:Fs} and~\ref{sec:Psis} are valid for $r \in [r_\epsilon/10,10r_\epsilon]$.

In the next lemma we determine upper (and in many cases also lower) bounds on the coefficients $b_1^{\pm,j}(z)$, $1 \leq j \leq 4$, as well as $\alpha_\pm(z)$, $\beta_\pm(z)$, $\gamma_\pm(z)$, $\delta_\pm(z)$, $d_\pm(z)$ in the expressions \eqref{eq:G+2} and \eqref{eq:G-2} for the resolvent kernels. In view of the obtained upper bounds, observe that the coefficients of $\upfy_4(r,z)\upfy_1(s,z)^t$ and $\upfy_1(r,z)\upfy_4(s,z)^t$ in \eqref{eq:G+2} and \eqref{eq:G-2} may potentially be too singular to take the limit $b\to0^+$ in Proposition~\ref{prop:PI} inside the integral. However as we will see in Corollary~\ref{cor:fy14cancel1} below, these potentially singular coefficients actually drop out.
\begin{lemma} \label{lem:bellj1}
For $0 < |z| \leq \delta_0$ with $\pm \Im z > 0$ we have
\begin{align*}
\begin{split}
&|b_1^{\pm,1}|\simeq |z|^{\frac{1}{2}}\log|z|^{-1},\quad |b_3^{\pm,1}|\simeq |z|^{\frac{1}{2}}\log|z|^{-1}, \quad |b_1^{\pm,3}|\simeq |z|^{\frac{1}{2}}, \quad |b_3^{\pm,3}|\simeq |z|^{\frac{1}{2}},\quad |b_2^{\pm,4}|\simeq |z|^{-1},\\
&|b_2^{\pm,1}|\lesssim e^{-2\sqrt{2}r_\epsilon(z)},\quad |b_1^{\pm,2}|\lesssim e^{-2\sqrt{2}r_\epsilon(z)},\quad|b_2^{\pm,2}|\lesssim e^{-4\sqrt{2}r_\epsilon(z)},\quad |b_3^{\pm,2}|\lesssim e^{-2\sqrt{2}r_\epsilon(z)},\\
&|b_2^{\pm,3}|\lesssim e^{-2\sqrt{2}r_\epsilon(z)},\quad |b_1^{\pm,4}|\lesssim e^{\frac{\sqrt{2}}{2}r_\epsilon(z)}, \quad |b_3^{\pm,4}|\lesssim e^{\frac{\sqrt{2}}{2}r_\epsilon(z)},\\
& |\alpha_{\pm}|\simeq |z|^{-1},\quad |\beta_\pm|\lesssim e^{-2\sqrt{2}r_\epsilon(z)},\quad |\gamma_{\pm}|\lesssim e^{\frac{\sqrt{2}}{2}r_\epsilon(z)},\quad |\delta_\pm|\simeq |z|^{\frac{1}{2}},\quad |d_\pm|\simeq |z|^{-\frac{1}{2}}.
\end{split}
\end{align*}
\end{lemma}
\begin{proof}
We will prove the estimates for the $+$ sign, and the ones for the $-$ sign follow from Lemma~\ref{lem:omegapm} and~\eqref{eq:bomegasystem1} and~\eqref{eq:alphabetabellj} below. Recall the following notation for the columns of $F_1$ and $F_2$ introduced at the end of Subsections~\ref{sec:F1} and~\ref{sec:F2}:
\begin{align}\label{eq:upfyjnotation1}
\begin{split}
\upfy_1=\pmat{z\calG_2f\\f},\quad \upfy_2=\pmat{g\\z\calG_1g},\quad \upfy_3=\pmat{z\calG_2\tilf\\\tilf},\quad\upfy_4=\pmat{\tilg\\z\calG_1^\dagger\tilg}.
\end{split}
\end{align}
Estimates on $f$, $\tilf$, $g$, $\tilg$ follow from Propositions~\ref{lem:F1_1} and~\ref{lem:F2_1}. 
For $\uppsi_\ell^\pm$ we use the notation
\begin{align*}
\begin{split}
\uppsi_\ell^\pm=\pmat{\uppsi_\ell^{\pm,1}\\\uppsi_\ell^{\pm,2}},
\end{split}
\end{align*}
where the individual components and their derivatives can be estimated using Lemmas~\ref{lem:wjXj} and~\ref{lem:contr}.
We claim that $c_{ij}(z)=W[\upfy_i(z),\upfy)j(z)]$
satisfy 
\begin{align}\label{eq:cijclaim1}
\begin{split}
c_{12}=c_{23}=0, \quad |c_{14}|,|c_{34}|\lesssim e^{-4\sqrt{2}r_\epsilon(z)},
\quad c_{13},c_{24}\simeq 1.
\end{split}
\end{align}
It follows that, see~\eqref{eq:omegacij1},
\begin{align}\label{eq:bomegasystem1}
\begin{split}
&\omega_{\ell 1}^\pm=-c_{13}b_\ell^{\pm,3}-c_{14}b_\ell^{\pm,4},\\
&\omega_{\ell 2}^\pm=-c_{24}b_\ell^{\pm,4},\\
&\omega_{\ell 3}^\pm=c_{13}b_\ell^{\pm,1}-c_{34}b_\ell^{\pm,4},\\
&\omega_{\ell 4}^\pm=c_{14}b_\ell^{\pm,1}+c_{24}b_\ell^{\pm,2}+c_{34}b_\ell^{\pm,3},
\end{split}
\end{align}
and
\begin{align}\label{eq:alphabetabellj}
\begin{split}
&\alpha_{\pm}=W[\uppsi_2^\pm,\upfy_2]=\omega_{2 2}^\pm,\quad \beta_{\pm} = -W[\uppsi_2^\pm,\upfy_1]=-\omega_{21}^\pm,\\
&\gamma_{\pm}=-W[\uppsi_1^\pm,\upfy_2]=-\omega_{12}^\pm,\quad \delta_{\pm}=W[\uppsi_1^\pm,\upfy_1]=\omega_{11}^\pm,\\
&d_{\pm}=W[\uppsi_1^\pm,\upfy_1]W[\uppsi_2^\pm,\upfy_2]-W[\uppsi_2^\pm,\upfy_1]W[\uppsi_1^\pm,\upfy_2]=\omega_{11}^\pm\omega_{22}^\pm-\omega_{21}^\pm\omega_{12}^\pm.
\end{split}
\end{align}
To prove \eqref{eq:cijclaim1} we evaluate some of the Wronskians at $r=0$ and some at $r=r_\epsilon(z)$. By inspection using the formulas for $\calG_1$, $\calG_2$, $\tilde{\calG_1}$, $\calG_1^\dagger$ from \eqref{eq:calG1G2} and \eqref{eq:calG1G2*}, as well as \eqref{eq:F1series1} and \eqref{eq:F2series1} near $r=0$ we get 
\begin{align}\label{eq:fgftildegtilde1}
\begin{split}
&f(r)=c_fr^{\frac{3}{2}}+O(r^{\frac{5}{2}}),\quad \calG_2f(r)=\tilc_fr^{\frac{3}{2}}+O(r^{\frac{5}{2}}),\\
&g(r)=c_gr^{\frac{3}{2}}+O(r^{\frac{5}{2}}),\quad \calG_1g(r)=\tilc_gr^{\frac{7}{2}}+O(r^{\frac{9}{2}}),\\
&\tilf(r)=c_\tilf r^{-\frac{1}{2}}+O(r^{\frac{1}{2}}),\quad \calG_2\tilf(r)=\tilc_\tilf r^{\frac{3}{2}}\log r+O(r^{\frac{3}{2}}),\\
&\tilg(r)=c_\tilg r^{-\frac{1}{2}}+O(r^{\frac{1}{2}}),\quad \calG_1^\dagger\tilg(r)=\tilc_\tilg r^{-\frac{1}{2}}+O(r^{\frac{1}{2}}),
\end{split}
\end{align}
where $c_f$, $c_g$, etc., are constants and where the remainder terms $O(\ldots)$ enjoy symbol-type behavior under differentiation with respect to $r$.
By evaluating the Wronskians for small $r$ this already shows that $c_{12}=c_{23}=0$ and that $c_{13},c_{24}\simeq 1$. For the other Wronskians $c_{14}$ and $c_{34}$ we evaluate at $r=4r_\epsilon(z)$. The estimates from Propositions~\ref{lem:F1_1} and~\ref{lem:F2_1} then immediately give the desired results.

It remains to compute the Wronskians $W[\uppsi_\ell^+,\upfy_j]$ which we evaluate at an appropriate~$r$ in $[\frac{r_\epsilon(z)}{3},3r_\epsilon(z)]$. To estimate the resulting expressions we use Propositions~\ref{lem:F1_1} and~\ref{lem:F2_1} and Lemmas~\ref{lem:wjXj} and~\ref{lem:contr}. To simplify notation we will write $\uppsi_\ell$ for $\uppsi_\ell^+$ and $\omega_{jk}$ for $\omega_{jk}^+$ in the rest of this proof. We start with $\omega_{\ell3}$ given by
\begin{align}\label{eq:b11temppre1}
\begin{split}
\omega_{\ell3}=z\uppsi_\ell^1(r_\epsilon)(\partial_r\calG_2\tilf)(r_\epsilon)-z\partial_r\uppsi_\ell^1(r_\epsilon)\calG_2\tilf(r_\epsilon)+\partial_r\uppsi_\ell^2(r_\epsilon)\tilf(r_\epsilon)-\uppsi_\ell^2(r_\epsilon)\partial_r\tilf(r_\epsilon).
\end{split}
\end{align}
We claim that for $\ell=1$ this gives
\begin{align}\label{eq:b11temp1}
\begin{split}
|\omega_{13}|\simeq |z|^{\frac{1}{2}}\log|z|^{-1}.
\end{split}
\end{align}
Using Proposition~\ref{prop:Upsilon} and~\ref{lem:F2_1} and \eqref{eq:b11temppre1} one has as $z\to0$ 
\begin{equation}\label{eq:om13}
\begin{aligned}
    \omega_{13} &= - \Upsilon_1^2(r_\epsilon) \tilde f(r_\epsilon)+\Upsilon_1^1(r_\epsilon)\partial_r \tilde f(r_\epsilon)+\calO(|z|^{\frac{5}{2}}\log |z|^{-1}).
\end{aligned}
\end{equation}
 In view of Proposition~\ref{lem:F2_1}, Lemma~\ref{lem:contr}, and Lemma~\ref{lem:wjXj} (specifically \eqref{eq:Hankelsmall1}), the main contribution is from
\begin{equation}
    \begin{aligned}
        \Upsilon_1^2(r_\epsilon) \tilde f(r_\epsilon)-\Upsilon_1^1(r_\epsilon)\partial_r \tilde f(r_\epsilon) &= k_1 h_+'(k_1 r_\epsilon) f_2(r_\epsilon) - h_+ (r_\epsilon) f_2'(r_\epsilon) + \rho \\
        &= c_1 \sqrt{k_1(z)}\log r_\epsilon -c_1 k_1^\frac12 \log(k_1r_\epsilon) - c_2 k_1^\frac12 + \calO(k_1^{\frac52} r_\epsilon^2 \log\epsilon \log r_\epsilon) + \rho\\
        &=-(c_1\log k_1+c_2)k_1^{\frac{1}{2}}+ \calO(k_1^{\frac52} r_\epsilon^2 \log\epsilon \log r_\epsilon) + \rho.
    \end{aligned}
\end{equation}
The error $\rho$ satisfies
\[
|\rho| \lesssim \epsilon^2  |z|^{\frac12} \log |z|^{-1},
\]
which proves \eqref{eq:b11temp1}.
For $\ell=2$ evaluating \eqref{eq:b11temppre1} at $r=3r_\epsilon$, by a naive bound we get
\begin{align*}
\begin{split}
|\omega_{23}|\lesssim e^{-2\sqrt{2}r_\epsilon(z)}.
\end{split}
\end{align*}
For $\omega_{33}$ by a similar argument as for $\omega_{13}$ we have 
\begin{align*}
\begin{split}
|\omega_{33}|\simeq |z|^{\frac{1}{2}}\log |z|^{-1}.
\end{split}
\end{align*}
Next, for $\omega_{\ell4}$ we have
\begin{align}\label{eq:bell2-1}
\begin{split}
\omega_{\ell4}=\uppsi_\ell^1(r_\epsilon)\partial_r\tilg(r_\epsilon)-\partial_r\uppsi_\ell^1(r_\epsilon)\tilg(r_\epsilon)+z\partial_r\uppsi_\ell^2(r_\epsilon)\calG_1^\dagger \tilg(r_\epsilon)-z\uppsi_\ell^2(r_\epsilon)(\partial_r\calG_1^\dagger\tilg)(r_\epsilon).
\end{split}
\end{align}
For $\ell=1$, evaluating at $r=3r_\epsilon$ a naive estimate gives
\begin{align*}
\begin{split}
|\omega_{14}|\lesssim e^{-2\sqrt{2}r_\epsilon(z)}, \quad |\omega_{24}|\lesssim e^{-4\sqrt{2}r_\epsilon(z)}, \quad |\omega_{34}|\lesssim e^{-2\sqrt{2}r_\epsilon(z)}.
\end{split}
\end{align*}

Next, we write
\begin{align*}
\begin{split}
\omega_{\ell1}=z\uppsi_\ell^1(r_\epsilon)(\partial_r\calG_2f)(r_\epsilon)-z\partial_r\uppsi_\ell^1(r_\epsilon)\calG_2f(r_\epsilon)+\partial_r\uppsi_\ell^2(r_\epsilon)f(r_\epsilon)-\uppsi_\ell^2(r_\epsilon)\partial_rf(r_\epsilon).
\end{split}
\end{align*}
For $\ell=1$ we get
\begin{align*}
\begin{split}
|\omega_{11}|\simeq |z|^{\frac{1}{2}}.
\end{split}
\end{align*}
The argument here is similar to the one for \eqref{eq:b11temp1}, where the main contribution comes from replacing $\uppsi_1$ by the corresponding components of $w_1$ and $f$ by $f_1$ and is given by $c k_1(z)^{\frac{1}{2}}+\calO(|z|^{\frac{5}{2}})$ (see Lemma~\ref{lem:fjgj_1} and~\ref{lem:wjXj}) for a suitable constant $c \neq 0$.
For $\ell=2$, evaluating at $r=3r_\epsilon$, a naive estimate gives
\begin{align*}
\begin{split}
|\omega_{21}|\lesssim e^{-2\sqrt{2}r_\epsilon(z)}.
\end{split}
\end{align*}
For $\ell=3$ by the same argument as for $\ell=1$ we get 
\begin{align*}
\begin{split}
|\omega_{31}|\simeq |z|^{\frac{1}{2}}.
\end{split}
\end{align*}

Next, we write
\begin{align*}
\begin{split}
\omega_{\ell2}=\uppsi_\ell^1(r_\epsilon)\partial_rg(r_\epsilon)-\partial_r\uppsi_\ell^1(r_\epsilon)g(r_\epsilon)+z\partial_r\uppsi_\ell^2(r_\epsilon)\calG_1g(r_\epsilon)-z\uppsi_\ell^2(r_\epsilon)(\partial_r\calG_1g)(r_\epsilon).
\end{split}
\end{align*}
For $\ell=1$ we evaluate at $r=\frac{r_\epsilon}{3}$. A naive estimate then gives 
\begin{align*}
\begin{split}
|\omega_{12}|\lesssim e^{\frac{\sqrt{2}}{2}r_\epsilon(z)}.
\end{split}
\end{align*}
For $\ell=2$ we evaluate at $r=r_\epsilon$ and simply get
\begin{align*}
\begin{split}
|\omega_{22}|\simeq |z|^{-1}.
\end{split}
\end{align*}
Here the main contribution comes from replacing $\uppsi_2$ by $w_2$ and $g$ by $g_1$ (see Lemma~\ref{lem:fjgj_1} and~\ref{lem:wjXj}). Using $k_2=i\sqrt{2}+\calO(|z|^2)$ this gives $\frac{4\sqrt{2}ci}{c_{24}z}+\calO(|z|)$. For $\omega_{32}$ by a similar argument as for $\omega_{12}$ we have 
\begin{align*}
\begin{split}
|\omega_{32}|\lesssim e^{\frac{\sqrt{2}}{2}r_\epsilon(z)}.
\end{split}
\end{align*}
The estimates for $d_\pm$, $\alpha_\pm$, $\beta_\pm$, $\gamma_\pm$, and $\delta_\pm$ follow from those for $\omega_{\ell j}$ and~\eqref{eq:alphabetabellj}. For the bounds on $b_\ell^{\pm,j}$, we conclude from  \eqref{eq:cijclaim1} and the second line in  \eqref{eq:bomegasystem1}
\[
|\omega_{\ell2}|\simeq |b_\ell^4|,\quad |\omega_{\ell1}|\simeq |b_\ell^3|, \quad|\omega_{\ell3}|\simeq |b_\ell^1|,\quad |\omega_{\ell4}|\simeq |b_\ell^2|.
\]
Using the bounds on the $\omega$-coefficients we arrive at the esimtae for $b_\ell^4$. With these, the bounds for $b_\ell^1$ and $b_\ell^3$ can then be deduced from the first and third line of \eqref{eq:bomegasystem1}. The bounds for $b_\ell^1$, $b_\ell^3$, and $b_\ell^4$ together with the last line of \eqref{eq:bomegasystem1} then give the bonds for $b_\ell^2$.
\end{proof}
An important corollary of Lemma~\ref{lem:bellj1} is that the expressions \eqref{eq:G+2} and \eqref{eq:G-2} are regular enough to allow us to take the limit $b\to0^+$ in Proposition~\ref{prop:PI} inside the integral. Most notably, as mentioned earlier, the potentially singular coefficients of $\upfy_4(r,z)\upfy_1(s,z)$ and $\upfy_1(r,z)\upfy_4(s,z)$ in \eqref{eq:G+2} and \eqref{eq:G-2} cancel. 
\begin{cor} \label{cor:fy14cancel1}
For $0 < |z| \leq \delta_0$ with $\pm \Im (z) > 0$ we have 
\begin{align}\label{eq:G+3}
\begin{split}
\calG_{+}(r,s;z)&=\sum_{j=1}^3m_{+}^j(z)\big(\upfy_j(r,z)\upfy_1^t(s,z)\sigma_1\mathds{1}_{\{0 < s \leq r\}}+\upfy_1(r,z)\upfy_j^t(s,z)\sigma_1\mathds{1}_{\{r < s\}}\big)\\
&\quad+\sum_{j=1}^4n_{+}^j(z)\big(\upfy_j(r,z)\upfy_2^t(s,z)\sigma_1\mathds{1}_{\{0 < s \leq r\}}+\upfy_2(r,z)\upfy_j^t(s,z)\sigma_1\mathds{1}_{\{r < s\}}\big)
\end{split}
\end{align}
and
\begin{align}\label{eq:G-3}
\begin{split}
\calG_{-}(r,s;z)&=\sum_{j=1}^3m_{-}^j(z)\big(\upfy_j(r,z)\upfy_1^t(s,z)\sigma_1\mathds{1}_{\{0 < s \leq r\}}+\upfy_1(r,z)\upfy_j^t(s,z)\sigma_1\mathds{1}_{\{r < s\}}\big)\\
&\quad+\sum_{j=1}^4n_{-}^j(z)\big(\upfy_j(r,z)\upfy_2^t(s,z)\sigma_1\mathds{1}_{\{0 < s \leq r\}}+\upfy_2(r,z)\upfy_j^t(s,z)\sigma_1\mathds{1}_{\{r < s\}}\big),
\end{split}
\end{align}
where
\begin{equation} \label{equ:resolvent_kernels_z_coeff_bounds}
\begin{aligned}
&|m_\pm^1(z)|\lesssim \log|z|^{-1},\quad |m_\pm^2(z)|\lesssim e^{-\frac{3\sqrt{2}}{2}r_\epsilon(z)},\quad |m_\pm^3(z)|\lesssim 1,\\
&|n_\pm^1(z)|\lesssim  e^{-2\sqrt{2}r_\epsilon(z)},\quad |n_\pm^2(z)|\lesssim  e^{-4\sqrt{2}r_\epsilon(z)},\quad |n_\pm^3(z)|\lesssim e^{-2\sqrt{2}r_\epsilon(z)},\quad |n_\pm^4(z)|\lesssim 1.
\end{aligned}
\end{equation}
\end{cor}
\begin{proof}
The estimates on $m_\pm^i(z)$ and $n_\pm^i(z)$ follow directly from the statement of Lemma~\ref{lem:bellj1} and equations~\eqref{eq:G+3} and~\eqref{eq:G-3}. The more subtle point is that the (potentially singular) coefficient of $\upfy_4(r,z)\upfy_1^t(s,z)\sigma_1\mathds{1}_{\{0 < s \leq r\}}+\upfy_1(r,z)\upfy_4^t(s,z)\sigma_1\mathds{1}_{\{r < s\}}$ vanishes. The proof is almost identical for~\eqref{eq:G+3} and~\eqref{eq:G-3}, so we present the details only for~\eqref{eq:G+3}. Based on~\eqref{eq:G+2}, we need to show that
\begin{align}\label{eq:fy13cancel1}
 \alpha_+b_1^{+,4}+\gamma_{+}b_2^{+,4}=0.
\end{align}
From \eqref{eq:alphabetabellj},
\begin{align*}
    \alpha_{+}=\omega_{22}^+,\qquad \gamma_{+}=-\omega_{12}^+.
\end{align*}
On the other hand, from the second line of~\eqref{eq:bomegasystem1},
\begin{align*}
    b_{1}^{+,4}=-\frac{1}{c_{24}}\omega_{12}^+,\qquad b_{2}^{+,4}=-\frac{1}{c_{24}}\omega_{22}^+.
\end{align*}
Combining these identities we get
\begin{align*}
    \alpha_+b_1^{+,4}+\gamma_{+}b_2^{+,4}= -\frac{1}{c_{24}}\omega_{22}^+\omega_{12}^++\frac{1}{c_{24}}\omega_{12}^+\omega_{22}^+=0,
\end{align*}
proving the desired identity~\eqref{eq:fy13cancel1}.
\end{proof}
We also record the following corollary of the proof of Lemma~\ref{lem:bellj1}.
\begin{cor}\label{cor:omegaij}
For $0 < |z| \leq \delta_0$ with $\pm \Im(z) > 0$, we have
\begin{align*}
\begin{split}
&|\omega_{13}^\pm|, |\omega_{33}^\pm|\simeq |z|^{\frac{1}{2}}\log|z|^{-1},\quad |\omega_{11}^\pm|, |\omega_{31}^\pm|\simeq |z|^{\frac{1}{2}}, \quad |\omega_{22}^\pm|\simeq |z|^{-1}\\
&|\omega_{14}^\pm|, |\omega_{34}^\pm|\lesssim e^{-2\sqrt{2}r_\epsilon(z)}, \quad |\omega_{23}^\pm|\lesssim e^{-2\sqrt{2}r_\epsilon(z)},\quad 
|\omega_{24}^\pm|\lesssim e^{-4\sqrt{2}r_\epsilon(z)},\quad |\omega_{21}^\pm|\lesssim e^{-2\sqrt{2}r_\epsilon(z)},\\
&
|\omega_{12}^\pm|, |\omega_{32}^\pm|\lesssim e^{\frac{\sqrt{2}}{2}r_\epsilon(z)},
\quad|\omega_{41}^\pm|\lesssim e^{\frac{\sqrt{2}}{5}r_\epsilon(z)},\quad |\omega_{42}^\pm|\lesssim e^{\frac{\sqrt{2}}{5} r_\epsilon(z)}.
\end{split}
\end{align*}
\end{cor}
\begin{proof}
 Except for $\omega_{41}^\pm$ and $\omega_{42}^\pm$, these estimates are taken directly from the proof of Lemma~\ref{lem:bellj1} above. The estimates for $\omega_{41}^\pm$ and for $\omega_{42}^\pm$ are proved in a similar way by evaluating the corresponding Wronskians at $r=\frac{r_\epsilon}{6}$.
\end{proof}

\subsection{Computing the jump of the resolvent} \label{sec:thejump}

We are now in the position to return to the frequency localized evolution $e^{t\calL} P_I$, which was introduced in Definition~\ref{def:PIevol} for any interval $I \subseteq [-\delta_0,\delta_0]$ as a limit in the weak sense of testing against compactly supported functions $\phi, \psi \in L^2_r((0,\infty))$,
\begin{equation} \label{equ:localevol_recalled}
    \begin{aligned}
        \bigl\langle e^{t\calL} P_I \phi, \psi \bigr\rangle &:= \lim_{b\to0+} \frac{1}{2\pi i} \int_I e^{it\lambda} \bigl\langle \bigl[ e^{-bt} (i\calL-(\lambda +ib))^{-1} - e^{bt} (i\calL-(\lambda -ib))^{-1} \bigr] \phi, \psi \bigr\rangle \ud \lambda.
    \end{aligned}
    \end{equation}
In the next proposition we establish that the limit \eqref{equ:localevol_recalled} exists and we extract a distorted Fourier transform representation of the evolution at small energies by computing the jump of the resolvent across the essential spectrum.
The proof builds on the analysis of the resolvent kernels obtained in the previous subsections.
We emphasize that the bounds \eqref{equ:resolvent_kernels_z_coeff_bounds} on the coefficients appearing in the representations of the resolvent kernels \eqref{eq:G+3}--\eqref{eq:G-3} ensure that these coefficients are not too singular so that the limit $b\to0+$ can be moved inside the integral.

\begin{prop}\label{prop:PIStone1}
For any interval $I \subseteq [-\delta_0,\delta_0]$ and any compactly supported functions $\phi, \psi \in L^2_r((0,\infty))$, 
\begin{equation} \label{eq:StonePI3}
\begin{aligned}
    \bigl\langle e^{t\calL}P_I \phi, \psi \bigr\rangle &= \frac{1}{2\pi i} \int_I e^{it\lambda} \biggl\langle \int_0^\infty F_1(\cdot,\lambda) C(\lambda) F_1(s,\lambda)^t \sigma_1 \phi(s) \ud s, \psi(\cdot) \biggr\rangle_{L^2_r} \ud \lambda, 
\end{aligned}
\end{equation}
where
\begin{equation} \label{equ:Clambda_definition}
 C(\lambda) := \kappa(\lambda) D_-(\lambda)^{-1} \underline{e}_{11} D_+(\lambda)^{-t}
\end{equation}
with 
\begin{equation}
 \kappa(\lambda) = -2ik_1(\lambda)(1+\calO(|\lambda|)), \qquad \underline{e}_{11} := \begin{bmatrix} 1 & 0 \\ 0 & 0 \end{bmatrix},\qquad D_\pm(\lambda):=D_\pm(\lambda\pm i0). 
\end{equation}
\end{prop}
\begin{proof}
 In what follows, we use the continuity of $F_1(\cdot,z)$ and $F_2(\cdot,z)$ as $z$ crosses the real line. Since $F_1(\cdot,z)$, $F_2(\cdot,z)$ are given in terms of power series of $z$ for $r\leq r_0=\epsilon|z|^{-1}$, continuity holds in this region. The validity for all $r>0$ is a consequence of the standard theory of existence, uniqueness, and continuous dependence on parameters for ordinary differential equations. It follows from the expressions~\eqref{eq:G+2}--\eqref{eq:G-2} and the bounds \eqref{equ:resolvent_kernels_z_coeff_bounds} from Corollary~\ref{cor:fy14cancel1} that the integrand in \eqref{equ:localevol_recalled} is non-singular on the interval~$I$, and the limit can be taken inside the integral to give
 \begin{equation}\label{eq:StonePI2}
    \bigl\langle e^{t\calL} P_I \phi, \psi \bigr\rangle = \frac{1}{2\pi i} \int_I e^{it\lambda} \bigl\langle \bigl[ (i\calL-(\lambda +i0^{+}))^{-1}- (i\calL-(\lambda -i0^{+}))^{-1}\bigr] \phi, \psi \bigr\rangle \ud \lambda.
 \end{equation}
We denote the integral kernel of the jump of the resolvent across the essential spectrum by 
\begin{equation} 
    S(r,s; \lambda) := \bigl( i\calL - (\lambda +i0^{+}) \bigr)^{-1}(r,s) - \bigl( i\calL - (\lambda - i0^{+}) \bigr)^{-1}(r,s).
\end{equation}
Inserting \eqref{eq:G+} and \eqref{eq:G-}, we obtain
\begin{equation} \label{equ:definition_Srslambda}
 \begin{aligned}
  S(r,s; \lambda) &= i \Bigl( \Psi_+(r,\lambda) D_+(\lambda)^{-t} F_1(s,\lambda)^t - \Psi_-(r,\lambda) D_-(\lambda)^{-t} F_1(s, \lambda)^t \Bigr) \sigma_1 \mathds{1}_{[0 < s \leq r]} \\
  &\quad + i \Bigl( F_1(r,\lambda) D_+(\lambda)^{-1} \Psi_+(s,\lambda)^t - F_1(r,\lambda) D_-(\lambda)^{-1} \Psi_-(s,\lambda)^t \Bigr) \sigma_1 \mathds{1}_{[r \leq s < \infty]},
 \end{aligned}
\end{equation}
where $\Psi_{\pm}(\lambda):=\Psi_{\pm}(\lambda\pm i0)$ and similarly for $D_\pm(\lambda)$. Observe that the structure of \eqref{equ:definition_Srslambda} implies that
\begin{equation} \label{equ:Srslambda_transpose_identity}
 S(r,s;\lambda)^t = \sigma_1 S(s,r;\lambda) \sigma_1.
\end{equation}
On the other hand, for $z \neq 0$ we can write
\begin{equation} \label{equ:Psipm_in_terms_of_F12}
 \Psi_\pm(r,z) = F_2(r,z) N_\pm(z) + F_1(r,z) M_\pm(z),
\end{equation}
where the entries of $N_\pm$ can be determined in terms of the coefficients $b_\ell^j$ in Section~\ref{sec:connection}. 
Since $\calW[F_1(\cdot,z) M_\pm(z), F_1(\cdot,z)] = 0$ by the behavior of $F_1(r,z)$ as $r \searrow 0$, we have
\begin{equation}
 D_\pm(z) = \calW\bigl[ \Psi_\pm(\cdot,z), F_1(\cdot, z) \bigr] = N_\pm(z)^t \calW\bigl[F_2(\cdot,z), F_1(\cdot,z)\bigr].
\end{equation}
Let
\begin{equation}
 D(z) := \calW\bigl[ F_2(\cdot, z), F_1(\cdot, z) \bigr]
\end{equation}
and note that the entries of this matrix are precisely the coefficients $c_{ij}$ introduced in \eqref{eq:omegacij1}, with $\det D = c_{13}c_{24}-c_{23}c_{14}$. In particular, it follows from Lemma~\ref{lem:bellj1} and \eqref{eq:cijclaim1} that $D_{\pm}(z)$ and $D(z)$, and hence $N_\pm(z)$, are invertible for small $z\neq0$ with 
\begin{equation}
 D_\pm(z)^{-t} = N_\pm(z)^{-1} D(z)^{-t}.
\end{equation}
Then inserting \eqref{equ:Psipm_in_terms_of_F12} into \eqref{equ:definition_Srslambda}, we find that the terms with one factor $F_1$ and one factor $F_2$ cancel out. For this reason $S(r,s;\lambda)$ must be of the form
\begin{equation}
 \begin{aligned}
  S(r,s;\lambda) &= i F_1(r,\lambda) \calM_<(s,\lambda) \sigma_1 \mathds{1}_{[0 < s \leq r]} + i F_1(r,\lambda) \calM_>(s,\lambda) \sigma_1 \mathds{1}_{[0 < r \leq s]}.
 \end{aligned}
\end{equation}
By continuity, we must have $F_1(r,\lambda)(\calM_<(r,\lambda)-\calM_>(r,\lambda))=0$. But $\det F_1\neq0$, because otherwise we could find $c=c(r,\lambda)$ such that $cf=\lambda \calG_1g$ and $cg=\lambda\calG_2f$ in the notation \eqref{eq:upfyjnotation1}. But this is not possible in view of the small $r$ asymptotics in \eqref{eq:fgftildegtilde1}. It follows that $\calM_< = \calM_>$, whence $S(r,s;\lambda)$ is of the form
\begin{equation}
 S(r,s;\lambda) = i F_1(r,\lambda) \calM(s,\lambda) \sigma_1.
\end{equation}
From \eqref{equ:Srslambda_transpose_identity} we now conclude
\begin{equation}
 \begin{aligned}
  i \sigma_1 \calM(s,\lambda)^t F_1(r,\lambda)^t = S(r,s;\lambda)^t = \sigma_1 S(s,r;\lambda) \sigma_1 = i \sigma_1 F_1(s,\lambda) \calM(r,\lambda) \sigma_1^2,
 \end{aligned}
\end{equation}
whence
\begin{equation}
 \begin{aligned}
  F_1(s,\lambda)^{-1} \calM(s,\lambda)^t = \calM(r,\lambda) F_1(r,\lambda)^{-t} =: C(\lambda).
 \end{aligned}
\end{equation}
It follows that
\begin{equation} \label{equ:Srs_symmetric_form}
 \begin{aligned}
  S(r,s;\lambda) = i F_1(r,\lambda) C(\lambda) F_1(s,\lambda)^t \sigma_1,
 \end{aligned}
\end{equation}
and from \eqref{equ:Srslambda_transpose_identity} that
\begin{equation}
 C(\lambda)^t = C(\lambda).
\end{equation}

Now back to \eqref{equ:definition_Srslambda}, we have for $r > s$ that
\begin{equation}
 \begin{aligned}
  S(r,s; \lambda) &= i \Bigl( \Psi_+(r,\lambda) D_+(\lambda)^{-t} - \Psi_-(r,\lambda) D_-(\lambda)^{-t}  \Bigr) F_1(s, \lambda)^t \sigma_1.
 \end{aligned}
\end{equation}
Hence \begin{equation}
 \begin{aligned}
  F_1(r,\lambda) C(\lambda) = \Psi_+(r,\lambda) D_+(\lambda)^{-t} - \Psi_-(r,\lambda) D_-(\lambda)^{-t}.
 \end{aligned}
\end{equation}
Since in view of the large $r$ asymptotics of the columns of $\Psi_{+}$
\begin{equation}
 \calW\bigl[ \Psi_+(\cdot,\lambda) D_+(\lambda)^{-t}, \Psi_+(\cdot, \lambda) \bigr] = D_+(\lambda)^{-1} \calW\bigl[ \Psi_+(\cdot,\lambda), \Psi_+(\cdot,\lambda) \bigr] = 0,
\end{equation}
we conclude that
\begin{equation}
 \begin{aligned}
  \calW\bigl[ F_1(\cdot, \lambda) C(\lambda), \Psi_+(\cdot,\lambda) \bigr] &= - \calW\bigl[ \Psi_-(\cdot, \lambda) D_-(\lambda)^{-t}, \Psi_+(\cdot,\lambda) \bigr] \\
  &= - D_-(\lambda)^{-1} \calW\bigl[ \Psi_-(\cdot,\lambda), \Psi_+(\cdot, \lambda) \bigr].
 \end{aligned}
\end{equation}
Then from
\begin{equation}
 \begin{aligned}
  \calW\bigl[ F_1(\cdot, \lambda) C(\lambda), \Psi_+(\cdot,\lambda) \bigr] &= C(\lambda)^t \calW\bigl[ F_1(\cdot, \lambda), \Psi_+(\cdot,\lambda) \bigr] = - C(\lambda) D_+(\lambda)^t,
 \end{aligned}
\end{equation}
we arrive at
\begin{equation}\label{eq:CDpm1}
 \begin{aligned}
  C(\lambda) = D_-(\lambda)^{-1} \calW\bigl[ \Psi_-(\cdot,\lambda), \Psi_+(\cdot, \lambda) \bigr] D_+(\lambda)^{-t}.
 \end{aligned}
\end{equation}
Finally, we observe that (letting $r\to\infty$ and using that $k_1(\lambda) = - k_3(\lambda)$ for $\lambda \in \bbR$)
\begin{equation}
 \begin{aligned}
  \calW\bigl[ \Psi_-(\cdot,\lambda), \Psi_+(\cdot, \lambda) \bigr] = \begin{bmatrix} \kappa(\lambda) & 0 \\ 0 & 0 \end{bmatrix}
 \end{aligned}
\end{equation}
where with the notation from \eqref{eq:four roots},
\begin{equation}
 \kappa(\lambda) := -2 i k_1(\lambda) (1+\calO(|\lambda|)). 
\end{equation}
More precisely, 
\begin{align}\label{eq:kappaodd1}    
\kappa(\lambda)=W[\uppsi_1^-(\cdot,\lambda),\uppsi_1^+(\cdot,\lambda)]=W[\sigma_3\uppsi_1^+(\cdot,-\lambda),\uppsi_1^+(\cdot,\lambda)]
\end{align}
which using Lemmas~\ref{lem:Hankel0} and~\ref{lem:contr} gives the claimed result.
\end{proof}
It may seem surprising at first that \eqref{equ:Clambda_definition} involves the rank one matrix $\underline{e}_{11}$. However, as we will see in the next result, this has the effect of projecting away the unstable manifold near $r=\infty$. More concretely, expressing the columns of $F_1$ in terms of $\uppsi_j=\pmat{\Upsilon_{j,3}\\\Upsilon_{j,1}}$ we will see that the contributions of $\Upsilon_4$ drop out in \eqref{eq:StonePI3}. To state this result we need to recall and introduce some notation. First, recall the notation $\omega_{ij}(z) = W[\uppsi_i(\cdot,z),\fy_j(\cdot,z)]$ introduced in \eqref{eq:cijclaim1}. Since $\uppsi_j(\cdot,z)$ form a fundamental system, we can find coefficients $a_\ell^{\pm,j}(z)$ such that
\begin{align*}
\begin{split}
    \upfy_\ell(r,z) = \sum_{j=1}^4 a_\ell^{\pm,j}(z) \uppsi_j^\pm(r,z).
\end{split}
\end{align*}
Finally, recall the notation $s_{ij}^\pm(z) = W\bigl[ \uppsi_i^\pm(\cdot,z), \uppsi_j^\pm(\cdot,z) \bigr]$ from \eqref{eq:omegacij1}, and note that in view of the exponential decay of $\uppsi_2^\pm(r,z)$ and by Lemma~\ref{lem:contr}, for small $z \neq 0$ we have
\begin{align*}
\begin{split}
    |s_{24}^\pm(z)| \simeq |z|^{-2}, \qquad s_{2j}^\pm(z) = 0, \, \, j\neq4. 
\end{split}
\end{align*}
Next, we uncover the tensorial structure of the expression $F_1(r,\lambda) C(\lambda) F_1^t(s,\lambda)$ appearing in \eqref{eq:StonePI3}. 

\begin{lemma}\label{lem:varfystructure}
For $\lambda \in [-\delta_0, \delta_0]$ one has 
\begin{equation} \label{equ:F1CF1rs_tensor_structure}
\begin{split}
    F_1(r,\lambda)C(\lambda)F_1(s,\lambda)^t = \frac{\kappa(\lambda)}{d_{+}(\lambda) d_{-}(\lambda)}\upfy(r,\lambda)\upfy^t(s,\lambda),
\end{split}
\end{equation}
where $d_\pm(\lambda) := \det D_\pm(\lambda\pm i0)$ and
\begin{align} \label{eq:upfyexpansion1}
\begin{split}
    \upfy(\cdot,\lambda)=\omega_{22}^+(\lambda)\upfy_1(\cdot,\lambda)-\omega_{21}^+(\lambda)\upfy_2(\cdot,\lambda). 
\end{split}
\end{align}
Moreover, we have the representation 
\begin{equation} \label{eq:upfyexpansion2}
    \upfy(\cdot,\lambda) = \sum_{j=1}^3 \gamma_j^{+}(\lambda) \uppsi_j^{+}(\cdot,\lambda), 
\end{equation}
with 
\begin{align}\label{eq:aomega1}
\begin{split}
    \gamma_1^+ &:= s_{24}^+ \bigl( a_2^{+,4} a_1^{+,1} - a_1^{+,4} a_2^{+,1} \bigr) = (s_{13}^+)^{-1} \bigl(\omega_{21}^+ \omega_{32}^+ - \omega_{22}^+ \omega_{31}^+ \bigr), \\
    \gamma_3^+ &:= s_{24}^+ \bigl( a_2^{+,4}a_1^{+,3}-a_1^{+,4}a_2^{+,3} \bigr) = (s_{31}^+)^{-1} \bigl(\omega_{21}^+\omega_{12}^+ - \omega_{22}^+ \omega_{11}^+ \bigr), \\
    \gamma_2^+ &:= s_{24}^+ \bigl( a_2^{+,4} a_1^{+,2} - a_1^{+,4} a_2^{+,2} \bigr) = (s_{24}^{+})^{-1} \bigl(\omega_{21}^+ \omega_{42}^+ - \omega_{22}^+ \omega_{41}^+ \bigr) - s_{14}^+ \bigl(a_2^{+,4}a_1^{+,1}-a_1^{+,4} a_2^{+,1}\bigr) \\
    &\qquad \qquad \qquad \qquad \qquad \qquad \quad - s_{34}^+ \bigl(a_2^{+,4} a_1^{+,3} - a_1^{+,4} a_2^{+,3}\bigr),
\end{split}
\end{align}
and
\begin{equation} \label{equ:gamma_bounds}
    \bigl| \gamma_1^+(\lambda) \bigr| \simeq |\lambda|^{-\frac32}, \quad |\gamma_3^{+}(\lambda)| \simeq |\lambda|^{-\frac32}, \quad |\gamma_2^{+}(\lambda)| \lesssim e^{\frac{\sqrt{2}}{5} r_\epsilon(\lambda)}.
\end{equation}
\end{lemma}
\begin{proof}
Writing $D_\pm=\pmat{\omega_{11}^\pm&\omega_{12}^\pm\\\omega_{21}^\pm&\omega_{22}^\pm}$, from \eqref{eq:CDpm1} and the relation between $\omega_{ij}^+$ and $\omega_{ij}^-$ from Lemma~\ref{lem:omegapm} we see that
\begin{align*}
\begin{split}
C(\lambda)&=\frac{\kappa(\lambda)}{d_+(\lambda) d_-(\lambda)}\pmat{-\omega_{22}^+(\lambda)\omega_{22}^+(-\lambda)&\omega_{21}^+(\lambda)\omega_{22}^+(-\lambda)\\-\omega_{21}^+(-\lambda)\omega_{22}^+(\lambda)&\omega_{21}^+(\lambda)\omega_{21}^+(-\lambda)}\\
&=\frac{\kappa(\lambda)}{d_+(\lambda) d_-(\lambda)}\pmat{(\omega_{22}^+(\lambda))^2&-\omega_{21}^+(\lambda)\omega_{22}^+(\lambda)\\-\omega_{21}^+(\lambda)\omega_{22}^+(\lambda)&(\omega_{21}^+(\lambda))^2}.
\end{split}
\end{align*}
To simplify notation we will write $\omega_{ij}$ for $\omega_{ij}^+$ in the rest of this proof. Then writing 
\begin{equation*}
    F_1=\upfy_1\pmat{1&0}+\upfy_2\pmat{0&1} \quad \text{and} \quad F_1^t=\pmat{1\\0}\upfy_1^t+\pmat{0\\1}\upfy_2^t,
\end{equation*}
the identity \eqref{equ:F1CF1rs_tensor_structure} with $\upfy$ given by \eqref{eq:upfyexpansion1} follows by direct computation using the representation above for the matrix $C$.
The identity \eqref{eq:upfyexpansion2} follows from the observation that $s_{2j}^+=0$ unless $j=4$. 
Suppressing the $+$ sign from the notation in the remainder of the proof, we then infer from \eqref{eq:upfyexpansion2} that
\begin{align*}
\begin{split}
\omega_{21}\omega_{k2}-\omega_{22}\omega_{k1}=s_{24}\sum_{j=1}^3s_{jk}(a_{2}^4a_{1}^j-a_{1}^4a_{2}^j).
\end{split}
\end{align*}
Since $s_{12}=s_{32}=s_{11}=s_{33}=0$, we get
\begin{align*}
\begin{split}
    \gamma_1 &= s_{24} \bigl( a_2^4a_1^1-a_1^4a_2^1 \bigr) = s_{13}^{-1}(\omega_{21}\omega_{32}-\omega_{22}\omega_{31}), \\
    \gamma_3 &= s_{24} \bigl( a_2^4a_1^3-a_1^4a_2^3 \bigr) = s_{31}^{-1}(\omega_{21}\omega_{12}-\omega_{22}\omega_{11}), \\
    \gamma_2 &= s_{24} \bigl( a_2^4a_1^2-a_1^4a_2^2 \bigr) = s_{24}^{-1} \bigl(\omega_{21}\omega_{42}-\omega_{22}\omega_{41}\bigr) - s_{14}\bigl( a_2^4a_1^1-a_1^4a_2^1 \bigr) - s_{34}\bigl(a_2^4a_1^3-a_1^4a_2^3\bigr).
\end{split}
\end{align*}
In view of the large $r$ asymptotics of $\uppsi_j$ we have $|s_{13}|\simeq |\lambda|$ and $|s_{24}| \simeq |\lambda|^{-2}$. The asserted estimates $|\gamma_1| \simeq |\gamma_3| \simeq |\lambda|^{-\frac32}$ then follow from Corollary~\ref{cor:omegaij}.
In order to obtain a bound for $\gamma_2$, we use Proposition~\ref{prop:Upsilon} and Lemma~\ref{lem:contr} to evaluate the Wronskians $s_{14}$ and $s_{34}$ at $r=r_\epsilon/6$, which gives 
\begin{align*}
\begin{split}
|s_{14}|, |s_{34}| \lesssim e^{\sqrt{2}r_\epsilon/6}.
\end{split}
\end{align*}
It then follows from the previous computations and Corollary~\ref{cor:omegaij} that $|\gamma_2| \lesssim e^{\frac{\sqrt{2}}{5} r_\epsilon}$, as claimed.
This finishes the proof of the lemma.
\end{proof}

In view of Lemma~\ref{lem:varfystructure} the contribution of $\upfy_1$ to $F_1(r,\lambda)C(\lambda)F_1(s,\lambda)^t$ is $\frac{\kappa (\omega_{22}^+)2}{d_+ d_-}\upfy_1(r,\lambda)\upfy_1^t(s,\lambda)$. The estimates for $\omega_{22}^+$, $d_\pm$,  $\kappa$, and $\upfy_1$ then show that in the region $\{r\lesssim |\lambda|^{-1}\}$  the leading order is $r^{\frac{1}{2}}s^{\frac{1}{2}}$. This should be compared with the flat Fourier transform where the corresponding term is of order $(\lambda r)^{\frac{1}{2}}(\lambda s)^{\frac{1}{2}}$. A similar computation can be done for $\{r\gtrsim |\lambda|^{-1}\}$ and considering the contribution of $\uppsi_1$ and $\uppsi_3$ instead of $\upfy_1$. This leads to a leading order behavior which is a linear combination of $\calO(|\lambda|^{-1})e^{\pm k_1(\lambda)r}$. To get a more favorable estimate we need to compute the expression $\frac{\kappa(\lambda) }{d_+(\lambda) d_-(\lambda)}\upfy(r,\lambda)\upfy^t(s,\lambda)$ more carefully and observe a cancellation between $\pm\lambda >0$. This is achieved in the next two lemmas.
\begin{lemma} \label{lem:pmlambda1}
There exists a constant $c \neq 0$ such that for all $\lambda \in [-\delta_0,\delta_0]$,
\begin{align}\label{eq:pmcancel1}
\begin{split}
    \frac{\kappa(\lambda)}{d_{+}(\lambda) d_{-}(\lambda)} =\frac{W\bigl[\uppsi_1^+(\cdot,\lambda),\sigma_3\uppsi_1^+(\cdot,-\lambda)\bigr]}{d_{+}(\lambda) d_{+}(-\lambda)}= c \lambda^2 \mathrm{sgn}(\lambda) + \calO\bigl(|\lambda|^3\bigr).
\end{split}
\end{align}
\end{lemma}
\begin{proof}
Recall from \eqref{eq:kappaodd1} that
\begin{align*}
\kappa(\lambda) = -W\bigl[\uppsi_1^+(\cdot,\lambda),\sigma_3\uppsi_1^+(\cdot,-\lambda)\bigr],
\end{align*}
which shows that $\kappa(\lambda)$ is odd and $|\kappa(\lambda)|\simeq |\lambda|$. On the other hand, writing $D_\pm=\pmat{\omega_{11}^\pm&\omega_{12}^\pm\\\omega_{21}^\pm&\omega_{22}^\pm}$ and using Lemma~\ref{lem:omegapm}, we see that
\begin{align*}
    d_{-}(\lambda)=\omega_{11}^-(\lambda)\omega_{22}^-(\lambda)-\omega_{21}^-(\lambda)\omega_{12}^-(\lambda)=-\omega_{11}^+(-\lambda)\omega_{22}^+(-\lambda)+\omega_{21}^+(-\lambda)\omega_{12}^+(-\lambda)=-d_{+}(-\lambda).  
\end{align*}
Combining this with the expression for $\kappa$ gives the first equality in \eqref{eq:pmcancel1}. Since the resulting expression is an odd function of $\lambda$, the second equality in \eqref{eq:pmcancel1} then follows from the bounds on~$\omega_{ij}$ in Corollary~\ref{cor:omegaij}.
\end{proof}
Lemma~\ref{lem:pmlambda1} already suffices to observe the desired cancellation between positive and negative $\lambda$ for $r\lesssim |\lambda|^{-1}$. For large $r \gtrsim |\lambda|^{-1}$ we need to understand the structure of the sum $\sum_{j=1}^3 \gamma_j^{+}(\lambda) \uppsi_j^+(r,\lambda)$ in \eqref{eq:upfyexpansion2} better. This is the content of the next lemma. 
\begin{lemma} \label{lem:upfylarger1}
For $\lambda \in [-\delta_0,\delta_0]$ we have 
\begin{equation} \label{equ:upfylarger1}
\begin{aligned}
    \upfy(\cdot,\lambda) &= \mu_1(\lambda) \big( \omega_{11}^+(\lambda)\sigma_3\uppsi_1^+(\cdot,-\lambda)+\omega_{11}^+(-\lambda)\uppsi_1^+(\cdot,\lambda) \big) \\
    &\quad + \mu_2(\lambda) \sigma_3 \uppsi_1^+(\cdot,-\lambda) + \mu_3(\lambda) \uppsi_1^+(\cdot,\lambda) + \mu_4(\lambda) \uppsi_2^+(r,\lambda),
\end{aligned}
\end{equation}
where 
\begin{equation} \label{equ:mu_asymptotics_and_bounds}
    \begin{aligned}
        \mu_1(\lambda) = c_1 \lambda^{-2} + \calO\bigl(|\lambda|^{-1}\bigr), \quad |\mu_2(\lambda)| \lesssim e^{-\sqrt{2}r_\epsilon(\lambda)} , \quad |\mu_3(\lambda)| \lesssim e^{-\sqrt{2}r_\epsilon(\lambda)}, \quad |\mu_4(\lambda)| \lesssim e^{\frac{\sqrt{2}}{5}r_\epsilon(\lambda)}
    \end{aligned}
\end{equation}
for some constant $c_1 \neq 0$.
\end{lemma}
\begin{proof}
We suppress the $+$ signs from the notation in this proof. Recall from \eqref{eq:aomega1} that
\begin{align*}
\begin{split}
\gamma_1 = s_{13}^{-1}(\omega_{21}\omega_{32}-\omega_{22}\omega_{31}), \qquad \gamma_3 = s_{31}^{-1}(\omega_{21}\omega_{12}-\omega_{22}\omega_{11}).
\end{split}
\end{align*}
Since $s_{13}=-s_{31}$, we have for $\lambda \in \bbR$,
\begin{align}\label{eq:upfylargertemp1}
\begin{split}
\upfy(r,\lambda)&=s_{13}^{-1}(\lambda)\omega_{22}(\lambda)\big(\omega_{11}(\lambda)\uppsi_3(r,\lambda)-\omega_{31}(\lambda)\uppsi_{1}(r,\lambda)\big)\\
&\quad-s_{13}^{-1}(\lambda)\omega_{21}(\lambda)\big(\omega_{12}(\lambda)\uppsi_3(r,\lambda)-\omega_{32}(\lambda)\uppsi_{1}(r,\lambda)\big)+\gamma_2(\lambda)\uppsi_2(r,\lambda).
\end{split}
\end{align}
We focus on the expressions
\begin{align*}
\begin{split}
 \uppsi(r,\lambda):=\omega_{11}(\lambda)\uppsi_3(r,\lambda)-\omega_{31}(\lambda)\uppsi_{1}(r,\lambda), \qquad \widetilde{\uppsi}(r,\lambda):=\omega_{12}(\lambda)\uppsi_3(r,\lambda)-\omega_{32}(\lambda)\uppsi_{1}(r,\lambda).
\end{split}
\end{align*}
Observe that for any $\lambda\in\bbR\backslash{\{0\}}$, 
\begin{equation}
    \bigl\{ \uppsi_1(\cdot,\lambda), -\sigma_3 \uppsi_1(\cdot,-\lambda), \uppsi_2(\cdot,\lambda),\uppsi_4(\cdot,\lambda) \bigr\}
\end{equation}
form a fundamental system for the equation $i\calL\Psi=\lambda\Psi$.
Therefore, for each $\lambda\in\bbR\backslash\{0\}$ we can write
\begin{align*}
\begin{split}
\uppsi_3(\cdot,\lambda) = c_1(\lambda)\uppsi_1(\cdot,\lambda)+c_2(\lambda) (-\sigma_3) \uppsi_1(\cdot,-\lambda)+c_3(\lambda)\uppsi_2(\cdot,\lambda)+c_4(\lambda)\uppsi_4(\cdot,\lambda).
\end{split}
\end{align*}
Taking Wronskians with $\uppsi_2(\cdot,\lambda)$, we see that only the $\uppsi_4(\cdot,\lambda)$ has a non-zero contribution. Hence, $c_4=0$ and we have 
\begin{align}\label{eq:upfylargertemp2}
\begin{split}
\uppsi_3(\cdot,\lambda)=c_1(\lambda)\uppsi_1(\cdot,\lambda)+c_2(\lambda) (-\sigma_3) \uppsi_1(\cdot,-\lambda)+c_3(\lambda)\uppsi_2(\cdot,\lambda).
\end{split}
\end{align}
Recall that $\upfy_1(\cdot,\lambda) = \pmat{\lambda\calG_2f(\cdot,\lambda)\\f(\cdot,\lambda)}$ with $f(\cdot,\lambda)$ and $\calG_2f(\cdot,\lambda)$ both even in $\lambda$. It follows that $W\bigl[(-\sigma_3) \uppsi_1(\cdot,-\lambda),\upfy_1(\cdot,\lambda)\bigr] = \omega_{11}(-\lambda)$. This shows that
\begin{align*}
\begin{split}
\omega_{31}(\lambda)=c_1(\lambda)\omega_{11}(\lambda)+c_2(\lambda)\omega_{11}(-\lambda)+c_3(\lambda)\omega_{21}(\lambda),
\end{split}
\end{align*} 
and thus,
\begin{flalign}
\uppsi(\cdot,\lambda) = c_2(\lambda) \Bigl( \omega_{11}(\lambda)(-\sigma_3)\uppsi_1(\cdot,-\lambda)-\omega_{11}(-\lambda)\uppsi_1(\cdot,\lambda) \Big) + c_3(\lambda) \Big( \omega_{11}(\lambda)\uppsi_2(\cdot,\lambda)-\omega_{21}(\lambda)\uppsi_1(\cdot,\lambda) \Big). \label{eq:upfylargertemp3}&&
\end{flalign}
Similarly, since $\upfy_2(\cdot,\lambda)=\pmat{g(\cdot,\lambda)\\\lambda\calG_1g(\cdot,\lambda)}$ with $g(\cdot,\lambda)$ and $\calG_1g(\cdot,\lambda)$ even in $\lambda$, we conclude that  
\begin{align*}
\begin{split}
\omega_{32}(\lambda)=c_1(\lambda)\omega_{12}(\lambda)-c_2(\lambda)\omega_{12}(-\lambda)+c_3(\lambda)\omega_{22}(\lambda).
\end{split}
\end{align*}
A similar computation as above then shows that
\begin{flalign}
\widetilde{\uppsi}(\cdot,\lambda)=c_2(\lambda)\big(\omega_{12}(\lambda)\sigma\uppsi_1(\cdot,-\lambda)-\omega_{12}(-\lambda)\uppsi_1(\cdot,\lambda)\big)+c_3(\lambda)\big(\omega_{12}(\lambda)\uppsi_2(\cdot,\lambda)-\omega_{22}(\lambda)\uppsi_1(\cdot,\lambda)\big).\label{eq:upfylargertemp4}&&
\end{flalign}
Plugging these expressions into \eqref{eq:upfylargertemp1}, we arrive at the identity \eqref{equ:upfylarger1} with
\begin{equation}
    \begin{aligned}
        \mu_1(\lambda) &:= -s_{13}^{-1}(\lambda) \omega_{22}(\lambda) c_2(\lambda), \\
        \mu_2(\lambda) &:= s_{13}^{-1}(\lambda) \omega_{21}(\lambda) c_2(\lambda) \omega_{12}(\lambda), \\ 
        \mu_3(\lambda) &:= s_{13}^{-1}(\lambda) \omega_{21}(\lambda) c_2(\lambda) \omega_{12}(-\lambda), \\ 
        \mu_4(\lambda) &:= \gamma_2(\lambda)-s_{13}^{-1}(\lambda)c_3(\lambda) \bigl( \omega_{21}(\lambda) \omega_{12}(\lambda) - \omega_{22}(\lambda) \omega_{11}(\lambda) \bigr).
    \end{aligned}
\end{equation}  
Next, we determine $c_2(\lambda)$ and $c_3(\lambda)$. For $c_2(\lambda)$ we can either use Wronskians or simply look at the leading order asymptotics. For instance, by Proposition~\ref{prop:Upsilon} and Lemma~\ref{lem:contr}
\begin{align*}
\begin{split}
\uppsi_3^2(r,\lambda) = h_{-}(k_1(\lambda)r) + \calO(|\lambda|) h_+(k_1(\lambda)r) + h_3(r,\lambda),
\end{split}
\end{align*}
where $|h_3(\lambda,r)|\lesssim r^{-1}$ for large $r$. This can be seen from writing the first integral in the definition of $\calF_3(r,\lambda, \Upsilon_3^+(r;\lambda))$ in \eqref{eq:Upsilonapp3} as
\begin{align*}
\begin{split}
    w^+_1(r;\lambda)\int_{r_\infty}^{\infty}h_{-}(k_1(\lambda)s)\alpha_3(s;\lambda)\ud s-\int_{r}^{\infty}w^+_1(r;\lambda) h_{-}(k_1(\lambda)s) \alpha_3(s;\lambda) \ud s,
\end{split}
\end{align*}
and using similar estimates as those in the proof of Lemma~\ref{lem:contr}. Similarly,
\begin{align*}
\begin{split}
    \uppsi_1^2(r,\lambda) = h_{+}(k_1(\lambda)r) + \calO(|\lambda|) h_-(k_1(\lambda)r) + h_1(r,\lambda),
\end{split}
\end{align*}
where $|h_1(r,\lambda)|\lesssim r^{-1}$ for large $r$. Since $\uppsi_2^2(r,\lambda)$ is exponentially decaying as $r\to\infty$, and in view of the asymptotics of $h_\pm$ from Lemma~\ref{lem:Hankel0}, we conclude that
\begin{equation} \label{equ:c2_asymptotics}
\begin{split}
c_2(\lambda)=1+\calO(|\lambda|),\qquad \lambda\in \bbR\backslash\{0\}.
\end{split}
\end{equation}
To compute $c_3(\lambda)$ we evaluate the Wronskian of the two sides of \eqref{eq:upfylargertemp2} with $\uppsi_4(\cdot,\lambda)$ at $r=\frac{r_\epsilon(\lambda)}{6}$ to conclude that $|c_3(\lambda)|\lesssim e^{\frac{\sqrt{2}}{5}r_\epsilon(\lambda)}$. 

Finally,  note that by Corollary~\ref{cor:omegaij} and Lemma~\ref{lem:omegapm} we have $\omega_{22}(\lambda)=a\lambda^{-1}+\calO(1)$ for a suitable constant $a$, and in view of Proposition~\ref{prop:Upsilon} and Lemmas~\ref{lem:contr} and~\ref{lem:Hankel0}, we have   $s_{13}(\lambda) = b\lambda+\calO(\lambda^2)$ for a suitable constant $b \neq 0$. Together with \eqref{equ:c2_asymptotics}, it follows that $\mu_1(\lambda) = c_1 \lambda^{-2} + \calO(|\lambda|^{-1})$ for some constant $c_1 \neq 0$, as claimed in \eqref{equ:mu_asymptotics_and_bounds}. 
The other estimates on $|\mu_j(\lambda)|$, $j = 2, 3, 4$, asserted in \eqref{equ:mu_asymptotics_and_bounds}, follow from the preceding computations, Corollary~\ref{cor:omegaij}, and \eqref{equ:gamma_bounds}.
\end{proof}
\begin{rem}
The presence of $c_1(\lambda)\uppsi_1(\cdot,\lambda)$ in the expansion \eqref{eq:upfylargertemp2} for $\uppsi_3(\cdot,\lambda)$ is the manifestation of the Stokes phenomenon. In fact, using Lemma~\ref{lem:Hankel0} and a similar computation as the one used to find $c_2(\lambda)$ in the proof of Lemma~\ref{lem:upfylarger1}, we can find that $c_1(\lambda)=\calO(|\lambda|)$ for $\lambda>0$ and $c_1(\lambda)=2i+\calO(|\lambda|)$ for $\lambda<0$. The fact that $c_1(\lambda)$ eventually drops out in the expressions~\eqref{eq:upfylargertemp3} and~\eqref{eq:upfylargertemp4} for $\uppsi$ and $\tilde{\uppsi}$ shows that the Stokes phenomenon does not affect the final expressions for the distorted Fourier basis.
\end{rem}

\section{Proof of Theorem~\ref{thm:summary}} \label{sec:proofofthm}

In this section we establish the proof of Theorem~\ref{thm:summary}.
We recall that throughout this paper we work with a fixed string of small absolute constants $0 < \delta_0 \ll \epsilon_\infty \ll \epsilon_0 \ll 1$.
We begin by collecting a few technical lemmas that will be used repeatedly.
Their statements involve a smooth even cut-off function $\chi_0 \in C_c^\infty(\bbR)$ satisfying $\chi_0(x) = 1$ for $|x| \leq \epsilon/2$ and $\chi_0(x) = 0$ for $|x| \geq \epsilon$, where $\epsilon_\infty \ll \epsilon \ll \epsilon_0$ is as defined in \eqref{equ:repsilon_definition}. We set $\chi_1 := 1 - \chi_0$. 

The next lemma furnishes a decomposition of the distorted Fourier basis element $\uppsi_1^+(r,\lambda)$ into a leading order term and a remainder term with improved $L^2_r$ integrability properties.
\begin{lemma} \label{lem:decomposition_uppsioneplus}
    We have the decomposition
    \begin{equation}
        \uppsi_1^+(r,\lambda) = \begin{bmatrix}
            \frac{ik_1(\lambda)^2}{\lambda} h_+\bigl(k_1(\lambda)r\bigr) \\
            h_+\bigl(k_1(\lambda)r\bigr)
        \end{bmatrix}
        +
        \Upgamma(r,\lambda)
    \end{equation}
    with 
    \begin{equation} \label{equ:L2bound_remainder}
        \bigl\| \Upgamma(\cdot,\lambda) \chi_1(\lambda \cdot) \bigr\|_{L^2_r} \lesssim |\lambda|^{\frac32}.
    \end{equation}
\end{lemma}
\begin{proof}   
Recall from \eqref{eq:Upsilonapp1}, \eqref{equ:Upsilon1_integral_equation}, and \eqref{equ:uppsi_definition} that 
\begin{equation}
    \begin{aligned}
        \uppsi_1^+(r,\lambda) &= \begin{bmatrix} w_{1,3}^+(r,\lambda) \\ w_{1,1}^+(r,\lambda) \end{bmatrix} 
                                 + 
                                 \begin{bmatrix} \calF_{1,3}\bigl(r, \lambda; \Upsilon_1^+(r;\lambda)\bigr) \\ \calF_{1,1}\bigl(r, \lambda; \Upsilon_1^+(r;\lambda)\bigr) \end{bmatrix}, 
    \end{aligned}   
\end{equation}
where for $j = 1, 3,$
\begin{equation}
    \begin{aligned}
        \calF_{1,j}\bigl(r, \lambda; \Upsilon_1^+(r;\lambda)\bigr) &= - \int_r^\infty w_{1,j}^+(r,\lambda) h_-\bigl(k_1(\lambda)s\bigr) \alpha_1(s;\lambda)  \ud s + \int_{r_\infty}^r w^+_{2,j}(r;\lambda) h_{-}\bigl(k_2(\lambda)s\bigr)\beta_1(s;\lambda)\ud s \\
        &\quad \quad + \int_r^\infty w^+_{3,j}(r;\lambda) h_{+}\bigl(k_1(\lambda)s\bigr)\alpha_1(s;\lambda)\ud s+\int_r^\infty w^+_{4,j}(r;\lambda) h_{+}\bigl(k_2(\lambda)s\bigr)\beta_1(s;\lambda) \ud s, \\ 
        &=: R_{1,j}^1(r,\lambda) + R_{1,j}^2(r,\lambda) + R_{1,j}^3(r,\lambda) + R_{1,j}^4(r,\lambda).
    \end{aligned}
\end{equation}
By direct integration we infer from \eqref{eq:alphabeaell1} that
\begin{equation}
    \bigl|\alpha_1(r;\lambda)\bigr| \lesssim |\lambda| r^{-2}, \quad \bigl|\beta_1(r;\lambda)\bigr| \lesssim |\lambda|^2 r^{-2}, \quad r \geq r_\infty.
\end{equation}
Using the preceding bounds, we obtain by direct computation that for $r \geq r_\infty$,
\begin{equation}
    \begin{aligned}
        \bigl|R_{1,1}^1(r,\lambda)\bigr| &\lesssim |\lambda| r^{-1}, \\
        \bigl|R_{1,1}^2(r,\lambda)\bigr| &\lesssim |\lambda|^2r^{-2}\\
        \bigl|R_{1,1}^3(r,\lambda)\bigr| &\lesssim |\lambda| r^{-1}, \\ 
        \bigl|R_{1,1}^4(r,\lambda)\bigr| &\lesssim |\lambda|^2 r^{-2},
    \end{aligned}
\end{equation}
as well as 
\begin{equation}
    \begin{aligned}
        \bigl|R_{1,3}^1(r,\lambda)\bigr| &\lesssim |\lambda|^2 r^{-1}, \\
        \bigl|R_{1,3}^2(r,\lambda)\bigr| &\lesssim |\lambda|r^{-2}\\
        \bigl|R_{1,3}^3(r,\lambda)\bigr| &\lesssim |\lambda|^2 r^{-1}, \\ 
        \bigl|R_{1,3}^4(r,\lambda)\bigr| &\lesssim |\lambda| r^{-2}.
    \end{aligned}
\end{equation}
Here to estimate $R_{1,1}^2(r,\lambda)$ and $R_{1,3}^2(r,\lambda)$ for $r>2r_\infty$, we have divided the region of integration into $[r_\infty,r/2)$ and $[r/2,r]$. It follows by direct integration that
\begin{equation}
    \bigl\| R_{1,j}^\ell(\cdot,\lambda) \chi_1(\lambda \cdot) \bigr\|_{L^2_r} \lesssim |\lambda|^{\frac32}, \quad j = 1, 3, \quad 1\leq \ell\leq 4,
\end{equation}
as claimed.
\end{proof}

Next, we record a simple $L^2_r$-bound for the exponentially decaying distorted Fourier basis element $\psi_2^+(r,\lambda)$.
\begin{lemma} \label{lem:basic_bound_L2_uppsi2}
    We have uniformly for all $0 < |\lambda| \leq \delta_0$ that 
    \begin{equation} \label{equ:basic_bound_L2_uppsi2}
        \bigl\| \uppsi_2^+(\cdot, \lambda) \chi_1(\lambda \cdot) \bigr\|_{L^2_r} \lesssim e^{-\sqrt{2} \frac{r_\epsilon}{3}}.
    \end{equation}
\end{lemma}
\begin{proof}
The asserted estimate follows by direct integration from the observation that $\chi_1(\lambda r) = 0$ for $r \leq \frac{r_\epsilon}{2}$ by definition of the cut-off $\chi_1$ and that $|\uppsi_2^+(r,\lambda)| \lesssim |\lambda|^{-1}e^{-\sqrt{2}r}$ for $r \geq 10^{-3} r_\epsilon$ by Proposition~\ref{prop:Upsilon}.
\end{proof}

Now, we establish $L^2_r([0,\infty)) \to L^2_\lambda([0,\delta_0])$ bounds for the Fourier integral operators induced by the modified Hankel functions $h_+(\pm k_1(\lambda) r)$. 

\begin{lemma} \label{lem:L2bound_hplus}
    Let $0 < \delta_0 \ll 1$ be as in the statement of Theorem~\ref{thm:summary}. Then we have 
    \begin{equation} \label{equ:L2bound_hplus}
        \biggl( \int_0^{\delta_0} \biggl| \int_0^\infty h_+\bigl(\pm k_1(\lambda) r\bigr) \chi_1(\lambda r) v(r) \ud r \biggr|^2 \ud \lambda \biggr)^{\frac12} \lesssim \|v\|_{L^2}.
    \end{equation}
    It follows that 
    \begin{equation} \label{equ:L2bound_uppsioneplus} 
        \Bigl\| \bigl\langle \uppsi_1^+(\cdot, \pm \lambda) \chi_1(\lambda \cdot), \bmu(\cdot) \bigr\rangle \Bigr\|_{L^2_\lambda([0,\delta_0])} \lesssim \|\bmu\|_{L^2_r}.
    \end{equation}
\end{lemma}
\begin{proof}
    We only discuss the case of $h_+(+k_1(\lambda)r)$, the other one being analogous. 
    Define the operator
    \begin{equation}
        (Tv)(\lambda) := \int_0^\infty e^{ik_1(\lambda)r} \frakm(\lambda, r) \chi_1(\lambda r) v(r) \ud r, \quad 0 \leq \lambda \leq \delta_0,
    \end{equation}
    with
    \begin{equation}
        \frakm(\lambda,r) := e^{-ik_1(\lambda)r} h_+\bigl(k_1(\lambda)r\bigr).
    \end{equation}
    Then the estimate \eqref{equ:L2bound_hplus} in the statement of the lemma asserts that $T$ is bounded as an operator $T \colon L^2_r([0,\infty)) \to L^2_\lambda([0,\delta_0))$. 
    We intend to use the Cotlar-Stein lemma to prove this.
    Setting $\varphi(y) := \chi_0(y/2) - \chi_0(y)$, we decompose the operator $T = \sum_{j=1}^\infty T_j$ dyadically,
    \begin{equation}
        (T_j v)(\lambda) := \int_0^\infty e^{ik_1(\lambda)r} \frakm(\lambda, r) \varphi\bigl( 2^{-j} \lambda \bigr) v(r) \ud r, \quad 0 \leq \lambda \leq \delta_0.
    \end{equation}
    The main work now goes into proving the off-diagonal decay of the operator norms $\|T_j T_k^\ast\|_{L^2_\lambda \to L^2_\lambda}$ and $\|T_j^\ast T_k\|_{L^2_r \to L^2_r}$, $1 \leq j,k \leq \infty$.
    The integral kernels of these operators are given by
    \begin{equation}
        \begin{aligned}
            \bigl( T_j T_k^\ast u \bigr)(\lambda) &= \int_0^{\delta_0} K_{jk}(\lambda, \mu) u(\mu) \ud \mu, \\ 
            \bigl( T_j^\ast T_k v \bigr)(r) &= \int_0^\infty M_{jk}(r,s) v(s) \ud s,
        \end{aligned}
    \end{equation}
    with 
    \begin{equation}
        \begin{aligned}
            K_{jk}(\lambda,\mu) &:= \int_0^\infty e^{i(k_1(\lambda)-k_1(\mu))r} \frakm(\lambda,r) \overline{\frakm(\mu,r)} \varphi\bigl(2^{-j} \lambda r\bigr) \varphi\bigl(2^{-k} \mu r\bigr) \ud r, \\
            M_{jk}(r,s) &:= \int_0^{\delta_0} e^{-ik_1(\lambda)(r-s)} \overline{\frakm(\lambda,r)} \frakm(\lambda,s) \varphi\bigl(2^{-j} \lambda r\bigr) \varphi\bigl(2^{-k} \lambda s\bigr) \ud \lambda.
        \end{aligned}
    \end{equation}
    We discuss the operator norm bounds for $T_j T_k^\ast$, $1 \leq j, k \leq \infty$, the ones for $T_j^\ast T_k$ being more of the same. To this end we observe that by \eqref{eq:H0rec},
    \begin{equation} \label{equ:symbol_type_frakm}
        \bigl| (r \partial_r)^\ell \frakm(\lambda,r) \bigr| \lesssim 1, \quad 0 \leq \ell \leq 2.
    \end{equation}
    Moreover, in view of the supports of the cut-off functions, for the integral kernel $K_{jk}(\lambda,\mu)$ to be non-zero, in effect we must have 
    \begin{equation}
        \begin{aligned}
            2^{-j} \lambda r \simeq 1 \simeq 2^{-k} \mu r \quad \Rightarrow \quad \frac{\lambda}{2^j} \simeq \frac{\mu}{2^k}.
        \end{aligned}
    \end{equation}
    Integrating by parts twice using \eqref{equ:symbol_type_frakm} or trivially bounding the integral, yields the kernel bounds
    \begin{equation}
        \begin{aligned}
            \bigl|K_{j,k}(\lambda,\mu)\bigr| \lesssim \frac{\frac{2^j}{\lambda}}{\bigl( 1 + \frac{2^j}{\lambda}\bigl| k_1(\lambda)-k_1(\mu) \bigr|\bigr)^2} \simeq \frac{\frac{2^k}{\mu}}{\bigl( 1 + \frac{2^k}{\mu}\bigl| k_1(\lambda)-k_1(\mu) \bigr|\bigr)^2}. 
        \end{aligned}
    \end{equation}
    Using that $k_1(\lambda)-k_1(\mu) \simeq \lambda - \mu$ for $0 \leq \lambda,\mu \leq \delta_0$, we obtain by Schur's test that $\|T_j T_k^\ast\|_{L^2_\lambda \to L^2_\lambda} \lesssim 2^{-|j-k|}$. Analogous arguments also yield the off-diagonal decay estimates $\|T_j^\ast T_k\|_{L^2_\lambda \to L^2_\lambda} \lesssim 2^{-|j-k|}$. The asserted $L^2$-bound \eqref{equ:L2bound_hplus} now follows by the Cotlar-Stein lemma.

    The second asserted bound \eqref{equ:L2bound_uppsioneplus} follows immediately from \eqref{equ:L2bound_hplus} and Lemma~\ref{lem:decomposition_uppsioneplus}.
\end{proof}

Finally, we are in the position to prove Theorem~\ref{thm:summary}.

\begin{proof}[Proof of Theorem~\ref{thm:summary}]
The first two statements, $(i)$ and $(ii)$, were proved in Lemma~\ref{lem:HYos}. The lower bound on the growth of $\|e^{t\calL}\|$, statement $(iii)$, was proved in Lemma~\ref{lem:nilpotent}.
It remains to prove the operator norm bounds stated in $(iv)$. Our goal is to show that there exists some absolute constant $C > 0$ such that for all intervals $I \subset [-\delta_0,\delta_0]$ we have 
\begin{equation*}
    \bigl| \bigl\langle e^{t\calL} P_I \bmv,\bmw \bigr\rangle\bigr| \leq C \jap{t} \|\bmv\|_{L^2_r}\|\bmw\|_{L^2_r}, \quad \bmv = \begin{bmatrix} v_1 \\ v_2 \end{bmatrix}, \quad \bmw = \begin{bmatrix} w_1 \\ w_2 \end{bmatrix}.
\end{equation*}
The latter implies the operator norm bound \eqref{equ:operator_norm_bound_thm} in the statement of Theorem~\ref{thm:summary}. An inspection of the proof below also yields the significantly easier uniform-in-time operator norm bound \eqref{equ:operator_norm_bound_thm_away_zero} for intervals supported away from zero energy.

In what follows, we can assume without loss of generality that $I = [-\delta_0,\delta_0]$.
We also recall the hierarchy of fixed small absolute constants $0 < \delta_0 \ll\epsilon_\infty\ll\epsilon\ll\epsilon_0 \ll 1$. 
In view of \eqref{eq:StonePI3} and Lemma~\ref{lem:varfystructure}, we have 
\begin{equation} \label{eq:L2boundtemp1}
\bigl\langle e^{t\calL} P_I \bmv, \bmw \bigr\rangle = \frac{1}{2\pi i} \int_I e^{it\lambda}\frac{\kappa(\lambda)}{d_{+}(\lambda) d_{-}(\lambda)} \bigl\langle \upfy(\cdot, \lambda), \sigma_1\bmv(\cdot) \bigr\rangle \bigl\langle \upfy(\cdot, \lambda), \bmw(\cdot) \bigr\rangle \ud \lambda.
\end{equation}
First, we decompose the inner products $\langle \upfy(\cdot, \lambda), \sigma_1\bmv(\cdot) \rangle$ and $\langle \upfy(\cdot, \lambda), \bmw(\cdot) \rangle$ into the integration regions $r \lesssim \epsilon |\lambda|^{-1}$ and $r \gtrsim \epsilon |\lambda|^{-1}$ for $\epsilon_\infty \ll \epsilon \ll \epsilon_0$. 
In the region $r \lesssim \epsilon |\lambda|^{-1}$ we will use the representation \eqref{eq:upfyexpansion1} for $\upfy$, while for $r \gtrsim \epsilon |\lambda|^{-1}$ the representation \eqref{eq:upfyexpansion2} will be amenable to good estimates.
Correspondingly, we write 
\begin{equation}
    \begin{aligned}
        2 \pi i \bigl\langle e^{t\calL} P_I \bmv, \bmw \bigr\rangle &= \sum_{0 \leq j,k \leq 1} \int_I e^{it\lambda}\frac{\kappa(\lambda)}{d_{+}(\lambda) d_{-}(\lambda)} \calI_j(\lambda) \calJ_k(\lambda) \, \ud \lambda 
    \end{aligned}
\end{equation}
with 
\begin{equation}
    \begin{aligned}
        \calI_j(\lambda) := \bigl\langle \upfy(\cdot, \lambda) \chi_j(\lambda \cdot), \sigma_1\bmv(\cdot) \bigr\rangle, \quad \calJ_k(\lambda) := \bigl\langle \upfy(\cdot, \lambda) \chi_k(\lambda \cdot), \bmw(\cdot) \bigr\rangle.
    \end{aligned}
\end{equation}
We now consider the contributions of $\calI_j(\lambda) \calJ_k(\lambda)$, $0 \leq j, k \leq 1$, separately. Going forward we use the notation $\|\cdot\|_{L^2}$ for $\|\cdot\|_{L^2_r}$.

\medskip 

\noindent \underline{Contribution of $\calI_0(\lambda) \calJ_0(\lambda)$.}
Inserting \eqref{eq:upfyexpansion1} we find that
\begin{equation}
    \begin{aligned}
        &\int_I e^{it\lambda} \frac{\kappa(\lambda)}{d_{+}(\lambda) d_{-}(\lambda)} \calI_0(\lambda) \calJ_0(\lambda) \, \ud \lambda \\
        &= \int_I e^{it\lambda} \frac{\kappa(\lambda) \omega_{22}^+(\lambda)^2}{d_{+}(\lambda) d_{-}(\lambda)} \bigl\langle \upfy_1(\cdot, \lambda) \chi_0(\lambda \cdot), \sigma_1\bmv(\cdot) \bigr\rangle \bigl\langle \upfy_1(\cdot, \lambda) \chi_0(\lambda \cdot), \bmw(\cdot) \bigr\rangle \, \ud \lambda \\ 
        &\quad - \int_I e^{it\lambda} \frac{\kappa(\lambda) \omega_{22}^+(\lambda) \omega_{21}^+(\lambda)}{d_{+}(\lambda) d_{-}(\lambda)} \bigl\langle \upfy_1(\cdot, \lambda) \chi_0(\lambda \cdot), \sigma_1\bmv(\cdot) \bigr\rangle \bigl\langle \upfy_2(\cdot, \lambda) \chi_0(\lambda \cdot), \bmw(\cdot) \bigr\rangle \, \ud \lambda \\ 
        &\quad - \int_I e^{it\lambda} \frac{\kappa(\lambda) \omega_{22}^+(\lambda) \omega_{21}^+(\lambda)}{d_{+}(\lambda) d_{-}(\lambda)} \bigl\langle \upfy_2(\cdot, \lambda) \chi_0(\lambda \cdot), \sigma_1\bmv(\cdot) \bigr\rangle \bigl\langle \upfy_1(\cdot, \lambda) \chi_0(\lambda \cdot), \bmw(\cdot) \bigr\rangle \, \ud \lambda \\ 
        &\quad + \int_I e^{it\lambda} \frac{\kappa(\lambda) \omega_{21}(\lambda)^2}{d_{+}(\lambda) d_{-}(\lambda)} \bigl\langle \upfy_2(\cdot, \lambda) \chi_0(\lambda \cdot), \sigma_1\bmv(\cdot) \bigr\rangle \bigl\langle \upfy_2(\cdot, \lambda) \chi_0(\lambda \cdot), \bmw(\cdot) \bigr\rangle \, \ud \lambda \\ 
        &=: I + II + III + IV.
    \end{aligned}
\end{equation}
Next, we record that by Lemma~\ref{lem:bellj1}, Corollary~\ref{cor:omegaij}, and Proposition~\ref{prop:PIStone1}, we have
\begin{equation}
    \biggl|\frac{\kappa(\lambda) \omega_{22}^+(\lambda)^2}{d_{+}(\lambda) d_{-}(\lambda)} \biggr| \simeq 1, \quad \biggl|\frac{\kappa(\lambda) \omega_{22}^+(\lambda) \omega_{21}^+(\lambda)}{d_{+}(\lambda) d_{-}(\lambda)}\biggr| \lesssim |\lambda| e^{-2\sqrt{2}r_\epsilon}, \quad \biggl|\frac{\kappa(\lambda) \omega_{21}^+(\lambda)^2}{d_{+}(\lambda) d_{-}(\lambda)}\biggr| \lesssim |\lambda|^2 e^{-4\sqrt{2}r_\epsilon}.
\end{equation}
The term $I$ is the most delicate to bound and requires exploiting a subtle cancellation, while the terms $II$, $III$, and $IV$ are straightforward.
So to get started, we consider the term $II$. 
Recall from Proposition~\ref{lem:F1_1} that uniformly for all $0 < |\lambda| \leq \delta_0$, 
\begin{equation}
    \begin{aligned}
        \bigl| \upfy_1(r,\lambda) \bigr| \lesssim \jap{r}^{\frac12}, \quad \bigl| \upfy_2(r,\lambda) \bigr| \lesssim e^{\sqrt{2}r}, \quad 0 < r \leq r_\epsilon(\lambda).
    \end{aligned}
\end{equation}
Using the Cauchy-Schwarz inequality repeatedly and dropping some excessive powers of $|\lambda|$, we find
\begin{equation}
    \begin{aligned}
        |II| &\lesssim \int_I  e^{-2\sqrt{2}r_\epsilon} \bigl\| \upfy_1(\cdot, \lambda) \chi_0(\lambda \cdot) \bigr\|_{L^2} \|\bmv\|_{L^2} \bigl\| \upfy_2(\cdot, \lambda) \chi_0(\lambda \cdot) \bigr\|_{L^2} \|\bmw\|_{L^2} \, \ud \lambda \\ 
        &\lesssim \biggl( \int_I |\lambda|^{-1} e^{-\sqrt{2}r_\epsilon} \ud\lambda \biggr) \|\bmv\|_{L^2} \|\bmw\|_{L^2} \\
        &\lesssim \|\bmv\|_{L^2} \|\bmw\|_{L^2}.
    \end{aligned}
\end{equation}
The term $III$ can be estimated analogously, and the term $IV$ is more of the same.

It remains to consider the term $I$.
Upon inserting \eqref{equ:definition_upfy_one_and_two} for $\upfy_1(r,\lambda)$ with $f(r,\lambda)$ defined in \eqref{equ:definition_f_and_g}, we obtain that 
\begin{equation}
    \begin{aligned}
        I &= \int_I e^{it\lambda} \frac{\kappa(\lambda) \omega_{22}^+(\lambda)^2}{d_{+}(\lambda) d_{-}(\lambda)} \langle f(\cdot, \lambda) \chi_0(\lambda \cdot), v_1(\cdot) \rangle \langle f(\cdot, \lambda) \chi_0(\lambda \cdot), w_2(\cdot) \rangle \, \ud \lambda \\ 
        &\quad + \int_I e^{it\lambda} \frac{\kappa(\lambda) \omega_{22}^+(\lambda)^2}{d_{+}(\lambda) d_{-}(\lambda)} \, \lambda \, \langle f(\cdot, \lambda) \chi_0(\lambda \cdot), v_1(\cdot) \rangle \langle (\calG_2 f)(\cdot, \lambda) \chi_0(\lambda \cdot), w_1(\cdot) \rangle \, \ud \lambda \\ 
        &\quad + \int_I e^{it\lambda} \frac{\kappa(\lambda) \omega_{22}^+(\lambda)^2}{d_{+}(\lambda) d_{-}(\lambda)} \, \lambda \, \langle (\calG_2 f)(\cdot, \lambda) \chi_0(\lambda \cdot), v_1(\cdot) \rangle \langle f(\cdot, \lambda) \chi_0(\lambda \cdot), w_2(\cdot) \rangle \, \ud \lambda \\ 
        &\quad + \int_I e^{it\lambda} \frac{\kappa(\lambda) \omega_{22}^+(\lambda)^2}{d_{+}(\lambda) d_{-}(\lambda)} \, \lambda^2 \, \langle (\calG_2 f)(\cdot, \lambda) \chi_0(\lambda \cdot), v_1(\cdot) \rangle \langle (\calG_2 f)(\cdot, \lambda) \chi_0(\lambda \cdot), w_1(\cdot) \rangle \, \ud \lambda \\
        &=: I_1 + I_2 + I_3 + I_4.
    \end{aligned}
\end{equation}
To treat the term $I_1$, we observe that by Lemma~\ref{lem:omegapm}, Corollary~\ref{cor:omegaij}, and Lemma~\ref{lem:pmlambda1},
\begin{equation}
    \frac{\kappa(\lambda) \omega_{22}^+(\lambda)^2}{d_{+}(\lambda) d_{-}(\lambda)} = \tilde{c} \, \mathrm{sgn}(\lambda) + \calO(|\lambda|)
\end{equation}
for some constant $\tilde{c} \neq 0$. This gives
\begin{equation}
    \begin{aligned}
        I_1 &= \int_I e^{it\lambda} \Bigl( \tilde{c} \, \mathrm{sgn}(\lambda) + \calO(|\lambda|) \Bigr) \langle f(\cdot, \lambda) \chi_0(\lambda \cdot), v_1(\cdot) \rangle \langle f(\cdot, \lambda) \chi_0(\lambda \cdot), w_1(\cdot) \rangle \, \ud \lambda.
    \end{aligned}
\end{equation}
Recall that $I = [-\delta_0, \delta_0]$ for some small $0 < \delta_0 \ll 1$.
For the contribution of the leading order term, we make a change of variables $\lambda \mapsto -\lambda$ for the integration over $-\delta_0 \leq \lambda \leq 0$ and we exploit that the maps $\lambda \mapsto f(r,\lambda)$ and $\lambda \mapsto \chi_0(\lambda \cdot)$ are even. This gives
\begin{equation}
    \begin{aligned}
        I_1^{\mathrm{main}}&:=\tilde{c} \int_I e^{it\lambda} \mathrm{sgn}(\lambda)  \langle f(\cdot, \lambda) \chi_0(\lambda \cdot), v_1(\cdot) \rangle \langle f(\cdot, \lambda) \chi_0(\lambda \cdot), w_2(\cdot) \rangle \, \ud \lambda \\
        &= 2i \tilde{c} \int_0^{\delta_0} \sin(t\lambda) \langle f(\cdot, \lambda) \chi_0(\lambda \cdot), v_1(\cdot) \rangle \langle f(\cdot, \lambda) \chi_0(\lambda \cdot), w_2(\cdot) \rangle \, \ud \lambda.
    \end{aligned}
\end{equation}
Using that $|\sin(t\lambda)| \leq \jap{t} \lambda$ uniformly for $0 < \lambda \leq \delta_0$, we obtain an additional power of $\lambda$ at the expense of a linear growth in time $t$. At this point we can proceed as in the preceding estimates to conclude that 
\begin{equation}
    \begin{aligned}
        I_1^{\mathrm{main}}&\lesssim \jap{t} \Bigl\|\bigl\| |\lambda|^{\frac12} f(\cdot, \lambda) \chi_0(\lambda \cdot) \bigr\|_{L^2} \Bigr\|_{L^2_\lambda(I)}^2\|v_1\|_{L^2} \|w_2\|_{L^2}\\
        &\lesssim\jap{t}\Bigl\| \bigl\| |\lambda|^{\frac12} \jap{\cdot}^{\frac12} \chi_0(\lambda \cdot) \bigr\|_{L^2} \Bigr\|_{L^2_\lambda(I)}\|v_1\|_{L^2} \|w_2\|_{L^2},\\
        &\lesssim \jap{t}\|v_1\|_{L^2} \|w_2\|_{L^2}.
    \end{aligned}
\end{equation}
Here to pass to the last line we have used Schur's test to show that
\begin{equation} \label{equ:smallr_schur}
    \Bigl\| \bigl\| |\lambda|^{\frac12} \jap{\cdot}^{\frac12} \chi_0(\lambda \cdot) \bigr\|_{L^2} \Bigr\|_{L^2_\lambda(I)} \lesssim 1.
\end{equation}
The contribution of the remainder term $\calO(|\lambda|)$ to $I_1$ can also be bounded analogously, where due to the extra factor of $|\lambda|$ the final estimate is independent of $t$.  The terms $I_2$, $I_3$ and $I_4$ are more of the same.
This finishes the discussion of the contributions of $\calI_0(\lambda) \calJ_0(\lambda)$.

\medskip 
\noindent \underline{Contribution of $\calI_1(\lambda) \calJ_1(\lambda)$.} 
Inserting the representation \eqref{equ:upfylarger1} of the distorted Fourier basis element, we find that
\begin{equation}
    \begin{aligned}
        &\int_I e^{it\lambda} \frac{\kappa(\lambda)}{d_{+}(\lambda) d_{-}(\lambda)} \calI_1(\lambda) \calJ_1(\lambda) \, \ud \lambda \\
        &= \int_I e^{it\lambda} \frac{\kappa(\lambda) \mu_1(\lambda)^2}{d_{+}(\lambda) d_{-}(\lambda)} \bigl\langle  \big( \omega^+_{11}(\lambda) \sigma_3 \uppsi_1^+(\cdot,-\lambda)+\omega_{11}^+(-\lambda)\uppsi_1^+(\cdot,\lambda) \big) \chi_1(\lambda \cdot), \sigma_1\bmv(\cdot) \bigr\rangle \\ 
        &\qquad \qquad \qquad \qquad \qquad \qquad \times \bigl\langle \big( \omega_{11}^+(\lambda)\sigma_3\uppsi_1^+(\cdot,-\lambda)+\omega_{11}^+(-\lambda)\uppsi_1^+(\cdot,\lambda) \big) \chi_1(\lambda \cdot), \bmw(\cdot) \bigr\rangle \, \ud \lambda \\         
        &\quad + \int_I e^{it\lambda} \frac{\kappa(\lambda)}{d_{+}(\lambda) d_{-}(\lambda)} \bigl\langle \mu_1(\lambda) \big( \omega_{11}^+(\lambda) \sigma_3 \uppsi_1^+(\cdot,-\lambda)+\omega_{11}^+(-\lambda)\uppsi_1^+(\cdot,\lambda) \big) \chi_1(\lambda \cdot), \sigma_1\bmv(\cdot) \bigr\rangle \\ 
        &\qquad \qquad \qquad \qquad \qquad \qquad \times \bigl\langle \bigl( \mu_2(\lambda) \sigma_3 \uppsi_1^+(\cdot,-\lambda) + \mu_3(\lambda) \uppsi_1^+(\cdot,\lambda) + \mu_4(\lambda) \uppsi_2^+(r,\lambda) \bigr) \chi_1(\lambda \cdot), \bmw(\cdot) \bigr\rangle \, \ud \lambda \\ 
        &\quad + \int_I e^{it\lambda} \frac{\kappa(\lambda)}{d_{+}(\lambda) d_{-}(\lambda)} \bigl\langle \bigl( \mu_2(\lambda) \sigma_3 \uppsi_1^+(\cdot,-\lambda) + \mu_3(\lambda) \uppsi_1^+(\cdot,\lambda) + \mu_4(\lambda) \uppsi_2^+(r,\lambda) \bigr) \chi_1(\lambda \cdot), \sigma_1\bmv(\cdot) \bigr\rangle \\
        &\qquad \qquad \qquad \qquad \qquad \qquad \times \bigl\langle \upfy(\cdot,\lambda) \chi_1(\lambda \cdot), \bmw(\cdot) \bigr\rangle \, \ud \lambda \\          
        &=: I + II + III.
    \end{aligned}
\end{equation}
We begin with the most delicate term $I$. By Lemma~\ref{lem:pmlambda1} and \eqref{equ:mu_asymptotics_and_bounds}, for some nonzero constant $\tilde{c}$ we have 
\begin{equation} \label{equ:large_large_interaction1}
     \frac{\kappa(\lambda) \mu_1(\lambda)^2}{d_{+}(\lambda) d_{-}(\lambda)} = \tilde{c} \, \lambda^{-2} \mathrm{sgn}(\lambda) + \calO\bigl(|\lambda|^{-1}\bigr).
\end{equation}
Moreover, by Lemma~\ref{lem:decomposition_uppsioneplus} we have the decomposition
\begin{equation} \label{equ:large_large_interaction2}
    \uppsi_1^+(r,\lambda) = \begin{pmatrix} 0 \\ 1 \end{pmatrix} h_+\bigl(k_1(\lambda)r\bigr) + \calO\bigl(|\lambda|\bigr) \begin{pmatrix} 1 \\ 0 \end{pmatrix} h_+\bigl(k_1(\lambda)r\bigr) + \Upgamma(r,\lambda).
\end{equation}
Here we only discuss the contribution of the first terms on the right-hand sides of \eqref{equ:large_large_interaction1} and \eqref{equ:large_large_interaction2}, which are leading order in terms of the lack of powers of $|\lambda|$.
Making a change of variables $\lambda \mapsto -\lambda$ for the integration over $-\delta_0 \leq \lambda \leq 0$ and exploiting that the map
\begin{equation}
    \lambda \mapsto \omega_{11}^+(\lambda) \sigma_3 \begin{pmatrix} 0 \\ 1 \end{pmatrix} h_+\bigl(-k_1(\lambda)r\bigr) +\omega_{11}^+(-\lambda) \begin{pmatrix} 0 \\ 1 \end{pmatrix} h_+\bigl(k_1(\lambda)r\bigr)
\end{equation}
is odd, we obtain that
\begin{equation}
    \begin{aligned}
        &\int_I e^{it\lambda} \tilde{c} \lambda^{-2} \mathrm{sgn}(\lambda) \Bigl\langle \Bigl( \omega_{11}^+(\lambda) \sigma_3 \begin{pmatrix} 0 \\ 1 \end{pmatrix} h_+\bigl(-k_1(\lambda)\cdot\bigr) + \omega_{11}^+(-\lambda) \begin{pmatrix} 0 \\ 1 \end{pmatrix} h_+\bigl(k_1(\lambda)\cdot\bigr) \Bigr), \sigma_1\bmv(\cdot) \Bigr\rangle \\ 
        &\qquad \qquad \qquad \qquad \times \Bigl\langle \Bigl( \omega_{11}^+(\lambda) \sigma_3 \begin{pmatrix} 0 \\ 1 \end{pmatrix} h_+\bigl(-k_1(\lambda)\cdot\bigr) + \omega_{11}^+(-\lambda) \begin{pmatrix} 0 \\ 1 \end{pmatrix} h_+\bigl(k_1(\lambda)\cdot\bigr) \Bigr) \chi_1(\lambda \cdot), \bmw(\cdot) \Bigr\rangle \, \ud \lambda \\ 
        &= 2i \tilde{c} \int_0^{\delta_0} \sin(t\lambda) \lambda^{-2} \bigl\langle \bigl( -\omega_{11}^+(\lambda) h_+\bigl(-k_1(\lambda)\cdot\bigr) + \omega_{11}^+(-\lambda) h_+\bigl(k_1(\lambda)\cdot\bigr) \bigr) \chi_1(\lambda \cdot), v_1(\cdot) \bigr\rangle \\ 
        &\qquad \qquad \qquad \qquad \times \bigl\langle \bigl( -\omega_{11}^+(\lambda) h_+\bigl(-k_1(\lambda)\cdot\bigr) + \omega_{11}^+(-\lambda) h_+\bigl(k_1(\lambda)\cdot\bigr) \bigr) \chi_1(\lambda \cdot), w_2(\cdot) \bigr\rangle \, \ud \lambda.
    \end{aligned}
\end{equation}
Using that $|\sin(t\lambda)| \leq \jap{t} |\lambda|$ uniformly for all $0 \leq \lambda \leq \delta_0$ and that $|\omega_{11}^+(\pm\lambda)| \simeq |\lambda|^{\frac12}$ by Corollary~\ref{cor:omegaij}, we have uniformly for all $0 < \lambda \leq \delta_0$ that
\begin{equation}
    \bigl| \sin(t\lambda) \lambda^{-2} \omega_{11}^+(\pm\lambda) \omega_{11}^+(\pm\lambda) \bigr| \lesssim \jap{t}.
\end{equation}
Hence, by the Cauchy-Schwarz inequality and by Lemma~\ref{lem:L2bound_hplus}, we obtain
\begin{equation}
    \begin{aligned}
        &\biggl| 2i \tilde{c} \int_0^{\delta_0} \sin(t\lambda) \lambda^{-2} \bigl\langle \bigl( -\omega_{11}^+(\lambda) h_+\bigl(-k_1(\lambda)\cdot\bigr) + \omega_{11}^+(-\lambda) h_+\bigl(k_1(\lambda)\cdot\bigr) \bigr) \chi_1(\lambda \cdot), v_1(\cdot) \bigr\rangle \\ 
        &\qquad \qquad \qquad \qquad \times \bigl\langle \bigl( -\omega_{11}^+(\lambda) h_+\bigl(-k_1(\lambda)\cdot\bigr) + \omega_{11}^+(-\lambda) h_+\bigl(k_1(\lambda)\cdot\bigr) \bigr) \chi_1(\lambda \cdot), w_2(\cdot) \bigr\rangle \, \ud \lambda \biggr| \\ 
        &\lesssim \jap{t} \bigl\| \bigl\langle h_+(\pm k_1(\lambda) \cdot) \chi_1(\lambda \cdot), v_1(\cdot) \bigr\rangle \bigr\|_{L^2_\lambda([0,\delta_0])} 
        \bigl\| \bigl\langle h_+(\pm k_1(\lambda) \cdot) \chi_1(\lambda \cdot), w_2(\cdot) \bigr\rangle \bigr\|_{L^2_\lambda([0,\delta_0])} \\        
        &\lesssim \jap{t} \|v_1\|_{L^2_r} \|w_2\|_{L^2_r},
    \end{aligned}
\end{equation}
as desired. The contributions of all other terms on the right-hand sides of \eqref{equ:large_large_interaction1} and \eqref{equ:large_large_interaction2} come with sufficient powers of $|\lambda|$ factors so that they can be bounded in a straightforward manner by repeated applications of the Cauchy-Schwarz inequality along with Lemma~\ref{lem:L2bound_hplus} and the $L^2$-bound \eqref{equ:L2bound_remainder} for the remainder term $\Upgamma(r,\lambda)$ from Lemma~\ref{lem:decomposition_uppsioneplus}.

The same comments apply to estimating the less delicate terms $II$ and $III$. 
Here we note that whenever the term $\mu_4(\lambda) \uppsi_2^+(r,\lambda)$ appears, we place $\uppsi_2^+(r,\lambda)$ in $L^2_r$ by Cauchy-Schwarz and use Lemma~\ref{lem:basic_bound_L2_uppsi2}. The latter yields the decaying factor $e^{-\sqrt{2}\frac{r_\epsilon}{3}}$, which suffices to compensate the exponentially large factor $|\mu_4(\lambda)| \lesssim e^{\frac{\sqrt{2}}{5}r_\epsilon}$ and additional inverse powers of $|\lambda|$.
We leave the details to the reader.

\medskip 

\noindent \underline{Contribution of $\calI_0(\lambda) \calJ_1(\lambda)$.} 
Inserting the representations \eqref{eq:upfyexpansion1} and \eqref{equ:upfylarger1} of the distorted Fourier basis elements, we find 
\begin{equation}
    \begin{aligned}
        &\int_I e^{it\lambda} \frac{\kappa(\lambda)}{d_{+}(\lambda) d_{-}(\lambda)} \calI_0(\lambda) \calJ_1(\lambda) \, \ud \lambda \\
        &= \int_I e^{it\lambda} \frac{\kappa(\lambda) \omega_{22}^+(\lambda) \mu_1(\lambda)}{d_{+}(\lambda) d_{-}(\lambda)} \bigl\langle \upfy_1(\cdot,\lambda) \chi_0(\lambda \cdot), \sigma_1\bmv(\cdot) \bigr\rangle \\ 
        &\qquad \qquad \qquad \qquad \times \bigl\langle \big( \omega_{11}^+(\lambda)\sigma_3\uppsi_1^+(\cdot,-\lambda)+\omega_{11}^+(-\lambda) \uppsi_1^+(\cdot,\lambda) \big) \chi_1(\lambda \cdot), \bmw(\cdot) \bigr\rangle \, \ud \lambda \\         
        &\quad + \int_I e^{it\lambda} \frac{\kappa(\lambda) \omega_{22}^+(\lambda)}{d_{+}(\lambda) d_{-}(\lambda)} \bigl\langle \upfy_1(\cdot,\lambda) \chi_0(\lambda \cdot), \sigma_1\bmv(\cdot) \bigr\rangle \\
        &\qquad \qquad \qquad \qquad \times \bigl\langle \bigl( \mu_2(\lambda) \sigma_3 \uppsi_1^+(\cdot,-\lambda) + \mu_3(\lambda) \uppsi_1^+(\cdot,\lambda) + \mu_4(\lambda) \uppsi_2^+(r,\lambda) \bigr) \chi_1(\lambda \cdot), \bmw(\cdot) \bigr\rangle \, \ud \lambda \\ 
        &\quad - \int_I e^{it\lambda} \frac{\kappa(\lambda) \omega_{21}^+(\lambda)}{d_{+}(\lambda) d_{-}(\lambda)} \bigl\langle \upfy_2(\cdot,\lambda) \chi_0(\lambda \cdot), \sigma_1\bmv(\cdot) \bigr\rangle \bigl\langle \upfy(\cdot,\lambda) \chi_1(\lambda \cdot), \bmw(\cdot) \bigr\rangle \, \ud \lambda \\          
        &=: I + II + III.
    \end{aligned}
\end{equation}
Again, we only discuss the most delicate term $I$. By Lemma~\ref{lem:omegapm}, Corollary~\ref{cor:omegaij}, Lemma~\ref{lem:pmlambda1}, and \eqref{equ:mu_asymptotics_and_bounds},  we have 
\begin{equation} \label{equ:small_large_interaction0}
    \frac{\kappa(\lambda) \omega_{22}^+(\lambda) \mu_1(\lambda)}{d_{+}(\lambda) d_{-}(\lambda)} = \tilde{c} \lambda^{-1} \mathrm{sgn}(\lambda) + \calO(1)
\end{equation}
for some constant $\tilde{c} \neq 0$. 
Recall that 
\begin{equation} \label{equ:small_large_interaction1}
    \upfy_1(r,\lambda) = \begin{pmatrix} 0 \\ f(r,\lambda) \end{pmatrix} + \lambda \begin{pmatrix} \bigl(\calG_2 f\bigr)(r,\lambda) \\ 0 \end{pmatrix}
\end{equation}
and that by Lemma~\ref{lem:decomposition_uppsioneplus},
\begin{equation} \label{equ:small_large_interaction2}
    \uppsi_1^+(r,\lambda) = \begin{pmatrix} 0 \\ 1 \end{pmatrix} h_+\bigl(k_1(\lambda)r\bigr) + \calO\bigl(|\lambda|\bigr) \begin{pmatrix} 1 \\ 0 \end{pmatrix} h_+\bigl(k_1(\lambda)r\bigr) + \Upgamma(r,\lambda).
\end{equation}
The contributions of the first terms on the right-hand sides of \eqref{equ:small_large_interaction0}, \eqref{equ:small_large_interaction1}, and \eqref{equ:small_large_interaction2} are leading order in terms of the lack of powers of $|\lambda|$. 
Making a change of variables $\lambda \mapsto -\lambda$ for the integration over $-\delta_0 \leq \lambda \leq 0$ and exploiting that the leading order term in \eqref{equ:small_large_interaction0} is even with respect to $\lambda$, that the map $\lambda \mapsto f(r,\lambda)$ is even, while the map
\begin{equation}
    \lambda \mapsto \omega_{11}^+(\lambda) \sigma_3 \begin{pmatrix} 0 \\ 1 \end{pmatrix} h_+\bigl(-k_1(\lambda)r\bigr) +\omega_{11}^+(-\lambda) \begin{pmatrix} 0 \\ 1 \end{pmatrix} h_+\bigl(k_1(\lambda)r\bigr)
\end{equation}
is odd, we find that
\begin{equation}
    \begin{aligned}
        &\int_I e^{it\lambda} \tilde{c} \lambda^{-1} \mathrm{sgn}(\lambda) \Bigl\langle \begin{pmatrix} 0 \\ f(\cdot,\lambda) \end{pmatrix} \chi_0(\lambda \cdot), \sigma_1 \bmv(\cdot) \Bigr\rangle \\ 
        &\qquad \qquad \qquad \times \Bigl\langle \Bigl( \omega_{11}^+(\lambda) \sigma_3 \begin{pmatrix} 0 \\ 1 \end{pmatrix} h_+\bigl(-k_1(\lambda)\cdot\bigr) +\omega_{11}^+(-\lambda) \begin{pmatrix} 0 \\ 1 \end{pmatrix} h_+\bigl(k_1(\lambda)\cdot\bigr) \Bigr) \chi_1(\lambda \cdot), \bmw(\cdot) \Bigr\rangle \ud \lambda \\ 
        &= 2i \tilde{c} \int_0^{\delta_0} \lambda^{-1} \sin(t\lambda) \bigl\langle f(\cdot,\lambda) \chi_0(\lambda\cdot), v_1(\cdot) \bigr\rangle \\ 
        &\qquad \qquad \qquad \times \bigl\langle \bigl( -\omega_{11}^+(\lambda) h_+\bigl(-k_1(\lambda)\cdot\bigr) + \omega_{11}^+(-\lambda) h_+\bigl(k_1(\lambda)\cdot\bigr) \bigr) \chi_1(\lambda \cdot), w_2(\cdot) \bigr\rangle \ud \lambda. 
    \end{aligned}
\end{equation}
Since by Corollary~\ref{cor:omegaij} we have 
\begin{equation}
    \bigl| \lambda^{-1} \sin(t\lambda) \omega_{11}^+(\pm \lambda) \bigr| \lesssim \jap{t} |\lambda|^{\frac12},
\end{equation}
we can use the Cauchy-Schwarz inequality to estimate the preceding integral by
\begin{equation}
    \begin{aligned}
        &\biggl| 2i \tilde{c} \int_0^{\delta_0} \lambda^{-1} \sin(t\lambda) \bigl\langle f(\cdot,\lambda) \chi_0(\lambda\cdot), v_1(\cdot) \bigr\rangle \\ 
        &\qquad \qquad \qquad \times \bigl\langle \bigl( -\omega_{11}^+(\lambda) h_+\bigl(-k_1(\lambda)\cdot\bigr) + \omega_{11}^+(-\lambda) h_+\bigl(k_1(\lambda)\cdot\bigr) \bigr) \chi_1(\lambda \cdot), w_2(\cdot) \bigr\rangle \ud \lambda \biggr| \\ 
        &\lesssim \jap{t} \Bigl\| \bigl\| |\lambda|^{\frac12} f(\cdot,\lambda) \chi_0(\lambda \cdot) \bigr\|_{L^2_r} \Bigr\|_{L^2_\lambda([0,\delta_0])} \|v_1\|_{L^2_r} \bigl\| \bigl\langle h_+\bigl(\pm k_1(\lambda)\cdot\bigr), w_2(\cdot) \bigr\rangle \bigr\|_{L^2_\lambda([0,\delta_0])} \\ 
        &\lesssim \jap{t} \|v_1\|_{L^2_r} \|w_2\|_{L^2_r},
    \end{aligned}
\end{equation}
as desired.
In the last estimate we invoked \eqref{equ:smallr_schur} and Lemma~\ref{lem:L2bound_hplus}.

\medskip 
\noindent \underline{Contribution of $\calI_1(\lambda) \calJ_0(\lambda)$.} 
This case is analogous to the preceding and we omit the details.
\end{proof}

\appendix 

\section{The resolvent kernel for $i\calL_0$}\label{app:calL0}

Here we present the proof of Lemma~\ref{lem:calL0res1}. While the argument is similar to to some of the arguments in Section~\ref{sec:Greens}, it is logically completely independent. We use the notation introduced at the beginning of Section~\ref{sec:HY}.
\begin{proof}[Proof of Lemma~\ref{lem:calL0res1}]
Let $\tilD(z):=\calW[\Psi(\cdot,z),\Theta(\cdot,z)]$. By the argument  leading to~\eqref{eq:G+} and~\eqref{eq:G-}, the resolvent kernel for $i\calL_0-z$, $\Im z>0$, is given by
\begin{align*}
\begin{split}
\calG_0(r,s,z)=\rc \big(\Psi(r,z)\tilD^{-t}(z)\Theta^t(s,z)\mathds{1}_{[0 < s \leq r]}+\Theta(r,z)\tilD^{-1}(z)\Psi^t(s,z)\mathds{1}_{[ s > r]}\big)\sigma_1.
\end{split}
\end{align*}
The Wronskian matrix $\tilD(z)$ and its determinant $\tild(z):=\det\tilD(z)$ are given by
\begin{align*}
\begin{split}
\tilD=\pmat{\tils_{11}&\tils_{14}\\\tils_{23}&\tils_{24}},\qquad \tild=\tils_{11}\tils_{24}-\tils_{23}\tils_{14}.
\end{split}
\end{align*}
Arguing as in the proof of~\eqref{eq:G+2} and~\eqref{eq:G-2},
\begin{align*}
\begin{split}
\calG_0(r,s,z)=\frac{1}{\tild(z)}\big(\tils_{24}(z)\calG_{13}(r,s,z)-\tils_{14}(z)\calG_{23}(r,s,z)-\tils_{23}(z)\calG_{41}(r,s,z)+\tils_{13}(z)\calG_{24}(r,s,z)\big).
\end{split}
\end{align*}
For the $\tils_{ij}$ we have
\begin{align*}
\begin{split}
&\tils_{11}(z)=-\big(1+z^{-2}k_1^4(z)\big)W[\tilp_{+}(k_1(z)\cdot),\tilq_{+}(k_1(z)\cdot)],\\
&\tils_{14}(z)=-\big(1+z^{-2}k_1^2(z)k_2^2(z)\big)W[\tilp_{+}(k_1(z)\cdot),\tilq_{+}(k_2(z)\cdot)],\\
&\tils_{23}(z)=-\big(1+z^{-2}k_2^2(z)k_1^2(z)\big)W[\tilp_{+}(k_2(z)\cdot),\tilq_{+}(k_1(z)\cdot)],\\
&\tils_{24}(z)=-\big(1+z^{-2}k_2^4(z)\big)W[\tilp_{+}(k_2(z)\cdot),\tilq_{+}(k_2(z)\cdot)].
\end{split}
\end{align*}
Since $1+\frac{k_1^2(z)k_2^2(z)}{z^2}=0$, this shows that $\tils_{14}=\tils_{23}=0$ and $\tild=\tils_{13}\tils_{24}$. It follows that
\begin{align*}
\begin{split}
\calG_0(r,s,z)=\frac{1}{\tils_{13}(z)}\calG_{13}(r,s,z)+\frac{1}{\tils_{24}(z)}\calG_{24}(r,s,z).
\end{split}
\end{align*}
To compute $\tils_{13}$ and $\tils_{24}$ we use the small $\zeta$ asymptotics (see \cite[(9.1.10)--(9.1.11)]{AS})
\begin{align}\label{eq:tilhsmallzeta1}
\begin{split}
\tilp_{+}(\zeta)=c_1\zeta^{-\frac{1}{2}}+\calO(\zeta^{\frac{3}{2}}\log \zeta),\qquad \tilq_+(\zeta)=c_2\zeta^{\frac{3}{2}}+\calO(\zeta^{\frac{7}{2}}),
\end{split}
\end{align}
for some nonzero constants $c_1$ and  $c_2$. It then follows that
\begin{align*}
\begin{split}
W[\tilp_{+}(k_ir),\tilq_{+}(k_jr)]=2c_1c_2 k_i^{-\frac{1}{2}} k_j^{\frac{3}{2}}.
\end{split}
\end{align*}
For $(i,j)=(1,3)$ and $(i,j)=(2,4)$  we obtain
\begin{align*}
\begin{split}
|\tils_{13}(z)|\simeq |\tils_{24}(z)|\simeq \sqrt{|z|},
\end{split}
\end{align*}
completing the proof of the lemma.
\end{proof}

\section{Spectral projectors $P_I$}
\label{sec:PI}

Here we return to the question of how to define an operator $P_I$ to achieve the localization in Definition~\ref{def:PIevol}. For selfadjoint operators, these spectral projectors are routine,  but for $i\calL$ we need the construction of the resolvent kernels in order to justify this step.  Clearly, $P_I$ should be obtained by setting $t=0$ in~\eqref{eq:localevol}.  A delicate question is now the following, where $\delta_0$ be as in the statement of Theorem~\ref{thm:summary}: 

\begin{question}
    For intervals $I\subseteq[-\delta_0,\delta_0]$, does the limit
    \begin{align}
         P_I &:= \lim_{b\to0+} \frac{1}{2\pi i} \int_I    [ (i\calL-(\lambda +ib))^{-1}- (i\calL-(\lambda -ib))^{-1} ]\ud \lambda \label{eq:PIstrong}\\
    &= \frac{1}{2\pi i} \int_I    [ (i\calL-(\lambda +i0^{+}))^{-1}- (i\calL-(\lambda -i0^{+}))^{-1} ]\ud \lambda
    \end{align}
    exist in the strong $L^2((0,\infty))$ sense?
\end{question}

In particular, denoting the $b$-dependent operator in~\eqref{eq:PIstrong} by $P_I(b)$, this would then mean that $\sup_{0<b\ll 1} \|P_I(b)\|<\infty$ for the operator norms. While we strongly believe in an affirmative answer to this question, and that it should be accessible by the construction of the resolvents in the complex plane, which we obtain in this work, we do not attempt to answer it here. We note, however, that the $L^2$ bounds of the previous section imply that the integral in the line below~\eqref{eq:PIstrong} defines an $L^2$-bounded operator. But this does not settle the question.  Assuming that the question above has an affirmative answer, we now prove certain structural properties of $P_I$.

\begin{prop}
    \label{prop:PI}
    Let $I, I'\subset [-\delta_0,\delta_0]$ be  compact intervals, where $\delta_0>0$ is a small constant. Assuming the strong limit in~\eqref{eq:PIstrong} is valid, the bounded operators $P_I, P_{I'}$ satisfy $P_I^2=P_I$, and $\bfJ P_I^*=P_I\bfJ$ and $P_{I'}P_I=P_{I\cap I'}$. In particular, $\ker(P_I)$ and $\Ran(P_I)$ are skew-orthogonal relative to the symplectic form $\omega(\phi,\psi)=\langle \bfJ\phi,\psi\rangle$ (with the real inner product).  
\end{prop}
\begin{proof} 
Let $z=\lambda'+ib'$ and $\zeta=\lambda+ib$. Then  by the assumption
\[
P_I = \lim_{b\to0+} \frac{1}{2\pi i} \int_I [ (i\calL-\zeta)^{-1}- (i\calL-\bar{\zeta})^{-1}]  \ud \lambda
\]
in the strong $L^2$ sense and therefore 
\[
P_{I'}P_I = \lim_{b'\to0+} \lim_{b\to0+} \frac{-1}{4\pi^2 } \int_{I'} \int_I      [ (i\calL-z)^{-1}- (i\calL-\bar z)^{-1}][ (i\calL-\zeta)^{-1}- (i\calL-\bar \zeta)^{-1}]  \ud \lambda \ud \lambda'
\]
also in the strong sense. 
By the resolvent identity, with $R(z):=(i\calL-z)^{-1} $, 
\begin{align*}
    (R(z)-R(\bar z))(R(\zeta)-R(\bar\zeta)) &= \frac{R(\zeta)-R(z)}{\zeta-z}- \frac{R(\zeta)-R(\bar z)}{\zeta-\bar z} \\
    &\qquad - \frac{R(\bar \zeta)-R(z)}{\bar \zeta-z} + \frac{R(\bar\zeta)-R(\bar z)}{\bar\zeta-\bar z} \\
    &= R(\zeta) \frac{z-\bar z}{(\zeta-z)(\zeta-\bar z)} -R(z) \frac{\bar\zeta-\zeta}{(\zeta-z)(\bar\zeta- z)} \\
    &\qquad + R(\bar\zeta) \frac{\bar z- z}{(\bar\zeta-z)(\bar\zeta-\bar z)} - R(\bar z) \frac{\zeta-\bar \zeta}{(\zeta-\bar z)(\bar \zeta-\bar z)}.
\end{align*}
For  $\lambda'$ in the interior of~$I$ we have 
\begin{align*}
   \lim_{b'\to 0+} \lim_{b\to 0+} \int_I \frac{z-\bar z}{(\zeta-z)(\zeta-\bar z)}\ud \lambda &=  \lim_{b'\to 0+} \lim_{b\to 0+} \int_I \frac{2ib'}{(\lambda'-\lambda+ib)^2+(b')^2}\ud \lambda = 2i\pi  \chi_{I}(\lambda'),
\end{align*}
where $\chi_I$ is the indicator, while 
\begin{align*}
   \lim_{b'\to 0+} \lim_{b\to 0+} \int_I \frac{\zeta-\bar \zeta}{(\zeta-\bar z)(\bar \zeta-\bar z)}\ud \lambda &=   \lim_{b'\to 0+}\lim_{b\to 0+}   \int_I \frac{2ib}{(\lambda'-\lambda+ib')^2+b^2}\ud \lambda = 0.
\end{align*}
It then follows from the preceding that $P_{I'}P_I=P_{I\cap I'}$ which in particular means that $P_I^2=P_I$. The skew-adjoint relation holds due to $\calL^*=\bfJ\calL \bfJ$. 
\end{proof}

By the same argument, one can conclude that $e^{t\calL}P_I$ as a composition of bounded operators agrees with Definition~\ref{def:PIevol}.

\bibliographystyle{amsplain}
\bibliography{References}

\end{document}